\documentclass[11pt, a4paper]{article}
\usepackage{amssymb, amscd, amsmath,amsthm,amsfonts,enumitem}
\usepackage[T1]{fontenc}
\usepackage[utf8]{inputenc}
\usepackage{authblk}
\usepackage{epsfig}
\usepackage{graphicx}
\usepackage{float}
\usepackage{pict2e}
\usepackage{color}
\usepackage{amsfonts}
\usepackage{amsmath}
\usepackage{epsf}
\usepackage{enumitem} 
\usepackage{bbm}
\usepackage[numbers]{natbib}
\usepackage{subfigure,psfrag}
\theoremstyle{plain}
\newtheorem{theorem}{Theorem}[section]
\newtheorem{assumption}{Assumptions}[section]
\usepackage[markup=nocolor,final]{changes} 
\usepackage[markup=nocolor]{changes}

\newtheorem{lemma}{Lemma}[section]
\newtheorem{corollary}{Corollary}[section]

\newtheorem{claim}{Claim}[section]
\theoremstyle{definition}
\newtheorem{remark}{Remark}[section]
\newtheorem{example}{Example}[section]

\newcommand{\ind}{\mathbbm{1}}
\newcommand{\asarrow}{\overset{a.s.}{\rightarrow}}
\newcommand{\parrow}{\overset{\mathbb{P}}{\rightarrow}}

\newcommand{\darrow}{\overset{d}{\rightarrow}}

\begin{document}

\title{The duration of a supercritical $SIR$ epidemic on a configuration model}

\author
{Abid Ali Lashari}
\author
{Ana Serafimovi\'{c}}
\author
{Pieter Trapman}
\affil
{Stockholm University}

\maketitle

\begin{abstract}
We consider the spread of a supercritical stochastic
$SIR$ (Susceptible, Infectious, Recovered) epidemic on a configuration model random  graph.
We mainly focus on the final stages of a large outbreak and provide limit results for the duration of the entire epidemic, while we allow for non-exponential distributions of the infectious period  and for both finite and infinite variance of the asymptotic degree distribution in the graph.  

Our analysis relies on the analysis of some subcritical continuous time branching processes and on ideas from first-passage percolation.  

As an application we investigate the effect of vaccination with an all-or-nothing vaccine on the duration of the epidemic. We show that if vaccination fails to prevent the epidemic, it often -- but not always -- increases the duration of the epidemic. 


MSC Primary: 60K35, 92D30, 05C80; Secondary: 60J80

\end{abstract}

\section{Introduction}

Mathematical models have been widely used to study the spread of infectious diseases and to design control strategies for reducing the impact of those diseases \cite{diekmann2012mathematical}. In many models, a key assumption for the spread of epidemics is that the individuals are uniformly mixing, i.e.\ all pairs of individuals in the population contact each other at the same rate, independently of each other. 
In order to gain some realism, a (social) network structure may be introduced to the models where contacts are only possible between ``neighbors'', which are pairs of individuals that
share a connection in the network (see e.g.\ \cite{andersson1999epidemic,newman2002spread}).  
In this set-up,
each vertex represents an individual and an edge represents 
that two individuals have a relationship that makes it possible for the disease to transmit from one to the other. 

Much is already known for (variants of) epidemics on random graphs, e.g.\ the final size of the epidemic (the fraction of the population infected during the epidemic) and the probability of a large or major outbreak (to be defined below) \cite{britton2010stochastic,ball2009threshold}. In this paper we focus on the random duration of an epidemic on a configuration model graph.
The duration of an epidemic is especially relevant for animal diseases. When outbreaks of those diseases occur, trade bans are often imposed on import from affected counties. So, from an economics perspective, it might be more important to reduce the duration of an epidemic than to reduce the number of animals killed by it.

For uniformly mixing populations Barbour \cite{barbour1975duration} provides rigorous results on the duration of (Markov) $SIR$ (Susceptible, Infectious, Recovered; see Section \ref{sec:sir} for a definition) epidemics and Britton \cite{britton2010stochastic} also sketched some results about the duration of epidemic in a uniformly mixing population. 
A corollary of their results is that, if a major outbreak occurs in a population of size $n$, the time until the epidemic goes extinct divided by $\log n$ converges to an explicit constant as $n \to \infty$.

We consider $SIR$  epidemics  on configuration model graphs in the large population limit. Configuration model graphs are random graphs with specified vertex degrees (see Section \ref{sec:conf}, or for a detailed description see \cite{durrett2007random,van2016random}). In this graph each individual/vertex has his or her given degree (number of neighbors). The edges are
created in such a way that the graph is  uniform among all possible multigraphs with the given degree sequence. 

We only consider major outbreaks of the epidemic. Formally we say that an outbreak is major if more than $\log n$ individuals get infected, where $n$ is again the number of individuals in the
population. It can be shown that this is (as $n \to \infty$) equivalent to assuming that the number of ultimately infected individuals is of the same order as the population size. 
The beginning (until a small but non-negligible  fraction of the population is infected)  and the middle part (until a small but non-negligible  fraction of the ultimately infected individuals still has to be infected) of a major outbreak on a configuration model have been studied before (e.g.\ in \cite{barbour2013approximating,Decr12,volz2008sir,janson14}). Volz \cite{volz2008sir} studied a deterministic model for the spread of an $SIR$ epidemic through a network using a set of differential equations, keeping track of the probability that a vertex of given degree avoids infection as a function of time. Under some moment conditions his results were made rigorous by Decreusefond et al. \cite{Decr12}. Using a different mathematical approach  Barbour and Reinert \cite{barbour2013approximating} study (among other things) a stochastic model for the spread of $SIR$ epidemics on a configuration model with bounded degrees and minor conditions on the infectious period distribution. The approach of the paper is tailored for finding the distribution of the time a typical individual in the population gets infected, but is not directly suitable for finding the time of the last infected individual recovering. 
Janson et al.\ \cite{janson14} study the spread of Markov $SIR$ epidemics on quite general configuration models and their analysis heavily relies on the memoryless infectious period. In none of the papers mentioned above the time until the end of the epidemic is studied. 

The spread of epidemics on random graphs can also be studied using first-passage percolation \cite{Bham14,Bham17,Amin15}. In first passage percolation i.i.d.\ weights (lengths) are assigned to edges in the graph and questions regarding distances between typical vertices in the graph can be answered. In epidemiological terms the distance between a uniformly chosen vertex and the initially infected vertex may be interpreted  as the time of infection of that uniformly chosen vertex in an $SI$ epidemic (i.e.\ an $SIR$ epidemic with infinite infectious period).  
In this setting the question regarding the time until the last infection in the epidemic corresponds to the flooding time of the giant component of the random graph \cite{Amin15}. 

In the analysis of first passage percolation on random graphs in \cite{Bham14,Bham17} growing ``balls''  around vertices are explored and the time at which the balls touch provides precise results on the distance between the center vertices of those balls. These methods are well suited for obtaining the asymptotic distribution of the distance between two vertices, but are less fit for finding flooding times and diameters (however, see \cite{Amin15}). 

As written above, we focus on the duration of the entire epidemic, and in particular on the final stages of the epidemic. We use two definitions of the end of the epidemic: i)  the time at which there are no infectious individuals in the population anymore, which we call {\it strong extinction} and ii) the first time at which there are no more infectious individuals with susceptible neighbors in the population, which we call {\it weak extinction}.  We allow for quite general infectious period distributions (see Theorems \ref{mainthm} and \ref{mainthmsec} below), and do not have to restrict ourselves to infinite infectious periods as is the case in the first passage percolation literature. Furthermore, we pose milder conditions on the degree distribution of the configuration model than Barbour and Reinert \cite{barbour2013approximating}, who also allow for relatively general infectious period distributions. Our approach is to use the results of \cite{barbour2013approximating}, which are obtained through methods similar to those used in first passage percolation, to obtain the time until a typical vertex gets infected and then use subcritical branching processes to approximate the time between the infection of a typical vertex and the end of the epidemic.      
We show that the duration of the epidemic divided by $\log n$ converges in probability to a specified constant. We note that our result is weaker in nature than the results of \cite{barbour2013approximating,Bham14,Bham17}, where asymptotic distributions of infection times/distances of uniformly chosen vertices minus their typical distances are provided. However, as stated, we allow for more general distributions of the infectious period and degree distributions.

Finally, we shortly analyze the impact of vaccinating the entire population with an all-or-nothing vaccine. This vaccine either causes an individual to be completely immune or has no impact at all independently and with the same probability for different individuals. This vaccination strategy is asymptotically equivalent to vaccinating a uniform fraction of the population with a perfect vaccine, i.e.\ a vaccine which gives complete immunity.

\subsection{Outline of paper}

The paper is structured as follows. 
In Section \ref{sec:defnot} we formally define the model and provide the main theorems of the paper. 
In Section \ref{sec:vacc} we shortly discuss the impact of vaccination on the duration of the epidemic, using the results of Section \ref{sec:defnot} and some heuristics.
In Section \ref{sec:epid} we present some techniques for analyzing epidemics on graphs. Furthermore, we summarize results on continuous time branching processes that we need in the proofs of the main theorems.
In Section \ref{sec:heur} heuristics are given for the main theorems, while in Sections \ref{secinit} and \ref{secext} these theorems are proved rigorously. In these proofs the durations of the initial and final phase of the epidemic are analysed separately.

\section{Definitions, notation and main results}
\label{sec:defnot}
\subsection{Basic notation}
\label{sec:notation} 
The following basic notation and definitions are used throughout this paper (see also e.g.\ \cite[Section 1.2]{Jans11}). For  $f: \mathbb{R} \to \mathbb{R}$ and $g: \mathbb{R} \to \mathbb{R}_{\geq 0}$  and $x \to \infty$ we write, 
\begin{equation*}
\begin{array}{llll}
  f(x)& =&O(g(x)) & \textrm{if $\limsup |f(x)|/g(x)<\infty,$}\\
  f(x) &=& o(g(x))  & \textrm{if $\lim f(x)/g(x)=0$,}\\
  f(x) &=& \Theta(g(x)) &  \textrm{if $0<\liminf|f(x)|/g(x)\leq \limsup|f(x)|/g(x)<\infty.$}
\end{array}
\end{equation*}
Also for $f: \mathbb{R} \to \mathbb{R}$, we write $f(x-) =\lim_{y \nearrow x} f(y)$.

All random processes and random variables that we consider are defined on a rich enough probability space $(\Omega,\mathcal{F},\mathbb{P})$, which we do not further specify.
The population size is always denoted by $n$. In this paper, asymptotic results and limits are for $n \to \infty$, unless explicitly stated otherwise.
We say that an event occurs {\it with high probability} (w.h.p.) if the probability of the event converges to 1. Furthermore,  $\asarrow$ denotes almost sure convergence, $\parrow$ denotes convergence in probability, and $\darrow$ denotes convergence in distribution.

We denote the set of strictly positive integers by $\mathbb{N}$ and write $\mathbb{N}_0 = \mathbb{N} \cup \{0\}$. Furthermore, $\mathbb{N}_{\leq x} = [1,x] \cap \mathbb{N}$. The sets 
$\mathbb{N}_{\geq x}$, $\mathbb{N}_{< x}$ and $\mathbb{N}_{> x}$ are defined similarly.
Throughout, the cardinality of a set $\mathcal{X}$ is denoted by $ |\mathcal{X}|$.


\subsection{Construction of the random graph and assumptions on the degree distribution}\label{sec:conf}

The epidemic spreads on a random graph $G^{(n)}=(V^{(n)},E^{(n)})$.
The set $V^{(n)}$ consists of $n$ vertices that represent the individuals, and the edge set $E^{(n)}$ represent connections/relationships of individuals through which infection might transmit. For $v \in V^{(n)}$, the degree of vertex $v$ (i.e.\ the number of edges adjacent to vertex $v$) is denoted by $d_v$. We assume that $d_v \in \mathbb{N}$, since vertices of degree 0 will not be infected anyway.
$G^{(n)}$ is generated through a configuration model with given degree sequence $\{d_v\}_{v \in V^{(n)}}$.

The graph is constructed by assigning $d_v$ half-edges (edges with only one endpoint
assigned to a vertex) to the vertex $v$ for $v \in V^{(n)}$ and pairing those half-edges uniformly at random. By this construction every vertex has the right degree, although it is possible that there is more than one edge  between a pair of vertices (parallel edges) or that an edge connects a vertex to itself (a self-loop). In the graph, parallel edges are counted separately in the degree and a self-loop adds two to the degree of a vertex.

Define 
\begin{equation}
\label{ellell2}
\ell(n)= \sum_{v \in V^{(n)}} d_v \qquad \mbox{and} \qquad 
\ell_2(n)= \sum_{v \in V^{(n)}} (d_v)^2.
\end{equation}
Observe that $\ell(n)$ is even, since every edge in $E^{(n)}$ adds 2 to the total degree of the graph.
We make the following assumptions
\begin{assumption}\label{mainassump}
there exists an $\mathbb{N}$ valued  random variable $D$ such that,
\begin{itemize}
\item[(A1)] $n^{-1}\sum_{v \in V^{(n)}} \ind(d_v =k) \to p_k = \mathbb{P}(D=k) $,
\item[(A2)] $n^{-1}\ell(n) \to \mathbb{E}[D]< \infty$,
\item[(A3)] $n^{-1}\ell_2(n) \to \mathbb{E}[D^2] \leq \infty$,
\item[(A4)] If $\mathbb{P}(D \geq k)=0$ for $k \in \mathbb{N}$, then there exists $n_0 \in \mathbb{N}$, such that for all $n \in \mathbb{N}_{\geq n_0}$, it holds that $\sum_{v \in V^{(n)}} \ind(d_v \geq k) =0$. 
\end{itemize}
\end{assumption}

Assumption (A4) is introduced for technical purposes in the proofs. We expect that this condition is not needed for the results to be true. This assumption assures that $D$ provides in some sense enough information on the highest degree vertices of a large but finite graph. The Assumption obviously holds if $D$ has unbounded support or if the degrees of vertices in $V^{(n)}$ are i.i.d.\ and distributed as $D$.

The ``size biased'' random variable $\tilde{D}$ is defined through $$\mathbb{P}(\tilde{D}=k)=\tilde{p}_k= \frac{k p_k}{\mathbb{E}[D]}.$$ 

Let $D^{(n)}$ be a random variable with the same distribution as the degree of a vertex chosen uniformly at random from the graph. That is 
$$\mathbb{P}(D^{(n)}=k)= n^{-1}\sum_{v \in V^{(n)}} \ind(d_v =k) \qquad \mbox{ for $k \in \mathbb{N}$.}$$ By (A1) and (A2), $D^{(n)}\darrow D$ and  $\mathbb{E}[D^{(n)}]\darrow \mathbb{E}[D]$.

Let  $\tilde{D}^{(n)}$ be the size biased variant of  $D^{(n)}$, i.e.\ 
$$\mathbb{P}(\tilde{D}^{(n)}=k)=\ \frac{k \mathbb{P}(D^{(n)}=k)}{\mathbb{E}[D^{(n)}]} = \frac{k \sum_{v \in V^{(n)}} \ind(d_v =k)}{\ell(n)} \qquad \mbox{for $k \in \mathbb{N}$}.$$ Note that $\tilde{D}^{(n)}$ is distributed as the degree of a vertex adjacent to a uniformly chosen edge from the graph.
By (A1) and (A2), $\tilde{D}^{(n)}\darrow \tilde{D}$, while
by (A3),
\begin{equation}\label{Dtildeconv}
\mathbb{E}[\tilde{D}^{(n)}] = \frac{\ell_2(n)}{\ell(n)}\darrow \frac{\mathbb{E}[D^2]}{\mathbb{E}[D]}= \mathbb{E}[\tilde{D}] \in (0,\infty].
\end{equation}

For the epidemic process on the graph, we merge  parallel edges  and ignore self-loops. Because $\mathbb{E}(D)<\infty,$ this assumption has no impact on the asymptotic degree distribution \cite[Thm.~1.6]{van2017} although the number of self-loops and parallel edges diverges if $Var(D)=\infty$ \cite[p.~219]{van2016random}.

\subsection{The $SIR$ epidemic}\label{sec:sir}
We consider an $SIR$ (Susceptible, Infectious, Recovered) epidemic on $G^{(n)}$.
We say that a vertex is susceptible, infectious or recovered if the individual it represents is in this ``infection state''.
Neighbors in the population contact each other according to independent homogeneous Poisson processes with rate  $\beta$, and if the contact is between a susceptible and an infectious vertex, then the susceptible one becomes immediately infectious itself. Infectious vertices stay so for a random period distributed as the random variable $L$, which is $[0,\infty]$-valued. All infectious periods and Poisson processes are independent of each other.
A contact by an infectious vertex is called an \textit{infectious contact}, whether or not the ``contactee'' is susceptible.
Throughout we assume that at time 0, there is one infectious individual chosen uniformly at random from the population and all other individuals are susceptible. 
It is straightforward to extend the model to other initial conditions.  

The probability, $\psi$ say, that an infected vertex makes an infectious contact with a given neighbor (and infects it if that neighbor is still susceptible) is given by
\begin{equation}\label{psidef}
\psi=\int_{0}^{\infty}\beta \rm{e}^{-\beta t}\mathbb{P}(\mathit{L}>t)dt = 1-\int_{0}^{\infty} \rm{e}^{-\beta t}\mathit{L}(dt),
\end{equation}
where we used partial integration and  the shorthand $L(dt)=\mathbb{P}(L \in dt)$.

We denote the sets of susceptible, infectious and recovered individuals at time $t$ by $S^{(n)}(t)$, $I^{(n)}(t)$ and $R^{(n)}(t)$ respectively.
We say that the epidemic goes strongly extinct or ends before time $t$ if $|I^{(n)}(t)|=0$. 
Lastly, we let $X^{(n)}(t)$ be the number of pairs of neighbours of which one is susceptible and the other infectious. We say that the epidemic is weakly extinct at time $t$ if $X^{(n)}(t)=0$.

Throughout we  use continuous time branching processes \cite[Ch.\ 6]{jagers1975branching} to approximate the epidemic process. We rely on theory for those processes for which there exists a number $\alpha$ (called Malthusian parameter, or real-time growth rate) which satisfies
\begin{equation}\label{alphadef}
\int^{\infty}_{0}e^{-\alpha t}\mu(dt)=1,
\end{equation}
where $\mu(s)= \int_0^s \mu(dt)$ is the expected number of births of children of a particle up to time $s$, i.e.\ $\{\mu(s);s\geq 0\}$ defines the mean offspring measure of the branching process. 
Below we define and justify a branching process approximation for the early stages of a $SIR$ epidemic. The approximating branching process has mean offspring measure 
\begin{equation}
\label{mudtfirst}
\mu'(dt) = \mathbb{E}[\tilde{D}-1] \beta e^{-\beta t} \mathbb{P}(L >t) dt. 
\end{equation}
Following the terminology from epidemiology, we define the basic reproduction number $R_0$ as the expected total number of children of a particle in the branching process:
\begin{equation}
\label{R0def}
R_0=\mu'(\infty) = \int_0^{\infty} \mu'(dt) = \psi \mathbb{E}[\tilde{D}-1].
\end{equation}
Here we used \eqref{psidef} for the last identity.
If $R_0 > 1$  the epidemic is supercritical and $\alpha$ exists and is strictly positive. If on the other hand $R_0<1$, the process is subcritical and $\alpha$ might exist and if it does, $\alpha$ is strictly negative.  If $R_0=1$ the epidemic is critical and the corresponding $\alpha$ trivially equals 0.

In epidemic literature $R_{0}$ is arguably the most studied quantity (e.g. \cite{diekmann2012mathematical}). It is usually defined as the average number of secondary infections caused by a typical infected individual in the early stages of an epidemic in a further  susceptible population. This definition is consistent with \eqref{R0def}.

\subsection{The main results}
\label{sec:main}
In this subsection we state the main results of the paper. The proofs will be provided in Sections \ref{secinit} and \ref{secext}. We consider an $SIR$ epidemic on the configuration model graph $G^{(n)}= (V^{(n)},E^{(n)})$  with degrees satisfying Assumptions \ref{mainassump}. The infectious periods are distributed as $L$, and neighbors contact each other according to independent Poisson processes with intensity $\beta$.  Throughout we condition on a major outbreak, which we denote by $\mathcal{M}^{(n)}$ and define as an outbreak in which more than $\log n$ vertices get infected, i.e.\  
\begin{equation}
\label{majoroutb} 
\mathcal{M}^{(n)}= \{|S^{(n)}(0)\setminus S^{(n)}(\infty)|>\log n\}. 
\end{equation}
In this definition the $\log n$ term can be replaced by any increasing function which goes to infinity but is $o(n)$. It can be proved that $|S^{(n)}(0)\setminus S^{(n)}(\infty)| = \Theta(n)$ on $\mathcal{M}^{(n)}$ w.h.p.\ This can be shown in a similar way as the corresponding result in \cite[Thm.\ 3.5]{ball2014epidemics}.

Define the time until strong extinction of an epidemic in a population of  size $n$ by
\begin{equation}
\label{Tstardef}
T^{*}(n) = \inf\{t \geq 0; |I^{(n)}(t)|=0\}.
\end{equation}
We also consider the time of weak extinction $T^{\dagger}(n)$, i.e.\ the time after which no further infections are possible, because there are no more infected vertices with susceptible neighbors. That is,
\begin{equation}
\label{Tdaggerdef}
T^{\dagger}(n) = \inf\{t \geq 0; X^{(n)}(t)=0\}.
\end{equation}
For $\mathbb{E}[\tilde{D}-1]< \infty$ and $R_0=\psi \mathbb{E}(\tilde{D}-1)>1$, define $\alpha'$ by 
\begin{equation}
\label{beginmalt}
\alpha'= \{x \in \mathbb{R}; 1= g'(x)\},
\end{equation}
where for $x \in \mathbb{R}$ 
\begin{equation}
\label{gprimedef}
 g'(x) = \int_0^{\infty} e^{-x t} \mu'(dt)
\end{equation}
and
$\mu'(dt)= \mathbb{E}[\tilde{D}-1] \beta e^{-\beta t} \mathbb{P}(L >t) dt$ as defined in \eqref{mudtfirst}.
Because $g'(x)$ is continuous for $x>0$, $\alpha'$ is well defined. 
If $\mathbb{E}[\tilde{D}-1]= \infty$, we set $\alpha'= \infty$.

Define 
\begin{equation}
\label{endmalt1}
\alpha^* = \inf\{x \in \mathbb{R} ; g^* (x) <1\},
\end{equation}
where for $x \in \mathbb{R}$
\begin{equation}
\label{gstardef}
 g^*(x) = \int_0^{\infty} e^{-x t} \mu^*(dt),
\end{equation}
\begin{equation}
\label{endmalt2}
\mu^*(dt) = \mathbb{E}\left[(\tilde{D}-1)(1-\psi + \psi \tilde{q}^*)^{\tilde{D}-2}\right]  \beta e^{-\beta t} \mathbb{P}(L>t) dt,
\end{equation}
and
\begin{equation}
\label{endext}
\tilde{q}^* = \min \left\{s \geq 0; s= \mathbb{E}\left[(1-\psi + \psi s)^{\tilde{D}-1}\right]\right\}.
\end{equation} 
From Claim \ref{dalem} below it follows that $\mathbb{E}\left[(\tilde{D}-1)(1-\psi + \psi \tilde{q}^*)^{\tilde{D}-2}\right] <\infty$ and thus that $\mu^*(dt)$ is well defined.

For further use, define
\begin{equation}
\label{R0stardef}
R_0^* = g^*(0)= \int_0^{\infty} \mu^*(dt) = \psi \mathbb{E}\left[(\tilde{D}-1)(1-\psi + \psi \tilde{q}^*)^{\tilde{D}-2}\right].
\end{equation}

By standard theory on supercritical branching processes \cite{jagers1975branching}, we obtain $\tilde{q}^* \in (0,1)$, because $\tilde{q}^*$ is the extinction probability of a supercritical branching process of which the number of children of a particle is Mixed Binomially distributed, with ``number of trials'' distribution $\tilde{D}-1$ and ``success probability'' $\psi$, and therefore with offspring mean $R_0>1$ \cite{ball2009threshold}. By Lemma \ref{subend} below the branching process defined through $\mu^*(dt)$ is subcritical.
If $\int_0^{\infty} e^{-\alpha^* t} \mu^*(dt) =1$ then $\alpha^*$ is a Malthusian parameter.
A sufficient (but far from necessary) condition for $\alpha^*$ to be a Malthusian parameter is that both 
$\mathbb{P}(L>t_0)=0$ for some $t_0>0$ and $\mathbb{E}[\tilde{D}-1] < \infty$ hold.

Before stating the main theorem, we provide the following lemma, the proof of which is provided in Section \ref{secext}.
\begin{lemma}\label{subend}
If $R_0>1$ then $R_0^* <1$ and  $\alpha^* <0$. 
\end{lemma}

The first main theorem then reads. 
\begin{theorem}\label{mainthm}
Assume that Assumptions \ref{mainassump} hold and $R_0>1$.
Then we have for all $\epsilon >0$ that  $$\mathbb{P}\left(\left|\frac{T^{*}(n)}{\log n} - \left(\frac{1}{\alpha'}+\frac{1}{|\alpha^{*}|}\right)\right|>\epsilon \mid \mathcal{M}^{(n)}\right) \to 0,$$ 
if and only if
\begin{equation}
 \label{THMass}
\int_0^{\infty} e^{(|\alpha^*|-\eta) t} L(dt) < \infty \qquad \mbox{for all $\eta \in (0,|\alpha^*|)$.} 
\end{equation}
\end{theorem}

Using this main theorem, we obtain in a straightforward fashion a result for the time until weak extinction as well (see the proof in Section \ref{Secproofmainthmsec}) 
\begin{theorem}\label{mainthmsec}
Assume that Assumptions \ref{mainassump} hold and $R_0>1$.
Then we have for all $\epsilon >0$ that  
$$\mathbb{P}\left(\left|\frac{T^{\dagger}(n)}{\log n} - \left(\frac{1}{\alpha'}+\frac{1}{|\alpha^{*}|}\right)\right|>\epsilon \mid \mathcal{M}^{(n)}\right) \to 0.$$ 
\end{theorem}

\begin{remark}
In Theorems \ref{mainthm} and \ref{mainthmsec} the summand  $\frac{1}{\alpha'}$ is related to the duration of the early stage, i.e.\ the exponentially growing phase, of the epidemic, while the summand  $\frac{1}{|\alpha^*|}$ is related to the duration of the final phase, i.e.\ the exponentially declining phase, of the epidemic. 

\end{remark}

\begin{remark}
\label{REMass}
Observe that for $x>0$,
\begin{equation*}
\begin{split}
\int_0^{\infty} e^{x t} L(dt)&  =  \int_0^{\infty} (e^{x t}-1) L(dt) + \int_0^{\infty}  L(dt)\\  
\ & =  x \int_0^{\infty} \left(\int_0^t e^{x s} ds\right) L(dt) + 1 \\
\ & =   x  \int_0^{\infty} \left(\int_s^{\infty} L(dt)\right) e^{x s} ds + 1 \\
\ & =   x  \int_0^{\infty} \mathbb{P}(L >s) e^{x s} ds + 1.
\end{split}
\end{equation*}
So, condition \eqref{THMass} is equivalent to 
\begin{equation}
\label{equivalence}
\int_0^{\infty} e^{(|\alpha^*|-\eta) t} \mathbb{P}(L>t) dt < \infty \qquad \mbox{for all  $\eta \in (0,|\alpha^*|)$}.
\end{equation}
This condition  guarantees that w.h.p.\ none of the individuals infected during the epidemic will stay infectious for a time longer than $\log[n]/|\alpha^*|$. Condition \eqref{equivalence} gives that Theorem \ref{mainthm} provides the (scaled) duration of the epidemic if $\mathbb{P}(L>t)$ decays faster than exponential, but  not if $\mathbb{P}(L>t)$ decays slower than exponential or exponentially with too low a rate. 
\end{remark}

\begin{remark}
\label{Remadag}
In order to understand the definition of $\alpha^*$ in \eqref{endmalt1} it is good to study
$g^*(x)$. 
This function is strictly decreasing, and we may define
\begin{equation}
\label{alphadagger}
\alpha^{\dagger} = \inf\{x \in \mathbb{R}; g^*(x)<\infty\}.
\end{equation}
By definition $\alpha^{\dagger} \in [-\infty, \alpha^*]$.
Recalling the definition of $\mu^*(\cdot)$, we see that for all $x > -\beta$
\begin{multline*}
g^*(x) = \mathbb{E}\left[(\tilde{D}-1)(1-\psi + \psi \tilde{q}^*)^{\tilde{D}-2}\right]\int_0^{\infty} \beta e^{-(x+\beta) t} \mathbb{P}(L>t) dt\\ \leq \mathbb{E}\left[(\tilde{D}-1)(1-\psi + \psi \tilde{q}^*)^{\tilde{D}-2}\right]\int_0^{\infty} \beta e^{-(x+\beta) t} dt = \frac{R_0^*}{\psi} \frac{\beta}{x+\beta} ,
\end{multline*}
which is finite by $R_0^*<1$. So, we obtain $\alpha^{\dagger}\leq -\beta <0$. 

It is straightforward to see that $g^*(x)$ is continuous for $x > \alpha^{\dagger}$
and that if $\alpha^{\dagger}>-\infty$ then $\displaystyle\lim_{x \searrow \alpha^{\dagger}} g^*(x) = g^*(\alpha^{\dagger})$,
where $g^*(\alpha^{\dagger})$ may be infinite. Furthermore, $g^*(0) =R_0^* <1$. Together this has the following implications.
\begin{itemize}
 \item If $\alpha^{\dagger} = -\infty$ then $\displaystyle\lim_{x \to -\infty} g^*(x) = \infty$.
This implies that $\alpha^* \in(-\infty,0)$, while $g^*(\alpha^*) =1$ and $g^*(x) \in (1,\infty)$ for $x \in (-\infty,\alpha^*)$.
\item If $\alpha^{\dagger}>-\infty$ and $g^*(\alpha^{\dagger})> 1$, then by the same arguments as for  $\alpha^{\dagger} = -\infty$  we obtain $\alpha^* \in(\alpha^{\dagger},0)$, while also $g^*(\alpha^*)=1$ and $g^*(x) \in (1,\infty)$ for $x \in (\alpha^{\dagger},\alpha^*)$.
\item If $\alpha^{\dagger}>-\infty$ and $g^*(\alpha^{\dagger}) <1$, then there is no solution of $g^*(x)=1$ and the branching process with mean offspring measure $\mu^*(\cdot)$ does not have a Malthusian parameter, although $\alpha^*$ is well defined and equal to $\alpha^{\dagger}$. Furthermore, $g^*(x) = \infty$ for $x<\alpha^*$.
\item If $\alpha^{\dagger}>-\infty$ and $g^*(\alpha^{\dagger}) = 1$, then clearly $\alpha^* = \alpha^{\dagger}$, but $g^*(x) = \infty$ for $x<\alpha^*$.
\end{itemize}
\end{remark}

\begin{remark}
Intuition from first passage percolation (e.g.\ \cite{Bham14,Bham17}) and research on the epidemic curve \cite{janson14,barbour2013approximating} suggests that (possibly with some extra conditions on the distributions of the infectious period and degrees) $$T^{*}(n)- \left(\frac{1}{\alpha'}+\frac{1}{|\alpha^{*}|} \right) \log n$$ might converge in distribution to an a.s.\ finite random variable. We did not try to prove this or identify which extra conditions would be necessary for such a proof.  
\end{remark}

In order to prove  Theorem \ref{mainthm} we use some lemmas.
Let 
\begin{equation}
\label{qstar}
q^* = \mathbb{E}[(1-\psi + \psi \tilde{q}^*)^{D}],
\end{equation}
where $\tilde{q}^*$ is defined through \eqref{endext}. Copying the steps of the corresponding result for random intersection graphs as provided in \cite[Thm.~3.4]{ball2014epidemics}, we  obtain 
\begin{lemma}\label{lemint1}
Assume that Assumptions \ref{mainassump} hold.
Then $$\mathbb{P}(|n^{-1} |S^{(n)} (\infty)| - q^*|>\epsilon\mid \mathcal{M}^{(n)}) \to 0. \qquad \mbox{for all $\epsilon>0$.}$$
\end{lemma}

In order to formulate the main lemmas, define for $\gamma \in (0,1-q^*)$ 
\begin{equation}\label{Tprime}
T'_{\gamma}(n) = \inf\{t>0; n^{-1} |S^{(n)} (t)| < 1-\gamma\}.
\end{equation}
Theorem \ref{mainthm} now follows trivially from  the following lemmas, where the first is about the duration of the initial phase of the epidemic and the second about the duration of the final phase.
\begin{lemma}\label{lemint2}
Assume that (A1)-(A3) of Assumptions \ref{mainassump} hold.
Then
$$\mathbb{P}\left(\left|\frac{T'_{\gamma}(n)}{\log n} - \frac{1}{\alpha'}\right|>\epsilon\mid \mathcal{M}^{(n)}\right) \to 0 \qquad \mbox{for all $\epsilon>0$  and all $\gamma \in (0,1-q^*)$.}$$
\end{lemma}
\begin{lemma}\label{lemint3}
Assume that Assumptions \ref{mainassump} hold.
If and only if \eqref{THMass} holds, we have that there exists $\gamma \in (0,1-q^*)$, such that
$$\mathbb{P}\left(\left|\frac{T^{*}(n) - T'_{\gamma}(n)}{\log n}  - \frac{1}{|\alpha^*|}\right|>\epsilon\mid \mathcal{M}^{(n)} \right) \to 0 \qquad \mbox{for all $\epsilon>0$.}$$
\end{lemma}
Note that Lemma \ref{lemint2} implies that for all $\gamma, \gamma' \in (0,1-q^*)$ and all $\epsilon >0$ we have  $\mathbb{P}\left(\frac{|T'_{\gamma}(n)-T'_{\gamma'}(n)|}{\log n}>\epsilon| \mathcal{M}^{(n)}\right)\to 0$ and thus that Lemma \ref{lemint3} actually holds for all $\gamma \in (0,1-q^*)$.

\section{Vaccination}
\label{sec:vacc}
In this section we shortly discuss the effect of vaccination on the duration of an epidemic.  We give heuristics on the effect of vaccinating everybody in the population with an all-or-nothing vaccine in uniformly mixing populations and on configuration model graphs. We assume that the vaccination takes place before the outbreak starts. We only consider the case where the vaccination is not enough to make the epidemic process subcritical, i.e.\ the effective $R_0$ stays strictly larger than 1. 

Let $c \in (0,1]$. 
Recall that with an all-or-nothing vaccine, a vaccinated individual will not be affected by the vaccine (say with probability $c$) or will be immune to the infection (with probability $1-c$), independently of the effect of the vaccine on other individuals. In what follows, we decorate quantities associated with the epidemic in such a vaccinated population with a subscript $c$. 

The all-or-nothing vaccine effectively changes the (susceptible) population size to  $n_c$, a $Bin(n,c)$ distributed random variable and the limiting degree distribution of the initially susceptible vertices (say $D_c$) to a mixed binomial random variable with ``number of trials distribution'' $D$ and ``success probability'' $c$.
Below we use the trivial observation that $\frac{\log(n_c)}{\log n} \parrow 1$. Furthermore,
Straightforward computations give that if $\tilde{D}_c$ is the size biased version of $D_c$, then $\tilde{D}_c-1$ is mixed binomial with  ``number of trials distribution'' $\tilde{D}-1$ and ``success probability'' $c$. Similarly, for $x \in (0,1)$ we have
$$\mathbb{E}\left[(\tilde{D}_c-1)\left((1- x)\right)^{\tilde{D}_c-2}\right]
= c \mathbb{E}\left[(\tilde{D}-1)\left((1-c x)\right)^{\tilde{D}-2}\right].
$$

If $c \psi \mathbb{E}[\tilde{D}-1]>1$, then there is a positive probability of a large outbreak even after vaccination. Let $\mathcal{M}^{(n)}_c$ be the event of a large outbreak in a vaccinated population. 
Using this mixed binomial distribution in equations \eqref{mudtfirst},  \eqref{beginmalt}, \eqref{endmalt1} \eqref{endmalt2} and \eqref{endext} and assuming that $\alpha^*_c$ as defined below exists and satisfies \eqref{THMass}, we obtain that $T_c^*(n)$, the time until the end of the epidemic after vaccination satisfies for all $\epsilon>0$
$$\mathbb{P}\left(\left|\frac{T_c^*(n)}{\log n}-\left(\frac{1}{\alpha'_c}+ \frac{1}{|\alpha_c^*|}\right)\right|>\epsilon\mid\mathcal{M}_c^{(n)}\right)\to 0.$$ 
where $\alpha'_c$ and $\alpha^*_c$ satisfy 
\begin{equation}\label{alphacs}
1 = \int_0^{\infty} e^{-\alpha'_c t} \mu'_c(dt) \qquad \mbox{and} \qquad 
1 = \int_0^{\infty} e^{-\alpha^*_c t} \mu^*_c(dt),
\end{equation}
with
\begin{equation}
\label{vacmu'}
 \mu_c'(dt)= c \mathbb{E}[\tilde{D}-1] \beta e^{-\beta t} \mathbb{P}(L>t) dt,
\end{equation}
\begin{equation}
\label{vacmustar} 
\mu_c^*(dt)= c \mathbb{E}[(\tilde{D}-1)(1-c\psi + c \psi \tilde{q}^*_c)^{\tilde{D}-2}] \beta e^{-\beta t} \mathbb{P}(L>t) dt
\end{equation}
and
\begin{equation}
 \begin{split}
\label{vacq}
\tilde{q}^*_c & =\min
\left\{s \geq 0; s= \mathbb{E}
\left[
\left(1-c + c(1-\psi + \psi s)
\right)^{\tilde{D}-1}
\right]
\right\}\\ 
\ & =\min\left\{s \geq 0; s= \mathbb{E}\left[\left((1-c \psi + c \psi s)\right)^{\tilde{D}-1}\right]\right\}.
\end{split}\end{equation}

We now discuss heuristically the effect of vaccination in a uniformly mixing population (the second example) and in two different configuration model random graphs (the first and third example). In one of those graphs $(\alpha'_c)^{-1}+ |\alpha_c^*|^{-1}$ decreases as $c$ increases (i.e.\ the duration of a large outbreak increases as the vaccine becomes more effective) and an example in which $(\alpha'_c)^{-1}+ |\alpha_c^*|^{-1}$ is strictly less than $(\alpha'_1)^{-1}+ |\alpha_1^*|^{-1}$ for some $c\in (0,1)$ (i.e.\ the duration of a large outbreak is smaller in a vaccinated population than in an unvaccinated population).

\begin{example}
\label{example2}
Let $D$ be Poisson distributed with expectation $\lambda$. Furthermore,
let $L$ have bounded support, so that \eqref{THMass} is satisfied. This condition can easily be relaxed.
Recall from \eqref{psidef} that $\psi = \int_0^{\infty} \beta e^{-\beta t} \mathbb{P}(L>t)dt$. 
It is straightforward to check that 
$\mathbb{E}[\tilde{D}-1]= \lambda$, $\tilde{q}_c^*$
satisfies $\tilde{q}_c^*= e^{-\lambda c \psi (1-\tilde{q}_c^*)}$ and 
$$\mathbb{E}[(\tilde{D}-1)(1-c \psi (1-\tilde{q}_c^*))^{\tilde{D}-2})]= \lambda e^{-\lambda c \psi (1-\tilde{q}_c^*)} = \lambda \tilde{q}^*_c.$$
So, filling this in in \eqref{vacmu'} and \eqref{vacmustar}
we obtain 
\begin{equation}
 \label{Poissonmus}
 \mu_c'(dt) = c \lambda \beta e^{-\beta t} \mathbb{P}(L>t) dt \quad \mbox{and} \quad
\mu_c^*(dt) = c \lambda \tilde{q}^*_c \beta e^{-\beta t} \mathbb{P}(L>t) dt.
\end{equation}
We assume that the disease is supercritical even after vaccination. So, 
\begin{equation}
\label{upperex}
\int_0^{\infty} \mu_c'(dt) =  c \lambda \psi>1
\end{equation}
 and by Lemma \ref{subend}
 \begin{equation}
\label{lowerex}
\int_0^{\infty} \mu_c^*(dt) =  c \lambda \tilde{q}^*_c\psi <1.
\end{equation}

By the definition in \eqref{alphacs}, $\alpha_c'$ and $|\alpha_c^*|$ are given through
\begin{equation}
\label{Poisalp'}
1= \int_0^{\infty} e^{-\alpha'_c t} c \lambda \beta e^{-\beta t} \mathbb{P}(L>T) dt
\end{equation}
and 
\begin{equation}
\label{Poisalp*}
1= \int_0^{\infty} e^{|\alpha^*_c| t} c \tilde{q}^*_c \lambda \beta e^{-\beta t} \mathbb{P}(L>T) dt
\end{equation}
and it follows that $\alpha'_c$ is increasing in $c$ and $|\alpha^*_c|$ is decreasing in $c \tilde{q}^*_c$.

From $\tilde{q}_c^*= e^{-\lambda c \psi (1-\tilde{q}_c^*)}$, we deduce that for $\lambda c \psi>1$
 \begin{equation*}
\frac{d\tilde{q}_c^*}{dc}=-\tilde{q}_c^* \lambda \psi \left(1-\tilde{q}_c^*-c\frac{d\tilde{q}_c^*}{dc}\right)
\qquad \Rightarrow \qquad
\frac{d\tilde{q}_c^*}{dc}=\frac{\lambda \psi \tilde{q}_c^*(1-\tilde{q}_c^*)}{c \lambda \psi \tilde{q}_c^*-1}.
\end{equation*}
and
\begin{equation}
\label{cqder}
\frac{d(c \tilde{q}_c^*)}{dc} 
= \tilde{q}_c^* + \frac{c \lambda \psi \tilde{q}_c^*(1-\tilde{q}_c^*)}{c \lambda \psi \tilde{q}_c^*-1}
=  \tilde{q}_c^*  \frac{c \lambda \psi -1}{c \lambda \psi \tilde{q}_c^*-1},
\end{equation}
which is strictly negative by \eqref{upperex} and \eqref{lowerex}.
So, both $\alpha'_c$ and $|\alpha^*_c|$ are increasing in $c$, which implies that 
$(\alpha'_c)^{-1} + |\alpha^*_c|^{-1}$ and thus the limiting duration of a large outbreak is decreasing in $c$ (and increasing in the efficiency of the vaccine).
\end{example}

\begin{example}
In a uniformly mixing population, all pairs of individuals contact each other independently of each other at rate $\beta$. In order for the model to be interesting and the expected number of contacts per individual to stay constant if $n \to \infty$, we assume $\beta= \beta'/n$. The uniformly mixing population does not satisfy the conditions of our paper, but we can take several approaches to still analyze the uniformly mixing population. 
One is to deduce the branching process approximations used in this paper also for the uniformly mixing population and use that in Theorem \ref{mainthm} the quantities $\alpha'_c$ and $\alpha^*_c$ are the Malthusian parameters for those branching processes.

We however, use the previous example of analyzing an epidemic on a configuration model with a Poisson degree distribution, where $\beta' = \beta \lambda$, and where the expected degree $\lambda$ goes to infinity. It is easily checked that the epidemic generated graph (see e.g.\ \cite{ball2014epidemics}) of the epidemic on the configuration model converges locally to that of the uniformly mixing population.

Note that for all $t >0$ we have that as $\lambda \to \infty$, then $e^{-(\beta'/ \lambda)t} \to 1$ and 
\begin{equation}
\label{lampsi}
\lambda \psi = \lambda \int_0^{\infty} \frac{\beta'}{\lambda} e^{-(\beta'/\lambda) t} \mathbb{P}(L>t) dt \to \beta' \mathbb{E}[L].
\end{equation}
We further observe that for $x>1$, the solution of  $s = e^{-x(1-s)}$ is given by $s=-x^{-1} W(-xe^{-x})$, where $W(\cdot)$ is the principal branch of the Lambert $W$ function, which is a continuous function of $x$ \cite{Corl96}.
Together with \eqref{lampsi} this implies that, 
$$\tilde{q}_c^* \to \min\{s \geq 0; s= \mathbb{E}[e^{-c \beta' \mathbb{E}[L](1-s)}]\} \qquad \mbox{as $\lambda \to \infty$.}$$ 
Filling in $\beta = \beta'/\lambda$ in \eqref{Poisalp'} and \eqref{Poisalp*} and taking the limit $\lambda \to \infty$
gives
 that $\alpha_c'$ and $\alpha^*_c$  satisfy $$1= \int_0^{\infty} e^{-\alpha'_c t} c \beta' \mathbb{P}(L>t) dt \qquad \mbox{and } \qquad  1= \int_0^{\infty} e^{-\alpha_c^* t} c \beta' \tilde{q}_c^* \mathbb{P}(L>t) dt.$$ 
The inequalities \eqref{upperex} and \eqref{lowerex} should still hold.
It is again immediate that  $\alpha'_c$ is increasing in $c$ and
$|\alpha^*_c|$ is decreasing in $c \tilde{q}_c^*$.
Filling in $\lambda \psi = \beta'\mathbb{E}[L]$ in \eqref{cqder} we obtain
\begin{equation}\label{eq:58}
\frac{d(c\tilde{q}_c^*)}{dc}=\frac{c\beta' \mathbb{E}[L]-1}{c\beta' \mathbb{E}[L]\tilde{q}_c^*-1}\tilde{q}_c^*,
\end{equation}
which is strictly negative by  \eqref{upperex} and \eqref{lowerex}.

As in Example \ref{example2} this implies that increasing the efficiency of the vaccination, without making the epidemic subcritical increases the asymptotic duration of the epidemic.
\end{example}

\begin{example}
\label{exampel3}
For this example we use the following intuition.
Unvaccinated vertices of very high degree are very likely te be infected during the early stages of an epidemic, even if a fraction of their neighbors are vaccinated. Therefore, those vertices will hardly play a role in the duration of the final phase of the epidemic. Vertices who have initially 1 unvaccinated neighbor cannot be infected  and pass the disease on to other individuals, because the unvaccinated neighbor must be their infector. So, vertices with one unvaccinated neighbor that are still susceptible after the intermediate phase of the epidemic will shorten the final stage of the epidemic. 
If the infectious period is exponentially distributed, then vertices with two unvaccinated neighbors in the final stages of the epidemic do not lengthen the duration of the epidemic if they get infected, because the neighbor who is not the infector might be infected before, in which case the number of infectious, susceptible pairs decreases by the infection, or that neighbor is still susceptible in which case the number of infectious, susceptible pairs stays the same. 
So, as an example we consider a population in which vertices may have very large degree or have degree 1 or 2, and such that after vaccination of a small proportion of the population the ``effective degree distribution'' has more mass on 1. 

Using this intuition we consider the model with $$\mathbb{P}(D=1) = \mathbb{P}(D=2) = 100/201 \qquad \mbox{and} \qquad \mathbb{P}(D=100)=1/201.$$ This implies $$\mathbb{P}(\tilde{D}=1) =\mathbb{P}(\tilde{D}=100) = 1/4 \qquad \mbox{and} \qquad \mathbb{P}(\tilde{D}=2)=1/2.$$ 
Furthermore, let $t_0 >>0$ and assume that $\beta=99/100$ and $\mathbb{P}(L>t) = e^{-\mu t}$ for $t < t_0$ and $\mathbb{P}(L>t_0) =0$ with $\mu =1/100$, that is, $L$ is exponentially distributed with a cut-off at $t_0$. This cut-off is needed for  $\int_0^{\infty} t e^{-\alpha^* t} L(dt)$ to be finite. The above parameters make that in the limit $t_0 \to \infty$, we obtain $\psi = 99/100$. Without vaccination $\frac{1}{\alpha'} + \frac{1}{|\alpha^*|} = 2.04$, while with $1\%$ of the population vaccinated, i.e.\ with $c=0.99$ we obtain $\frac{1}{\alpha'_c} + \frac{1}{|\alpha^*_c|} = 2.021$. That is, vaccinating $1 \%$ of the population does not necessarily prevent the large outbreak and if a large outbreaks occurs it ends faster. 
\end{example}

\section{The epidemic on the graph}
\label{sec:epid}

\subsection{Construction of the graph together with the epidemic}
\label{sec:construct}

For the proof of the main theorems we rely on the following explicit step-by-step simultaneous construction of
the graph $G^{(n)}$ and the epidemic process $\{S^{(n)}(t),I^{(n)}(t),R^{(n)}(t); t \geq 0\}$, or more precisely, on the simultaneous construction of the cluster of vertices of $G^{(n)}$ which are ultimately recovered and the epidemic process. 
In this construction we see contacts as asymmetric: the times $v$ contacts $v'$ are not necessarily the same as when $v'$ contacts $v$, but contacts in both directions occur according to independent Poisson Processes with intensity $\beta$. Only when an infectious vertex contacts a susceptible neighbor (and not when a susceptible vertex contacts an infectious neighbor) the susceptible becomes infectious. Since in both the directed and undirected interpretation of contacts, contacts from an infected to a susceptible neighbor occur at intensity $\beta$, the spread of the epidemic is unaltered.  

Label the vertices in $V^{(n)}$ by $1, 2, \cdots, n$, such that
$$\mathbf{d}^{(n)}= d^{(n)}_1, \cdots, d^{(n)}_n = d_1,  \cdots, d_n$$ is a \textit{non-decreasing} degree sequence satisfying Assumptions \ref{mainassump}. 
Let  
$$\mathbf{s}^{(n)}= \{(1,1), (1,2),\cdots, (1,d_1), (2,1), \cdots, (2,d_2), \cdots, (n,1), \cdots, (n,d_n)\}$$ be the set containing $\ell(n)$ elements, corresponding to the half-edges used in the construction of $G^{(n)}$ and let $$\mathbf{x}^{(n)}=(x_1,x'_1),(x_2,x'_2), \cdots$$ be an infinite sequence of (2 dimensional) elements of $\mathbf{s}^{(n)}$, where the elements are chosen independently with replacement and uniformly at random.
Further define i.i.d.\ random variables $$\{\tau_{v,j}; (v,j) \in \mathbf{s}^{(n)}\}$$ which are exponentially distributed with expectation $1/\beta$. We may interpret $\tau_{v,j}$ as the first time after infection (if ever) $v$ makes a contact along the half-edge $(v,j)$.
Then define i.i.d.\ random variables $\{L_v; v \in V^{(n)}\}$ all distributed as $L$.
If $v$ becomes infectious during the epidemic, we interpret  $L_v$ as the infectious period of vertex $v$. Otherwise, $L_v$ has no epidemiological interpretation. 
Furthermore, let $x_0$ be the initially infected vertex, which is chosen uniformly at random from the population.
All random variables defined in this paragraph are independent of each other.

For reasons that will become clear in the proof of Theorem \ref{mainthmsec}, we define 
\begin{equation}\label{Lprime}
L'_v = \min(L_v, \max_{1 \leq j \leq d_v} \tau_{v,j}), \qquad \mbox{for $v\in V^{(n)}$.}
\end{equation}
So $L'_v$ time units after infection, $v$ is either recovered or has made contacts to all of its neighbors. This implies that $L'_v$ time units after infection $v$ is no longer the infectious vertex in an infectious-susceptible pair, because $v$ has either recovered or has made contacts to all of its neighbors (of which some might have been infected before). In particular, if we say that for $v \in V^{(n)}$, vertex $v$ recovers $L'_v$ instead of $L_v$ time units after $v$ got infected, the spread of the epidemic is unaltered.

We define the following process of partitions of the set of half-edges and vertices, in which the half-edges are paired at the moment a contact involving an infectious vertex is made.
\begin{multline}
 \{\mathcal{X}^{(n)}(t);t \geq 0\} 
\\
 =  \{(S^{(n)}(t),I^{(n)}(t),R^{(n)}(t),\mathcal{E}^{(n)}_S(t),\mathcal{E}^{(n)}_I(t),\mathcal{E}^{(n)}_R(t),\mathcal{E}^{(n)}_P(t);t \geq 0\}. 
\end{multline}
In this process $S^{(n)}(t)$, $I^{(n)}(t)$ and $R^{(n)}(t)$ are respectively the sets of susceptible, infectious and recovered vertices at time $t$. The set $\mathcal{E}^{(n)}_S(t)$ consists of the unpaired half-edges belonging to  vertices in $S^{(n)}(t)$, $\mathcal{E}^{(n)}_I(t)$ is the set of unpaired half-edges belonging to  vertices in $I^{(n)}(t)$, $\mathcal{E}^{(n)}_R(t)$ is the set of unpaired half-edges belonging to vertices in $R^{(n)}(t)$ and $\mathcal{E}^{(n)}_P(t)$ is the set of paired half-edges.
Let 
$$\sigma(v) = \inf\{t\geq 0; v \in I^{(n)}(t)\}, \qquad \mbox{for $v\in V^{(n)}$}$$ 
be the time that $v$ gets infected, which corresponds to the time at which the first half-edge belonging to vertex $v$ is added to $\{\mathcal{E}^{(n)}_P(t);t \geq 0\}$.
Throughout the process the sequence $\mathbf{x}^{(n)}$ is also explored element by element and $\mathbf{x}^{(n)}(t)$ is the set of elements of $\mathbf{x}^{(n)}$ explored before or at time $t$.

The construction of $\{\mathcal{X}^{(n)}(t);t \geq 0\}$ is as follows. 
\begin{itemize}
\item Start of construction: Choose the initial infected vertex $x_0$ uniformly at random. So, $I^{(n)}(0) = \{x_0\}$. Set $S^{(n)}(0) = V^{(n)}\setminus x_0$
and $\sigma(x_0) =0$. Note that $x_0$ has degree $d_{x_0}$ in $G^{(n)}$.
Furthermore, the set $\mathcal{E}^{(n)}_I(0)= \{(v,j)\in \mathbf{s}^{(n)};v=x_0\}$ consists of all half-edges attached to $x_0$, while all other half-edges are in $\mathcal{E}^{(n)}_S(0)= \{(v,j)\in \mathbf{s}^{(n)};v \neq x_0\}$. 
None of the elements of $\mathbf{x}^{(n)}$ are explored yet at time $0$, i.e.\ $\mathbf{x}^{(n)}(0)= \emptyset$. 

\item Assume that at time $t$, $\mathbf{x}^{(n)}$ is explored up to and including $(x_k,x'_k)$, i.e.\ $\mathbf{x}^{(n)}(t)=\{(x_1,x'_1),\cdots, (x_k,x'_k)\}$. Define $$
t_+(t)= \min (\{\sigma(v)+L_v; v \in I^{(n)}(t)\} \cup \{\sigma(v)+\tau_{v,j}; (v,j) \in \mathcal{E}^{(n)}_I(t)\}),$$
which can be interpreted as the first time after time $t$ that a change in $\{\mathcal{X}^{(n)}(t);t \geq 0\}$ occurs, by either a recovery of an infected vertex or a pairing of two half-edges and possibly the infection of a vertex. Because the distribution of the ``$\tau$ random variables'' do not have any atoms, the infection times of (infected) vertices are almost surely different and at $t_+(t)$ occurs almost surely only one event. 
In the interval $[t,t_+(t))$ the process $\mathcal{X}^{(n)}(t)$ is constant, while if 
$t_+(t)=\sigma(u) + L_u$ for some $u \in V^{(n)}$, then all $\{(v,j) \in \mathcal{E}^{(n)}_I(t); v=u\}$ are in $\mathcal{E}^{(n)}_R(t_+(t))$ and $u \in R(t_+(t))$.
If  $t_+(t) = \sigma(u)+\tau_{u,j}$ for some $(u,j) \in \mathcal{E}^{(n)}_I(t)$ , then $(u,j) \in \mathcal{E}^{(n)}_P(t_+(t))$. In addition, 
consider $(x_{k+1},x'_{k+1})$, which is the half-edge $(u,j)$ ``wants to'' be paired with if it is still possible. The half-edge $(x_{k+1},x'_{k+1})$ is considered explored from time $t_+(t)$ on, i.e.\ $(x_{k+1},x'_{k+1}) \in \mathbf{x}^{(n)}(t_+(t))$. We distinguish between the following cases for further changes in $\mathcal{X}^{(n)}(t)$ at time $t_+(t) = \sigma(u)+\tau_{u,j}$.
\begin{itemize}
\item If $(x_{k+1},x'_{k+1}) \in \mathcal{E}^{(n)}_S(t)$, then  $(x_{k+1},x'_{k+1}) \in \mathcal{E}^{(n)}_P(t_+(t))$, while all $d_{x_{k+1}}-1$ other half-edges belonging to $x_{k+1}$ (which  necessarily belong to $\mathcal{E}^{(n)}_S(t)$) move to $\mathcal{E}^{(n)}_I(t_+(t))$. Furthermore, $\sigma(x_{k+1})= t_+(t)$ and $x_{k+1} \in I^{(n)}(t_+(t))$.
\item If $(x_{k+1},x'_{k+1}) \in \mathcal{E}^{(n)}_I(t) \cup \mathcal{E}^{(n)}_R(t)$, then again $(x_{k+1},x'_{k+1}) \in \mathcal{E}^{(n)}_P(t_+(t))$, while none of the other half-edges and none of the vertices changes.
\item If $(x_{k+1},x'_{k+1}) \in \mathcal{E}^{(n)}_P(t)$, then 
take the same steps as above with $(x_{k+1},x'_{k+1})$ replaced by $(x_{k+2},x'_{k+2})$ and so on, while treating all considered half-edges as explored. 
\end{itemize}
\item Continue the above construction until $I^{(n)}(t) = \emptyset$. That is, until there are no infectious vertices left. 
\end{itemize}

\subsection{Branching processes theory background}
\label{sec:BT}

Throughout the manuscript we use several continuous time branching processes. In this section we summarize some of the results we use in the analysis of the duration of the epidemic.
Some of the branching processes that we use are  two stage branching processes in the sense that the reproduction law for the first ancestor is different from that of the other particles in the process. In  the exposition below we use a single stage branching process, but extending the results to two stage branching processes is straightforward. For further theory we refer to \cite[Chapter 6]{jagers1975branching} and \cite[Chapter 3]{van2016random}.

Assume that particles give birth to other particles according to a random point process distributed as $\{\xi(t);t \geq 0\}$. Define $\mu(t) = \mathbb{E}[\xi(t)]$.  
If $\mu(\infty) >1$  
then equation (\ref{alphadef}) has a strictly positive solution $\alpha$, which is called the Malthusian parameter of the process. We call a process supercritical if $\mu(\infty) >1$, critical if $\mu(\infty) = 1$ and subcritical if $\mu(\infty) <1$. We decorate particles in the branching process with a lifetime, distributed as some $[0,\infty]$-valued random variable $\Lambda$ and we assume that $\mathbb{P}(\xi(\Lambda)=\xi(\infty)) =1$.

Let $Z(t)$ be the number of particles in the branching process at time $t$ and $Z^{tot}(t)$ the number of particles born in the branching process up to and including time $t$. Furthermore, let $Z(t;a)$ be the number of particles alive at time $t$, born after time $t-a$.
The following Lemma follows immediately from Theorems 2.1 and 2.4 of \cite{Iksa15} and Theorem 5.4 of \cite{nerman1981}.
\begin{lemma}
\label{bpconv}
Assume  $\mu(\infty)>1$ and let $\alpha$ be the Malthusian parameter defined in  (\ref{alphadef}). Furthermore, for $\log^+ t= \log(\max(1,t))$, if there exist $\epsilon>0$ such that $\int_0^{\infty}t (\log^+ t)^{1+\epsilon}e^{-\alpha t} \mu(dt)<\infty$, then almost surely and in expectation
\begin{equation}
\label{Wlimit}
e^{-\alpha t}Z(t) \to W  \quad \mbox{and} \quad e^{-\alpha t}Z^{tot}(t) \to W' \qquad \mbox{as $t \to \infty$,} 
\end{equation}
where $W$ and $W'$ are a.s.\ finite random variables satisfying 
\begin{equation}
\label{positiveW}
\mathbb{P}(W>0)=\mathbb{P}(W'>0) =\mathbb{P}(Z(t) \not\to 0, \mbox{ for $t \to \infty$}). 
\end{equation}
If in addition $$\mathbb{E}\left[\int_0^{\infty}e^{-\alpha t}\xi(dt) \log^+\left(\int_0^{\infty}e^{-\alpha t}\xi(dt)\right)\right]<\infty,$$ then a.s.\ on $\{Z(t) \to \infty\}$ we have 
\begin{equation}
\frac{Z(t;a)}{Z(t)} \to \frac{\int_0^a\mathbb{P}(L>a)e^{-\alpha u} du}{\int_0^{\infty}\mathbb{P}(L>a)e^{-\alpha u} du}, \qquad \mbox{as $t \to \infty$}.
\end{equation}
\end{lemma}
We need the following corollary of \eqref{Wlimit} and \eqref{positiveW} in this lemma. 
\begin{corollary}
\label{bpsupercor}
Assume that the conditions of Lemma \ref{bpconv} hold.
For  $k \in \mathbb{N}$, define $\hat{T}_{k}= \inf\{t \geq 0;Z(t) \geq k\}$ and  
$\hat{T}'_{k}= \inf\{t \geq 0;Z^{tot}(t) \geq k\}$.
Then a.s.\ on $\{Z(t) \to \infty \mbox{ as $t \to \infty$} \}$ we have  that
\begin{equation}
\frac{\hat{T}_k}{\log k} \to \frac{1}{\alpha} \quad \mbox{and} \quad 
\frac{\hat{T}'_k}{\log k} \to \frac{1}{\alpha} \qquad \mbox{as $k \to \infty$.}
\end{equation}
\end{corollary}

\begin{proof}
We only provide the proof of $(\log k)^{-1}\hat{T}_k \asarrow \alpha^{-1}$ as $k \to \infty$. The second statement can be proved in an identical way. 

If  for $\omega \in \Omega$, $(\log k)^{-1}\hat{T}_k(\omega) \not\to \alpha^{-1}$ as $k \to \infty$, then  there exists $\epsilon >0$ such that $|(\log k)^{-1}\hat{T}_k(\omega) - \alpha^{-1}| >\epsilon$  for infinitely many $k \in \mathbb{N}$. Assume that $(\log k)^{-1}\hat{T}_k(\omega) - \alpha^{-1} >\epsilon$ for infinitely many $k$. If $(\log k)^{-1}\hat{T}_k(\omega) - \alpha^{-1} <-\epsilon$ for infinitely many $k$, the proof is similar.

Let $k_1, k_2, \cdots$ be an increasing infinite sequence of integers, for which $(\log k_j)^{-1}\hat{T}_{k_j}(\omega) - \alpha^{-1} >\epsilon$.
Since $(\log k_j)^{-1}\hat{T}_{k_j}(\omega) - \alpha^{-1}>\epsilon$ implies  
$\hat{T}_{k_j}(\omega) > (\alpha^{-1} + \epsilon) \log k_j$, we have 
$$e^{-\alpha \hat{T}_{k_j}(\omega)}Z(\hat{T}_{k_j}(\omega)) = e^{-\alpha \hat{T}_{k_j}(\omega)}k_j < (k_j)^{-\alpha \epsilon},
$$
which converges to $0$ as $j \to \infty$. So, $e^{-\alpha \hat{T}_{k_j}(\omega)}Z(\hat{T}_{k_j}(\omega))\to  0$ as $j \to \infty$. This implies that if $$\mathbb{P}\left(\sum_{k=1}^{\infty}   \ind\left( \frac{\hat{T}_k}{\log k} - \frac{1}{\alpha} >\epsilon\right)= \infty\right)>0,$$ then \eqref{Wlimit} and \eqref{positiveW} cannot both be true, which finishes the proof.
\end{proof}

To approximate the final phase of an epidemic we use a subcritical branching process. For these branching processes  equation \eqref{alphadef} does not necessarily have a solution. However if it has, then we may obtain some useful results. First note that $\alpha<0$.  Let the life-length of particles be distributed as $\Lambda$. From Theorem 6.2 of \cite{jagers1975branching}, we immediately obtain

\begin{lemma}
\label{sublem}
Let $\mu(\infty)< 1$. Assume 
\begin{itemize}
 \item[(i)] equation (\ref{alphadef}) has a solution, 
 \item[(ii)] $\int_0^{\infty} t e^{|\alpha| t} \Lambda(dt)<\infty$, 
 \item[(iii)] $\int_0^{\infty} t e^{|\alpha| t} \mu(dt)<\infty$ and 
 \item[(iv)] $\mathbb{E}\left[\int_0^{\infty}e^{|\alpha| t}\xi(dt) \log^+(\xi(\infty))\right]< \infty$,
\end{itemize}
then  
$e^{|\alpha| t}\mathbb{P}(Z(t)\!>\!0|Z(0)\!=\!1)$ converges to a strictly positive and finite limit.
\end{lemma}

Below we use the following Corollary of this Lemma.

\begin{corollary}
\label{bpsubcor}
Assume that the conditions of Lemma \ref{sublem} hold.
For  $k \in \mathbb{N}$, assume  $Z(0)=k$ and define $\hat{T}^*_{k}= \inf\{t \geq 0;Z(t) =0\}$.
Then, $\frac{\hat{T}^*_k}{\log k} \parrow \frac{1}{|\alpha|}$
as $k \to \infty$.
\end{corollary}

\begin{proof} 
It is enough to prove that for every $\delta \in (0,1)$ and as $k \to \infty$,
\begin{equation*}
\begin{split}
\mathbb{P}\left(\hat{T}_k^* \leq  \frac{1+\delta}{|\alpha|}\log k \right) &=  \mathbb{P}\left(Z\left(\frac{1+\delta}{|\alpha|}\log k\right)=0|Z(0)=k\right) \to 1 \quad \mbox{and}\\ 
\mathbb{P}\left(\hat{T}_k^* \leq  \frac{1-\delta}{|\alpha|}\log k \right) &=  \mathbb{P}\left(Z\left(\frac{1-\delta}{|\alpha|}\log k\right)=0|Z(0)=k\right) \to 0.
\end{split}
\end{equation*}
Note that $\{Z(t);t\geq 0\}$ is distributed as $\{\sum_{j=1}^{k}Z_j(t);t \geq 0\}$, where for $j \in \mathbb{N}$, the processes $\{Z_j(t);t \geq 0\}$ are independent branching processes distributed as the subcritical branching process satisfying $Z_j(0)=1$.
So, $$\{Z(t) =0\} = \cap_{j=1}^{k}\{Z_j(t)=0\}$$ and we obtain that 
$$\mathbb{P}(Z(t) = 0|Z(0)=k) = \left(\mathbb{P}(Z(t) = 0|Z(0)=1)\right)^k.$$
Therefore,
\begin{multline*}
\mathbb{P}\left(Z\left(\frac{1+\delta}{|\alpha|}\log k\right)=0|Z(0)=k\right)\\ = 
\left(\mathbb{P}\left(Z\left(\frac{1+\delta}{|\alpha|}\log k\right)=0|Z(0)=1\right) \right)^k.
\end{multline*}

By Lemma \ref{sublem} we know that there exists $t_0>0$ such that for all $t>t_0$ we have both 
$$\mathbb{P}(Z(t)>0|Z(0)=1) <  e^{-|\alpha|(1-\delta/2)t}$$ and
$$\mathbb{P}(Z(t)>0|Z(0)=1) >  e^{-|\alpha|(1+  \delta)t}.$$
So, we obtain
\begin{equation*}
\begin{split}
\mathbb{P}\left(Z\left(\frac{1+\delta}{|\alpha|}\log k\right)=0|Z(0)=k\right) 
&=  \left(\mathbb{P}\left(Z\left(\frac{1+\delta}{|\alpha|}\log k\right)=0|Z(0)=1\right) \right)^k \\
&>  \left(1- e^{-(1-\delta/2)(1+\delta)\log k}\right)^k\\
&=  \left(1- k^{-(1 + \delta/2-\delta^2/2)}\right)^k\\
&= \left(1- \frac{k^{-(\delta-\delta^2)/2}}{k}\right)^k,
\end{split}
\end{equation*}
which converges to 1 for $\delta\in (0,1)$, by $(1-c k^{-1})^k \to e^{-c}$ as $k \to \infty$. Similarly, 
\begin{equation*}
\begin{split}
\mathbb{P}\left(Z\left(\frac{1-\delta}{|\alpha|}\log k\right)=0|Z(0)=k\right)
&=   \left(\mathbb{P}\left(Z\left(\frac{1-\delta}{|\alpha|}\log k\right)=0|Z(0)=1\right) \right)^k \\
&< \left(1- e^{-(1+\delta)(1-\delta)\log k}\right)^k\\
&=  \left(1- k^{-(1-\delta^2)}\right)^k\\  
&=  \left(1- \frac{k^{\delta^2}}{k}\right)^k,
\end{split}
\end{equation*}
which converges to 0, by $(1-c k^{-1})^k \to e^{-c}$ as $k \to \infty$ and the proof of the corollary is complete.
\end{proof}

\section{Heuristics}
\label{sec:heur}

In this subsection we provide some heuristic arguments for Theorem \ref{mainthm}.
If a large outbreak occurs, the epidemic can be subdivided into three phases, which can be roughly described as follows. Let $\epsilon>0$ be small.
In the initial phase the number of susceptible vertices decreases from $n-1$ to $(1-\epsilon) n$. In the intermediate phase the number of susceptible vertices decreases from $(1-\epsilon) n$ to $(q^*+\epsilon)n$. While the final stage of the epidemic lasts from the moment that the number of susceptible vertices is $(q^*+\epsilon)n$ until there are no more infectious vertices in the population. 

\subsection{The initial and intermediate phase of the epidemic}

The primary intuition for the initial phase is that $|I(t)|$, the number of infectious vertices at time $t$ and $|I(t)| + |R(t)|$, the number of vertices infected before time $t$ are well approximated by a branching process with mean measure given by  \eqref{mudtfirst} as long as $n^{-1} |S(t)|>1-\epsilon$ for $\epsilon>0$ but small.
The result of Lemma \ref{lemint2} then follows by applying Corollary \ref{bpsupercor} to $k= \epsilon n$.

To justify the use of \eqref{mudtfirst}, assume that the degree of a vertex uniformly taken from the population of size $n$ has exactly the same distribution function as $D$, then a newly infected vertex has degree distribution $\tilde{D}$, because of size biasing effects (see e.g.\ \cite{durrett2007random}). Apart from one (the infector) all of the neighbors of this newly infected vertex are susceptible with high probability. 
A newly infected vertex stays infectious for a random time distributed as $L$. Neighbors contact each other with intensity $\beta$, and if the contact is between a susceptible and an infectious vertex then the susceptible one becomes infected, which can be interpreted as being a child of his or her infector in the approximating branching process.
So in an approximating branching process we obtain expression \eqref{mudtfirst}: 
$$\mu'(dt) = \mathbb{E}[\tilde{D}-1] \beta e^{-\beta t} \mathbb{P}(L>t) dt,$$
where $\mathbb{E}[\tilde{D}-1]$ is the expected number of susceptible neighbors of a newly infected vertex,
$\beta e^{-\beta t}$ is the density of the time since infection of the first contact with a given neighbor, while $\mathbb{P}(L>t)$ is the probability that the vertex is still infectious at this time of first contact.
The Malthusian parameter of this approximating branching process is therefore given by (\ref{beginmalt}).

In the intermediate phase of the epidemic, $|S(t)|$, $|I(t)|$, and the number of infectious-susceptible neighbor pairs are all $\Theta(n)$. This implies that changes in $n^{-1}|S(t)|$, occur at an $\Theta(1)$ rate and the intermediate phase has duration $\Theta(1)$.

Our proof of Lemma \ref{lemint2}, however makes use of the fact that the initial and intermediate phase of the epidemic are, with some extra conditions on $D$ and $L$, studied by Barbour and Reinert in
\cite{barbour2013approximating}. They study the evolution of  $|S(T_0 +(\alpha')^{-1} (\frac{1}{2}\log[n] +t))|$, where $T_0= \inf\{t \geq 0;|S(t)| \leq n -\sqrt n\}$ is the time when $\sqrt{n}$ vertices are infected or recovered.
As a corollary of the results of \cite{barbour2013approximating} it follows that for $T'_{\gamma}(n)$ defined as in Lemma \ref{lemint2},  $T'_{\gamma}(n)  - (\alpha')^{-1} \log n$ converges in distribution as $n \to \infty$.
We avoid the extra conditions of \cite{barbour2013approximating} at the cost of only being able to study the convergence of  $T'_{\gamma}(n)/(\log n)$.

\subsection{The final phase of the epidemic}

In order to describe the end of the epidemic more work is required. We use that for $1-q^*-\gamma>0$ but small, the time interval  between $T'_{\gamma}(n)$ and $T^*(n)$, none of the quantities $n^{-1}|S^{(n)}(t)|$ and $n^{-1}|\mathcal{E}^{(n)}_S(t)|$, $n^{-1}|\mathcal{E}^{(n)}_I(t)|$, $n^{-1}|\mathcal{E}^{(n)}_R(t)|$ and $n^{-1}|\mathcal{E}^{(n)}_P(t)|$ (as defined in Section \ref{sec:construct}) change much. So, we assume that during the final stages of the epidemic, the environment of newly infected vertices is  more or less constant. That is, in our approximation we assume that the degree distribution and the fraction of the neighbors which are still susceptible of newly infected vertices are constant during this final phase. In particular, the degree distribution of a vertex infected during the final phase of the epidemic should be well approximated by the size biased degree distribution of ultimately susceptible vertices, while the fraction of susceptible neighbors of a newly infected vertex in this phase of the epidemic should be well approximated by the fraction of susceptible neighbors of ultimately susceptible vertices. We now find those quantities. 

Let $D^*$ be a random variable, such that the degree of a uniformly chosen ultimately susceptible vertex converges in distribution to $D^*$ as $n \to \infty$.
And let $p^*_{ss}$ be the probability that a given neighbor of an ultimately susceptible vertex is ultimately susceptible itself. Below we show that $p_{ss}^*$ is indeed well defined, and whether a given neighbor of an ultimately susceptible vertex is susceptible is independent of the degree of that vertex.

The end of the epidemic is then described by offspring measure
\begin{equation}\label{mudt3}
\mu^*(dt) = \mathbb{E}[\tilde{D}^*-1]p_{ss}^* \beta e^{-\beta t} \mathbb{P}(L>t) dt,
\end{equation}
which is derived in the same way as equation \eqref{mudtfirst} and where $\tilde{D}^*$ is the size-biased variant of $D^*$.
Below we then derive that 
$$\mathbb{E}[\tilde{D}^*-1]= \frac{\mathbb{E}[(\tilde{D}-1)(1-\psi+\psi \tilde{q}^*)^{\tilde{D}-1}]}{\tilde{q}^*}
\quad \mbox{and} \quad 
p_{ss}^* = \frac{\tilde{q}^{*}}{1-\psi+\psi \tilde{q}^{*}}.$$
Combining the above with \eqref{mudt3} and Corollary \ref{bpsubcor} then gives Lemma \ref{lemint3}.

\subsubsection{Degree distribution of ultimately susceptible individuals}
In this section we use ideas related to so-called susceptibility sets (see e.g.\ \cite{ball2001stochastic,ball2009threshold,ball2014epidemics}), although we do not define those set here. The arguments of this section are self-contained.

It is important to note that in the epidemic process the event that a vertex is ultimately recovered does not depend on its infectious period, even when infectious periods are random. This fact helps us to derive the probability of a vertex being ultimately susceptible and of degree $k$ (as in \cite{andersson1999epidemic}), which then yields the degree distribution of the ultimately susceptible.

Assume that a large outbreak occurs, which happens with the same probability as the survival of the branching process approximating the early spread of the epidemic, (see e.g.\ \cite{ball2009threshold}). Recall that there is only one initially infectious individual. So, as $n \to \infty,$ the probability that a uniformly chosen vertex is the initial infectious vertex converges to 0. 
Therefore, the probability that a uniformly chosen vertex $v$ is ultimately susceptible (i.e.\ it escapes the epidemic) is given by
\begin{equation}\label{eq:14}\xi=\sum_{k=0}^{\infty}\xi_{k}p_{k},
\end{equation}
where $\xi_{k}$ is probability that a vertex of degree $k$ does not acquire the infection by any of its neighbors until the end of the epidemic. We denote a neighbor of  vertex $v$ by $u$. Recall that $1-\psi$ is the probability that $u$ does not contact $v$ during its infectious period, if $u$ would become infected.  Let $\tilde{q}^{*}$ denote the probability that $u$
escapes the epidemic (we determine $\tilde{q}^{*}$ later). Then, $\xi_{k}$ is given by
\begin{equation}\label{eq:15}
\xi_{k}=\sum_{l=0}^{k}\dbinom{k}{l}(\tilde{q}^{*})^{k-l}(1-\tilde{q}^{*})^{l}(1-\psi)^{l}=(1-\psi+\psi \tilde{q}^{*} )^{k},
\end{equation}
because each of the neighbors of $v$ should be either not infected itself, or, if infected, not contact $v$ during the infectious period.
Using \eqref{eq:15} in \eqref{eq:14}, the probability for a uniformly chosen vertex to escape the epidemic is given by
\begin{equation}\label{eq:16}
\xi=\sum_{k=0}^{\infty}p_{k} (1-\psi+\psi \tilde{q}^{*} )^{k}.
\end{equation}

Similarly, the probability $\tilde{q}^{*}$ that $u$ escapes infection by all of its neighbors other than $v$  is given by
\begin{equation}\label{eq:17}\tilde{q}^{*}=\sum_{k=0}^{\infty}\tilde{\xi_{k}}\tilde{p}_{k},
\end{equation}
where $\tilde{\xi_{k}}$ is the probability
that a degree $k$ vertex does not acquire the infection from
$k-1$ given neighboring vertices and is defined as
\begin{equation}\label{eq:18}
\tilde{\xi }_{k}=\sum_{l=0}^{k-1}\dbinom{k-1}{l}(\tilde{q}^{*})^{k-l-1}(1-\tilde{q}^{*})^{l}(1-\psi)^{l}=(1-\psi+\psi \tilde{q}^{*})^{k-1}.
\end{equation}
Here we consider only $k-1$ of the $k$ neighbors of $u$ because we assume that $u$ does not acquire infection from $v$. Equations \eqref{eq:17} and \eqref{eq:18} give that  $\tilde{q}^{*}$ is a solution of
\begin{equation}\label{eq:19}
\tilde{q}^{*}= \sum_{k=0}^{\infty} \tilde{p}_{k}(1-\psi+\psi \tilde{q}^{*})^{k-1}.
\end{equation} 
In this heuristic argument we claim without proof that $\tilde{q}^{*}$ is the smallest solution of this identity.
So we have an implicit expression for the probability $\tilde{q}^{*}$ that neighbor $u$ escapes the epidemic. 
Moreover, from \eqref{eq:15} we obtain the probability that a vertex of degree $k$ escapes the epidemic. From this we deduce that the probability that an ultimately susceptible individual has degree $k$ (say $p_{k}^{*}$) is given by
\begin{equation}\label{eq:21}
p_{k}^{*}=\frac{\xi_{k}p_{k}}{\xi}=\frac{(1-\psi+\psi \tilde{q}^{*})^{k}p_{k}}{\sum_{j=1}^{\infty}p_j(1-\psi+\psi \tilde{q}^{*})^j },\end{equation}
where $\xi$ is a normalizing constant and is defined in \eqref{eq:16}. The size biased distribution of the ultimately susceptible individuals is given through
\begin{equation}\label{eq:22}\begin{split}
\tilde{p}_{k}^{*}&=\frac{kp_{k}^{*}}{\sum_{j=1}^{\infty}jp_{j}^{*}}=\frac{kp_k(1-\psi+\psi \tilde{q}^{*})^{k}}{\sum_{j=1}^{\infty}jp_{j}(1-\psi+\psi \tilde{q}^{*})^{j}},\\\\
&=\frac{\tilde{p}_k(1-\psi+\psi \tilde{q}^{*})^{k-1}}{\sum_{j=1}^{\infty}\tilde{p}_{j}(1-\psi+\psi \tilde{q}^{*})^{j-1}}=\frac{\tilde{p}_k(1-\psi+\psi \tilde{q}^{*})^{k-1}}{\tilde{q}^{*}}.
\end{split}\end{equation}

\subsubsection{Fraction of ultimately susceptible neighbors of an ultimately susceptible vertex}
Let $v$ be an arbitrary vertex of degree $k$ and $u$ one of its neighbors. We compute the fraction of
neighbors of the ultimately susceptible which are also ultimately susceptible as the following conditional probability:
\begin{equation}\label{eq:25}\begin{split}&
p^*_{ss}(k)=\mathbb{P}\text{($u$ is ultimately susceptible $|$ $v$ is ultimately susceptible)},\\
&=\frac{\mathbb{P}\text{($v$ and $u$ are ultimately susceptible)}}{\mathbb{P}\text{($v$ is ultimately susceptible)}}=\frac{\tilde{q}^{*}\, \tilde{\xi}_{k}}{\xi_{k}}=\frac{\tilde{q}^{*}}{1-\psi+\psi \tilde{q}^{*}}.
\end{split}\end{equation}
Note that this probability is independent of the degree $k$ of vertex $v$ and therefore we can write $p^*_{ss}(k)=p^*_{ss}$.

To understand \eqref{eq:25}, recall that $ \tilde{q}^{*}$ is the probability that the initially susceptible neighbor $u$ escapes the infection from all its neighboring vertices, apart from possibly $v$, $\tilde{\xi}_k$ is the probability that $v$ escapes infection from all of its neighbors, apart from possibly $u$, and $\xi_k$ is the unconditional probability that vertex $v$ does not acquire the infection until the end of the epidemic.

\section{Proof of Lemma \ref{lemint2}}
\label{secinit}

We split the proof in two lemmas which trivially imply Lemma \ref{lemint2}.
\begin{lemma}
\label{lemint4}
Assume that A1-A3 of Assumptions \ref{mainassump} hold, then 
$$\frac{T'_{\gamma}(n)}{\log n} \ind(\mathcal{M}^{(n)})  \leq \frac{1}{\alpha'}+\delta \qquad \mbox{w.h.p.\ for every $\delta>0$.}$$
\end{lemma}
\begin{lemma}
\label{lemint5} Assume that A1-A3 of Assumptions \ref{mainassump} hold. Further assume that $\mathbb{E}[D^2]<\infty$, then
$$\frac{T'_{\gamma}(n)}{\log n}  \geq \frac{1}{\alpha'+\delta} \qquad \mbox{w.h.p.\ for every $\delta>0$.}$$
\end{lemma}
\noindent Note that $\mathbb{E}[D^2]=\infty$ implies $\alpha' = \infty$ and the equivalent of Lemma \ref{lemint5} is meaningless. Lemma \ref{lemint4} still holds in that case.

\begin{proof}[Proof of Lemma \ref{lemint4}] 
Assume first that $D^{(n)}$ has uniformly bounded support, that is, there exist $K>0$ such that 
$\mathbb{P}(D^{(n)} >K) =0$ for all $n \in \mathbb{N}$. 
Furthermore, assume that there exists $L_{\text{max}} \in (0,\infty)$ such that  $\mathbb{P}(L>L_{\text{max}})=0$, i.e.\ we assume that $L$ has bounded support. 
Under those assumptions the conditions of \cite[Thm.~3.3]{barbour2013approximating} are satisfied. 
Note that in the notation of \cite{barbour2013approximating}, $\lambda$ is the Malthusian parameter ($\alpha'$ in our notation) and $N$ is the population size ($n$ in our notation). It is easily deduced from equation (3.11) and the definition of $\tau_N$ on page 27 of \cite{barbour2013approximating} that $\tau_N/[\log N] \parrow  1/(2 \lambda)$ on $\mathcal{M}^{(n)}$.
Finally, the expression  $\hat{s}_l(u)$ in  \cite{barbour2013approximating} is independent of $N$ for all $l \in \{1,2,\cdots, K\}$.
Translating the notation of  \cite[Thm.~3.3]{barbour2013approximating} to our notation we obtain as an immediate corollary  that for every $\gamma' \in (0,1-q^*)$ and every  $\delta>0$, $$\frac{1}{n} \left|S^{(n)} \left(\left((\alpha')^{-1}+ \delta\right)\log n\right)\right|\ind(\mathcal{M}^{(n)}) < q^* +\gamma' \qquad \mbox{w.h.p.}$$  

To obtain the results without the extra conditions, let $K = K(\delta)$ be a large constant satisfying some properties specified later. Mark (before the pairing) all half-edges belonging to 
vertices with degree strictly larger than $K$.
By assumptions (A1) and (A2) one can make the fraction of half-edges that are marked to be arbitrary small by choosing $K$ and $n$ large enough.  
The next step is to pair all half-edges (ignoring whether they are marked and unmarked) uniformly at random as before. 
Then delete all edges which contain at least one marked half-edge. 
If a fraction $\delta_1 = \delta_1(K)$ of the half-edges is marked then the remaining degree distribution of the graph is a Mixed Binomial distribution with random ``number of trials parameter'' $D^{(n)} \ind(D^{(n)} \leq K)$ and ``probability parameter'' $1-\delta_1$. 
Let $D^{(n)}_K$ be distributed as this Mixed Binomial random variable, and $\tilde{D}^{(n)}_K$ be the size-biased variant of $D^{(n)}_K$.
It follows immediately from assumptions (A1) and (A2) that 
$$\lim_{K \to \infty} \lim_{n \to \infty} \mathbb{E}[D^{(n)}_K] = \mathbb{E}[D].$$
Furthermore, for both $\mathbb{E}[D^2]= \infty$ and $\mathbb{E}[D^2]<\infty$, (A1) and (A3) imply
that 
$$\lim_{K \to \infty} \lim_{n \to \infty} \mathbb{E}[(D^{(n)}_K)^2] = \mathbb{E}[D^2].$$
In particular, we obtain that $$\lim_{K \to \infty} \lim_{n \to \infty} \mathbb{E}[\tilde{D}^{(n)}_K -1] = \mathbb{E}[\tilde{D}-1].$$

In addition we consider an epidemic on the newly created (thinned) graph with infectious period distribution $$L' = L \ind(L < L_{\text{max}}) + L_{\text{max}} \ind(L \geq L_{\text{max}}).$$ So, in the new model we have deleted some edges and shortened some infectious periods, which make that the epidemic spreads faster in the original model.

For this new epidemic  we deduce from $(\ref{beginmalt})$ that the Malthusian parameter is the $x \in \mathbb{R}$ satisfying   
\begin{equation}
\label{beginmalt2}
\frac{1}{\mathbb{E}[\tilde{D}^{(n)}_K -1]}= \int_0^{L_{\text{max}}} e^{-x t}  \beta e^{-\beta t} \mathbb{P}(L>t) dt =f(x,L_{\text{max}}).
\end{equation}
We define $f(x,\infty)$ in the obvious way. 
Note that $f(x,L_{\text{max}})$ is continuous and decreasing in $x$ and continuous and increasing in $L_{\text{max}}$. Furthermore, if $\psi \mathbb{E}[\tilde{D}-1]>1$, then $$f(0,\infty) = \psi > \frac{1}{\mathbb{E}[\tilde{D}-1]}.$$   
While $\lim_{x \to \infty}  f(x,L_{\text{max}}) =0$ for all $L_{\text{max}} \in (0,\infty]$. It follows 
that the solution of (\ref{beginmalt2}) converges to $\alpha'$ as $K \to \infty$ and $L_{\text{max}} \to \infty$.
In particular, for every $\delta>0$, there exists $K_0<\infty$ and $L_0<\infty$ such that for all $K>K_0$ and $L_{\text{max}}>L_0$, the $x \in \mathbb{R}$ solving  
(\ref{beginmalt2}) satisfies $1/x < 1/\alpha'+\delta/2$.

So, by choosing $L_{\text{max}}$ and $K$ large enough (but finite), we are in the realm of \cite[Thm.~3.3]{barbour2013approximating} and for the corresponding model we obtain that for 
every $\gamma' \in (0,1-q^*)$ and $\delta>0$ with high probability it holds that, $$\frac{1}{n} \left|S^{(n)} \left(\left(\frac{1}{\alpha'}+ \delta/2+\delta/2\right)\log n\right)\right| < q^* +\gamma',$$
which finishes the proof of Lemma \ref{lemint4}.
\end{proof}

\begin{proof}[Proof of Lemma \ref{lemint5}]
In order to prove the lemma we prove the following stronger statement: 
The number of vertices affected by the epidemic up to time $\frac{\log n}{\alpha' + \delta}$ satisfies $|n-S^{(n)}(\frac{\log n}{\alpha' + \delta})|= o(n)$ with high probability  
for all $\delta>0$.

Now, for $\delta_1>0$, let $\alpha_1 = \alpha_1(\delta_1)$ satisfy   
\begin{equation}
\label{beginmalt3}
\frac{1}{\mathbb{E}[\tilde{D}-1]+\delta_1}= \int_0^{\infty} e^{-\alpha_1 t}  \beta e^{-\beta t} \mathbb{P}(L>t) dt.
\end{equation}
As before, because $R_0>1$, we know that $\alpha_1$ exists and is positive for all $\delta_1 \geq 0$ and is continuous increasing in $\delta_1$ on $[0,\infty)$. 
In particular, for every $\delta>0$, we can and do choose $\delta_1>0$ such that $\alpha_1(\delta_1) < \alpha'+ \delta/2$.

We use the notation of Section \ref{sec:construct}, where the vertices in $V^{(n)}$ are labeled such that the degree sequence $d_1,d_2,\cdots, d_n$ is non-decreasing.
We also use that $\ell_2(n) =O(n)$ 
by assumption (A3) and the assumption $\mathbb{E}[D^2]<\infty$ (or equivalently by $\mathbb{E}[\tilde{D}]<\infty$).

Let $\epsilon_1 \in (0,1)$ be a number to be specified later. For $i \leq \epsilon_1 n$ define
 the random variable $D'^{(n)}(\mathbf{x};i)$  through 
$$\mathbb{P}(D'^{(n)}(\mathbf{x};i) = k) = \frac{\sum_{v=1}^n \ind(d_v=k) \ind(v \not\in\{x_0,x_1,\cdots x_i\})}{\sum_{v=1}^n \ind(v \not\in\{x_0,x_1,\cdots x_i\})}.$$
That is, $D'^{(n)}(\mathbf{x};i)$ is the degree distribution of the vertices not chosen in the first $i$ elements of $\mathbf{x}$.

Note that for all $i \leq i_0 = \lfloor \epsilon_1 n \rfloor$, the random variable $D'^{(n)}(\mathbf{x};i)$ is stochastically dominated by $D''^{(n)}(\epsilon_1)$, which is defined through $$\mathbb{P}(D''^{(n)}(\epsilon_1) = k) = \frac{\sum_{v=i_0+1}^n \ind(d_v=k)}{n-i_0}.$$
So, $D''^{(n)}(\epsilon_1)$ is the degree distribution of a vertex chosen uniformly from the $n-i_0$ vertices with highest degrees, which is stochastically increasing in $i_0$.
Let $\tilde{D}''^{(n)}(\epsilon_1)$ be the size biased variant of $D''^{(n)}(\epsilon_1)$.
It follows that
$$\mathbb{E}[\tilde{D}''^{(n)}(\epsilon_1)]= \frac{\sum_{v=i_0+1}^n (d_v)^2}{\sum_{v=i_0+1}^n d_v}.$$
Observe that $\sum_{v=i_0+1}^n (d_v)^2 \leq \sum_{v=1}^n (d_v)^2 = \ell_2(n)$, while
$$\sum_{v=i_0+1}^n d_v = \sum_{v=1}^n d_v- \sum_{v=1}^{i_0} d_v \geq \ell(n)-i_0 \mathbb{E}[D^{(n)}] = \ell(n)-i_0 \frac{\ell(n)}{n} \geq \ell(n)(1-\epsilon_1).$$ So, 
$$\mathbb{E}[\tilde{D}''^{(n)}(\epsilon_1)] \leq \frac{\ell_2(n)}{\ell(n)(1-\epsilon_1)} = \frac{1}{1-\epsilon_1}\mathbb{E}[\tilde{D}^{(n)}].$$

Note that (apart from possibly $x_0$), in the construction of $\{\mathcal{X}^{(n)}(t);t \geq 0\}$ the degree of a vertex added to $V^{(n)} \setminus S^{(n)}(t)$ is stochastically smaller than $\tilde{D}''^{(n)}(\epsilon_1)$, as long as $t<t_0$, where $t_0= t_0(\epsilon_1) = \max\{t>0;|\mathbf{x}^{(n)}(t)| \leq i_0\}$. That is up to we add the $i_0$-th vertex to $\{\mathcal{X}^{(n)}(t);t \geq 0\}$ the number of vertices in $V^{(n)} \setminus S^{(n)}(t)$ is less than the number of particles in a branching process with offspring 
measure $$\mu''^{(n)}(dt;i_0) = \mathbb{E}[\tilde{D}''^{(n)}(\epsilon_1)] \beta e^{-\beta t} \mathbb{P}(L>t)dt.$$
Denote the number of particles in this branching process at time $t$ ($t \geq 0$) by $Z''^{(n)}(t)$.

Because $\mathbb{E}[\tilde{D}^{(n)}] \to \mathbb{E}[\tilde{D}]$ as $n \to \infty$, we have that for every $\delta_1>0$ we can choose $\epsilon_1= \epsilon_1(\delta_1)>0$ and $n_0= n_0(\delta_1) \in \mathbb{N}$, such that $\mathbb{E}[\tilde{D}''^{(n)}(\epsilon_1)]< \mathbb{E}[\tilde{D}] + \delta_1$ for all $n > n_0$.

So for $\epsilon_1= \epsilon_1(\delta_1)$ and $n_0 = n_0(\delta_1)$ as above,
$\{Z''^{(n)}(t); t \in (0,t_0)\}$ is dominated by a branching process 
$\{Z''(t); t \in (0,t_0)\}$ with offspring measure 
$$\mu''(dt) = (\mathbb{E}[\tilde{D}]+ \delta_1) \beta e^{-\beta t} \mathbb{P}(L>t)dt.$$
This branching process has a Malthusian parameter $\alpha_1(\delta_1)$, which satisfies equation (\ref{beginmalt3}) and is less than $\alpha'+ \delta/2$.

Now observe that by Lemma \ref{bpconv}, with high probability 
$$Z''\left(\frac{\log n}{\alpha' + \delta}\right)
= 
O\left(e^{(\alpha'+ \frac{\delta}{2})\frac{\log n}{\alpha' + \delta}}\right) = O\left(n^{\frac{\alpha'+ \delta/2}{\alpha' + \delta}}\right) = o(n).$$

Since $i_0 = \theta(n)$ and the number of individuals infected before time $t$ is stochastically less than
$Z''(\frac{\log n}{\alpha' + \delta})$, we obtain that  $n-|S(\frac{\log n}{\alpha' + \delta})| =o(n)$ with high probability for all $\delta>0$.
\end{proof}

\section{The final stage of the epidemic}
\label{secext}

\subsection{Proof of Lemma \ref{subend}}
Before we prove Lemma \ref{lemint3}, we first provide the proof of Lemma \ref{subend}. To this end 
we use the following almost trivial observation.

Let $D_x^*$ be a random variable with distribution defined through 
\begin{equation}
\label{Dxstar}
p_{k}^{*}(x) = \mathbb{P}(D_x^*=k) = \frac{p_k x^k}{\sum_{\ell=0}^{\infty}p_{\ell} x^{\ell}}
\end{equation}
for $k \geq 0$. Here $p_k= \mathbb{P}(D=k)$ as in Section \ref{sec:conf}.
\begin{claim}\label{dalem}
For $x \in (0,1)$ all moments of the random variable $D_x^*$
are finite, regardless of the distribution of $D$. 
\end{claim}

\begin{proof}Consider the $j$th moment of $D_x^*$
\begin{equation*}\label{eq:23}\begin{split}
\mathbb{E}[(D_x^*)^j]=\sum_{k=1}^{\infty}k^{j}p_{k}^{*}(x)=\sum_{k=1}^{\infty}k^{j}\frac{p_k x^k}{\sum_{l=0}^{\infty}p_l x^l} \leq \frac{\max_{k \in \mathbb{N}} k^{j}x^k}{\sum_{l=0}^{\infty}p_l x^l}.
\end{split}\end{equation*}
Because the numerator is finite and the denominator strictly positive the claim follows.
\end{proof}

In particular, note that 
\begin{multline*}
\mathbb{E}[\tilde{D}^*] = \frac{\mathbb{E}[\tilde{D}Q^{\tilde{D}-1}]}{\tilde{q}^*} =  \frac{\mathbb{E}[D^2 Q^{D-1}]}{\mathbb{E}[D] \tilde{q}^*}\\
=
\frac{\mathbb{E}[D^2(1-\psi+\psi \tilde{q}^*)^{D}]}{\mathbb{E}[Q^{D}]}\frac{\mathbb{E}[Q^{D}]}{\mathbb{E}[D] \tilde{q}^* Q}
= \mathbb{E}[(D^*_Q)^2] \frac{\mathbb{E}[Q^{D-1}]}{\mathbb{E}[D] \tilde{q}^*}<\infty,
\end{multline*}
where
\begin{equation}
\label{Qdef}
Q= 1-\psi+\psi \tilde{q}^*.
\end{equation}
For notational convenience we continue using $Q$ below.  

\begin{proof}[Proof of Lemma \ref{subend}]
Recall the definition of $\alpha^{\dagger}$ from \eqref{alphadagger}. By Claim \ref{dalem} $\mu^*(dt)$ is well defined.
In Remark \ref{Remadag} we argued that $\alpha^* \geq \alpha^{\dagger}$; $\alpha^{\dagger}<0$ and
the function $g^*(x)$ is continuous and strictly decreasing for $x >\alpha^{\dagger}$. 
This implies that  $g^*(x)$ is continuous and strictly decreasing on the non empty interval $(-\alpha^{\dagger},0)$ and $\alpha^*\in [-\alpha^{\dagger},0)$ if and only if $$R_0^* = \int_0^{\infty} \mu^*(dt)=g^*(0) <1.$$ 

To show that $R_0^*<1$, observe that
the function $$g(x) = \sum_{k=1}^{\infty} \tilde{p}_k (1-\psi + \psi x)^{k-1}
$$
is convex and analytic on $x \in [0,1]$ and has derivative
$$\frac{d}{dx}g(x) = \psi \sum_{k=1}^{\infty} (k-1) \tilde{p}_k (1-\psi + \psi x)^{k-2} =  \psi \mathbb{E}[(\tilde{D}-1)(1-\psi+\psi x)^{\tilde{D}-2}].$$
Furthermore, by the definition of 
 $\tilde{q}^*$ (see \eqref{endext}) and the convexity of $g(\cdot)$, $\tilde{q}^*$ and 1 are the only two solutions of  the equation $g(x)-x=0$ in $[0,1]$. We recall that $\tilde{q}^*<1$.
Because $g(x)-x$ is convex, we know that the function $g(x)-x$ has to be negative between its two zeros (i.e.\ between $\tilde{q}^*$ and 1). 
This, together with  $\frac{d}{dx}g(x)|_{x= \tilde{q}^*}= R_0^*$ (by \eqref{R0stardef})  implies that $R_0^*<1$, which finishes the proof.
\end{proof}

\subsection{Time until the end of the epidemic}

In this section we use the construction of the epidemic generated graph as presented in Section \ref{sec:construct}. We restrict ourselves to major outbreaks. Our approach is to define a random time $t_1=t_1^{(n)}$, when the fraction of  susceptible vertices among all vertices is larger than, but close to, its asymptotic value and sandwich (w.h.p.)\ the process $\{|I^{(n)}(t)|;t \geq t_1^{(n)}\}$ between two branching processes and then find the time until those branching processes go extinct. 

Let $\{\mathcal{X}^{(n)}(t);t \geq 0\}$
be as in Section \ref{sec:construct}. In our analysis below we consider
$|\mathcal{E}^{(n)}_S(t)|$, $|\mathcal{E}^{(n)}_P(t)|$  and $$\sum_{v\in S^{(n)}(t)} d_v \ind(d_v \geq k)= \sum_{v\in V^{(n)}} d_v \ind(d_v \geq k) \ind(v \in S^{(n)}(t)).$$
Note that $|\mathcal{E}^{(n)}_S(t)|$ and $\sum_{v\in S^{(n)}(t)} d_v \ind(d_v \geq k)$ are decreasing in $t$, while $|\mathcal{E}^{(n)}_P(t)|$ is increasing in $t$.

For all $n,k \in \mathbb{N}$ define the constant $\hat{d}^{(n)}_k$ by
$$\hat{d}^{(n)}_k = \sup_{n' \geq n}\mathbb{P}(D^{(n')} = k).$$
Observe that by definition $\hat{d}^{(n)}_k \geq \hat{d}^{(n')}_k \geq \mathbb{P}(D^{(n')} = k)$ for all $n' \in \mathbb{N}_{\geq n}$ and for all $k \in \mathbb{N}$.
Furthermore, it follows immediately from $D^{(n)} \darrow D$ that $\hat{d}^{(n)}_k \to \mathbb{P}(D = k)$. However, note that in general $\sum_{k=1}^{\infty} \hat{d}^{(n)}_k$ may be strictly larger than 1 and there is no reason to assume that
$\sum_{k=1}^{\infty} \hat{d}^{(n)}_k$ converges to 1 as $n \to \infty$.

For $\epsilon \in (0,\psi (1-\tilde{q}^*))= (0,1-Q)$ define
\begin{equation*}
\begin{split}
t_a^{(n)}(\epsilon) & = \inf\{t>0;|\mathcal{E}_S^{(n)}(t)| \leq \mathbb{E}[(Q + \epsilon)^{\tilde{D}}]\ell(n)\}, \\
t_b^{(n)}(\epsilon) & = \inf\{t>0;|\mathcal{E}_P^{(n)}(t)| \geq  \ell(n) -1 - (Q + \epsilon)^2 \ell(n)\},\\
t_c^{(n)}(\epsilon) & =\inf\{t>0; \sum_{v\in S^{(n)}(t)} d_v \ind(d_v \geq k)  \leq
\sum_{j=k}^{\infty} n j \hat{d}^{(n)}_j (Q + \epsilon)^j
\mbox{ for all $k \in \mathbb{N}$}\},
\end{split}
\end{equation*}
where the infimum of an empty set is $\infty$.
Let $$t_1^{(n)}(\epsilon) = \max(t_a^{(n)}(\epsilon) ,t_b^{(n)}(\epsilon) ,t_c^{(n)}(\epsilon))$$ and define the event $\mathcal{A}^{(n)}_1(\epsilon) = \{ t_1^{(n)}(\epsilon)< \infty\}$. 
Let $\mathcal{A}^{(n)}_2(\epsilon)$ be the event that all of the following events hold.
\begin{eqnarray*}
|\mathcal{E}_S^{(n)}(\infty)| & > & \mathbb{E}[(Q - \epsilon)^{\tilde{D}}] \ell(n), \\
|\mathcal{E}_P^{(n)}(\infty)| & < &  \ell(n) -1 - (Q - \epsilon)^2 \ell(n), \\
\sum_{v\in S^{(n)}(\infty) } d_v \ind(d_v \geq k) & \geq & \sum_{j=k}^{\infty} n j \mathbb{P}(D = j) 
(Q - \epsilon)^{j} \mbox{ for all $k \in \mathbb{N}_{\leq \lfloor 1/\epsilon\rfloor}$}.
\end{eqnarray*}
Finally define $\mathcal{A}^{(n)}(\epsilon)=\mathcal{A}^{(n)}_1(\epsilon)\cap \mathcal{A}^{(n)}_2(\epsilon)$.
\begin{lemma}
\label{eeslemma2}
For all $\epsilon \in (0,\psi (1-\tilde{q}^*))$, it holds that $\mathbb{P}(\mathcal{A}^{(n)}(\epsilon)|\mathcal{M}^{(n)}) \to 1$ and there exists $c_1>0$, such that 
$$\mathbb{P}(|S^{(n)}(t_1^{(n)}(\epsilon))|-|S^{(n)}(\infty)| > c_1 n|\mathcal{M}^{(n)}) \to 1.$$ 
\end{lemma}
\noindent The proof of this lemma is provided in Appendix \ref{Appendixsec}.

Now  we are almost ready to prove Lemma \ref{lemint3}.
In the proof we consider who infected whom, and since individuals can be infected only once, this leads to a tree representation of the infection process: the infection tree.
For $u,v \in V^{(n)}$, if $v$ is infected by $u$ then $u$ is the infector of $v$ and we write $u = \zeta(v)$.
We say that vertex $u$ is an ancestor of $v$  if there is $j \in \mathbb{N}$ and there are vertices $u = v_0, v_1, \cdots, v_j = v \in V^{(n)}$ such that for $i \in \mathbb{N}_{\leq j}$, $v_{i-1} = \zeta(v_i)$. To be complete we say that $v$ is an ancestor of itself.
 
Let $v$ be a vertex infected at time $\sigma(v)$. 
Then define 
$\{J_v^{(n)}(t);t \geq 0\}$, through $$J_v^{(n)}(t) = I^{(n)}(\sigma(v)+t) \cap \{u \in V^{(n)}; \mbox{$v$ is an ancestor of $u$} \}.$$ 

Furthermore, 
let $V_*^{(n)}(t)\subset V^{(n)}$, be the set of vertices infected after time $t$, of which the infector is infected before  time $t$, i.e.\
\begin{equation}
\label{Vstardef}
V_{*}^{(n)}(t) = \{v \in  S^{(n)}(t) \setminus  S^{(n)}(\infty);\zeta(v) \in I^{(n)}(t)\}.
\end{equation}
In the language of \cite{Bham14}, $V_{*}^{(n)}(t)$ is the coming generation at time $t$.

\subsubsection{Proof of Lemma \ref{lemint3}}

Lemma \ref{lemint3} follows trivially from the following three lemmas. In Lemma \ref{revlemup} we show that condition \eqref{THMass} is necessary to obtain the proper scaling for de duration of the end of the epidemic. If condition \eqref{THMass} holds, then we provide an upperbound for the duration of the end of the epidemic in Lemma \ref{endlemup} and we provide a lower bound in Lemma \ref{endlemdown}.
\begin{lemma}
\label{revlemup}
Assume that \eqref{THMass} does not hold, i.e.\ assume that there exist $\eta>0$ such that $\int_0^{\infty} e^{(|\alpha^*|-\eta)t}L(dt) = \infty$,
then there exists $\delta>0$, $\gamma>0$ and a strictly increasing infinite sequence of positive integers $n_1,n_2, \cdots$  
such that
$$\mathbb{P}\left(T^{*}(n_i)-T'_{\gamma}(n_i) > \frac{\log n_i}{|\alpha^*| - \delta}\mid \mathcal{M}^{(n_i)}\right) \to 1 \qquad \mbox{as $i \to \infty$.}$$
\end{lemma}

\begin{lemma}
\label{endlemup}
Assume that \eqref{THMass} holds.
For every $\delta \in (0,|\alpha^*|)$, there exist $\gamma>0$, such that 
$$\mathbb{P}\left(
T^{*}(n)-T'_{\gamma}(n) < \frac{\log n}{|\alpha^*| - \delta}\mid \mathcal{M}^{(n)}\right) \to 1.$$
\end{lemma}

\begin{lemma}
\label{endlemdown}
For every $\delta>0$, there exist $\gamma>0$, such that $$\mathbb{P}\left(T^{*}(n)-T'_{\gamma}(n) > \frac{\log n}{|\alpha^*|+\delta}\mid \mathcal{M}^{(n)}\right) \to 1.$$
\end{lemma}

\begin{proof}[Proof of Lemma \ref{revlemup}]
By Remark \ref{REMass} we know that  for  $\eta \in (0,|\alpha^*|)$,
\begin{equation}
\label{equivalence2}
\int_0^{\infty} e^{(|\alpha^*|-\eta)t}L(dt) = \infty \Leftrightarrow \int_0^{\infty} e^{(|\alpha^*|-\eta) t} \mathbb{P}(L>t) dt = \infty.
\end{equation}
Working with the right hand side, we continue by observing 
\begin{equation*}
\begin{split}
\ & \int_0^{\infty} e^{(|\alpha^*|-\eta) t} \mathbb{P}(L>t) dt\\
= & \sum_{n=1}^{\infty} \int_{\frac{\log n}{|\alpha^*|-\eta/2}}^{\frac{\log(n+1)}{|\alpha^*|-\eta/2}} e^{(|\alpha^*|-\eta) t} \mathbb{P}(L>t) dt\\
\leq & \sum_{n=1}^{\infty} \frac{\log(1+1/n)}{|\alpha^*|-\eta/2} e^{\frac{|\alpha^*|-\eta}{|\alpha^*|-\eta/2} \log (n+1)} \mathbb{P}\left(L>\frac{\log n}{|\alpha^*|-\eta/2}\right)\\
\leq & \sum_{n=1}^{\infty} \frac{\log 2}{|\alpha^*|-\eta/2}  n^{\frac{|\alpha^*|-\eta}{|\alpha^*|-\eta/2}\frac{\log(n+1)}{\log n}} \mathbb{P}\left(L>\frac{\log n}{|\alpha^*|-\eta/2}\right).
\end{split}
\end{equation*}
Furthermore, there exists $\epsilon_1>0$, which depends on $\eta$ such that for all large enough $n$
$$\frac{|\alpha^*|-\eta}{|\alpha^*|-\eta/2} \frac{\log(n+1)}{\log n}<1-2 \epsilon_1.$$
So, there exists $\epsilon_1>0$ such that,
\begin{equation}
\label{intimpl}
\int_0^{\infty} e^{(|\alpha^*|-\eta) t} \mathbb{P}(L>t) dt =\infty \Rightarrow 
\sum_{n=1}^{\infty}  n^{1-2\epsilon_1} \mathbb{P}\left(L>\frac{\log n}{|\alpha^*|-\eta/2}\right)= \infty.
\end{equation}

Now assume that $\int_0^{\infty} e^{(|\alpha^*|-\eta) t} \mathbb{P}(L>t) dt =\infty$ and that there exists $N_0 \in \mathbb{N}$ such that $\mathbb{P}\left(L>\frac{\log n}{|\alpha^*|-\eta/2}\right) \leq n^{-(1-\epsilon_2)}$ for all $\epsilon_2>0$ and all $n \in \mathbb{N}_{> N_0}$.
This implies with $\epsilon_2= \epsilon_1$ that
\begin{multline*}
\sum_{n=1}^{\infty} n^{1-2 \epsilon_1} \mathbb{P}\left(L>\frac{\log n}{|\alpha^*|-\eta/2}\right)\\
\leq \sum_{n=1}^{N_0}  n^{1-2 \epsilon_1} +
\sum_{n=N_0+1}^{\infty}  n^{1-2\epsilon_1}n^{-(1-\epsilon_1)} 
\leq   (N_0)^{1-2 \epsilon_1} +
\sum_{n=N_0+1}^{\infty}  n^{-\epsilon_1},
\end{multline*}
which is finite and therefore contradicts   $\int_0^{\infty} e^{(|\alpha^*|-\eta) t} \mathbb{P}(L>t) dt =\infty$.

So, there exists  $\epsilon_1>0$ such that  $\mathbb{P}\left(L>\frac{\log n}{|\alpha^*|-\eta/2}\right) >n^{-(1-\epsilon_1)}$, for infinitely many $n$.
Say that 
for the infinite increasing sequence  $n_1, n_2, \cdots$ we have 
\begin{equation}
\label{nilarge}
\mathbb{P}\left(L>\frac{\log n_i}{|\alpha^*|-\eta/2}\right) >(n_i)^{-(1-\epsilon_1)} \qquad \mbox{for $i \in \mathbb{N}$.} 
\end{equation}
We choose $\gamma>0$ and $c_1>0$ such that 
\begin{equation}
\label{niconv}
\mathbb{P}\left(|S^{(n_i)}(T'_{\gamma}(n_i))\setminus S^{(n_i)}(\infty)|>c_1 n_i\mid \mathcal{M}^{(n_i)}\right) \to 1 \qquad \mbox{ as $i \to \infty$,} 
\end{equation}
which we can do by Lemma \ref{eeslemma2}.
If the maximal infectious period of vertices in the set $|S^{(n_i)}(T'_{\gamma}(n_i))\setminus S^{(n_i)}(\infty)|$ is  larger than $\frac{\log n_i}{|\alpha^*|-\eta/2}$, then 
$$\left|I^{(n_i)}\left(T'_{\gamma}(n_i) + \frac{\log n_i}{|\alpha^*|-\eta/2}\right)\right|>0.$$
The probability that the maximum infectious period of vertices in the set  $|S^{(n_i)}(T'_{\gamma}(n_i))\setminus S^{(n_i)}(\infty)|$ is  larger than $\frac{\log n_i}{|\alpha^*|-\eta/2}$ is given by 
$$
1-\left(1-\mathbb{P}\left(L>\frac{\log n_i}{|\alpha^*|-\eta/2}\right)\right)^{|S^{(n_i)}(T'_{\gamma}(n_i))\setminus S^{(n_i)}(\infty)|}.$$
By \eqref{nilarge} this probability is larger than 
$$
1-\left(1-\frac{c_1(n_i)^{\epsilon_1}}{c_1 n_i} \right)^{|S^{(n_i)}(T'_{\gamma}(n_i))\setminus S^{(n_i)}(\infty)|}.
$$
So combining this with \eqref{niconv} we obtain
\begin{equation*}
\begin{split}
\ & \  \mathbb{P}\left(T^{*}(n_i)-T'_{\gamma}(n_i) > \frac{\log n_i}{|\alpha^*| - \eta/2}\mid \mathcal{M}^{(n)}\right)\\
= & \  
\mathbb{P}\left(\left|I^{(n_i)}\left(T'_{\gamma}(n_i) + \frac{\log n_i}{|\alpha^*|-\eta/2}\right)\right|>0\mid  \mathcal{M}^{(n_i)}\right) \\
\geq & \ 
1-\mathbb{P}(|S^{(n_i)}(T'_{\gamma}(n_i))\setminus S^{(n_i)}(\infty)| \leq c_1 n_i\mid \mathcal{M}^{(n_i)})\\
\ & \ - \mathbb{P}\left(\left|I^{(n_i)}\left(T'_{\gamma}(n_i) + \frac{\log n_i}{|\alpha^*|-\eta/2}\right)\right|>0 \bigcap \left|S^{(n_i)}(T'_{\gamma}(n_i))\setminus S^{(n_i)}(\infty)\right|>c_1 n_i\mid \mathcal{M}^{(n_i)}\right)\\
\geq & \ 
1-\mathbb{P}(|S^{(n_i)}(T'_{\gamma}(n_i))\setminus S^{(n_i)}(\infty)| \leq c_1 n_i\mid \mathcal{M}^{(n_i)})
-\left(1-\frac{c_1(n_i)^{\epsilon_1}}{c_1 n_i} \right)^{c_1 n_i},
\end{split}
\end{equation*}
which converges to 1 as $i \to \infty$.
\end{proof}

\begin{proof}[Proof of Lemma \ref{endlemup}]
We divide the proof in the following steps
\begin{enumerate}
\item Show that there exists with high probability a constant $\gamma>0$ such that for  $v \in V_*^{(n)}(T'_{\gamma}(n))$ and for $\delta \in (0,|\alpha^*|)$, we can construct a  branching process which dominates $\{J_v^{(n)}(t);t \geq 0\}$ and has Malthusian parameter in the interval $(-|\alpha^*|,-(|\alpha^*|- \delta))$.
 \item Show that the dominating branching process satisfy the conditions of Lemma \ref{sublem}.
 \item Show that 
 \begin{multline*}
 \sum_{v \in V_*^{(n)}(T'_{\gamma}(n))}\left|J_v^{(n)}\left(T'_{\gamma}(n) + \frac{\log n}{|\alpha^*| - \delta}-\sigma(v)\right)\right|\\
 + \left|I^{(n)}\left(T'_{\gamma}(n))\cap I^{(n)}(T'_{\gamma}(n) + \frac{\log n}{|\alpha^*| - \delta}\right)\right|=0 \qquad \mbox{w.h.p.}
 \end{multline*}
\end{enumerate}
Again we use $Q=1-\psi+\psi \tilde{q}^*$.\\

\noindent\textbf{Step 1:}

\noindent Let $\epsilon >0$ satisfy $\epsilon<Q$ and $Q+2 \epsilon<1$.
If at time $t$ a half-edge from $\mathcal{E}^{(n)}_I(t-)$ is paired with another half-edge, this other half-edge belongs to $\mathcal{E}^{(n)}_S(t-)$ with probability $\kappa^{(n)}(t)$, which is defined by
\begin{equation}
\label{kappat}
\kappa^{(n)}(t) = \frac{|\mathcal{E}^{(n)}_S(t-)|}{\ell(n)-|\mathcal{E}^{(n)}_P(t-)|-1}.
\end{equation}
Here  $\ell(n)-|\mathcal{E}^{(n)}_P(t-)|$ is the number of not-yet paired vertices just before time $t$ and the $-1$ appear in the denominator because  the half-edge from $\mathcal{E}^{(n)}_I(t-)$ cannot be paired with itself.  
Furthermore, the probability that if the half-edge is paired with a half-edge from $\mathcal{E}^{(n)}_S(t-)$, it belongs to a vertex of degree at least $k$ is given by $\pi_{\geq k}^{(n)}(t)$, which is defined by
\begin{equation}
\label{piknt}
\pi_{\geq k}^{(n)}(t) = \frac{\sum_{v\in S^{(n)}(t-)} d_v \ind(d_v \geq k)}{|\mathcal{E}^{(n)}_S(t-)|}. 
\end{equation}
The processes $\{|S^{(n)}(t)|; t \geq 0\}$ and $\{|\mathcal{E}^{(n)}_S(t)|;t \geq 0\}$ are decreasing in $t$, while $\{|\mathcal{E}^{(n)}_P(t)|;t \geq 0\}$ is increasing in $t$. 
So, for $t_1^{(n)}(\epsilon)$ as in Lemma \ref{eeslemma2}, and $t>t_1^{(n)}(\epsilon)$ and on $\mathcal{A}^{(n)}(\epsilon)$,
\begin{equation}
\label{gammaupbound1}
\kappa^{(n)}(t) \leq \frac{|\mathcal{E}^{(n)}_S(t_1^{(n)}(\epsilon))|}{\ell(n)-|\mathcal{E}^{(n)}_P(\infty)|-1}
< \frac{\mathbb{E}[(Q+\epsilon)^{\tilde{D}}] \ell(n)}{(Q- \epsilon)^2 \ell(n)} = \frac{\mathbb{E}[(Q+\epsilon)^{\tilde{D}}]}{(Q- \epsilon)^2 }.
\end{equation}
For future reference define
\begin{equation}
\label{kappa+}
\kappa_+(\epsilon) = \frac{\mathbb{E}[(Q+\epsilon)^{\tilde{D}}]}{(Q- \epsilon)^2 }
\end{equation}
Note that the second inequality is strict.
Similarly, on $\mathcal{A}^{(n)}(\epsilon)$ and for $k \geq 1$ and $t \geq t_1^{(n)}(\epsilon)$ we have by Lemma \ref{eeslemma2} that
\begin{multline}
\label{uppboundD1}
\pi_{\geq k}^{(n)}(t) = \frac{\sum_{v\in S^{(n)}(t-)} d_v \ind(d_v \geq k)}{|\mathcal{E}^{(n)}_S(t-)|} \\
\leq \min\left(1, \frac{\sum_{v\in S^{(n)}(t_1^{(n)}(\epsilon))} d_v \ind(d_v \geq k)}{|\mathcal{E}^{(n)}_S(\infty)|}\right)\\
 \leq \min\left(1, \frac{\sum_{j=k}^{\infty} n j \hat{d}^{(n)}_j (Q + \epsilon)^j}{\ell(n)\mathbb{E}[(Q-\epsilon)^{\tilde{D}}]}\right) 
 =
\mathbb{P}(\tilde{D}^{(n),+}(\epsilon) \geq k),
\end{multline}
where the final equality is the defines the random variable $\tilde{D}^{(n),+}(\epsilon)$.
That is, $\tilde{D}^{(n),+}(\epsilon)$ stochastically dominates the random variable defined through $\pi_{\geq k}^{(n)}(t)$ for $t>t^{(n)}_1(\epsilon)$. It is important to note that for $\epsilon$ as above and $n'>n$, 
 $\tilde{D}^{(n),+}(\epsilon)$ stochastically dominates  $\tilde{D}^{(n'),+}(\epsilon)$.

Recall the notation from Section \ref{sec:construct}
Let $v$ be a vertex infected at time $t$. Then $v$ has a random degree with distribution defined through $\pi_{\geq k}^{(n)}(t)$. 
One of the $d_v$ half-edges attached to $v$ is paired at time $t$, while the other $d_v-1$ are still unpaired at time $t$. Let $L_v$ be the infectious period of $v$ and without loss of generality we can assume that $v$ was infected through half-edge $(v,d_v)$. So, $\tau_{v,1}, \tau_{v,2}, \cdots \tau_{v,d_v-1}$ are the independent exponentially distributed random variables with expectation $1/\beta$ assigned to the different unpaired half-edges of $v$.  For $i \leq d_v-1$, if $\tau_{v,i}  \leq L_v$, then $t+\tau_{v,i}$ is the time at which a contact is made along the half-edge (and the half-edge is paired) and the contact made at time $t+\tau_{v,i}$ is with a susceptible with probability $\kappa^{(n)}(t+\tau_{v,i}-)$. 
By \eqref{gammaupbound1}, \eqref{kappa+} and \eqref{uppboundD1} 
we thus obtain that 
 for all $v \in V_*^{(n)}(t^{(n)}_1(\epsilon))$,
$\{|J_v^{(n)}(t)|;t \geq 0\}$ is stochastically dominated by a branching process in which particles give birth at ages given by the point process 
\begin{equation}
\label{pointproc+1}
\left\{\sum_{k=1}^{\tilde{D}^{(n),+}(\epsilon)-1} \ind(\tau_{k}<\min(L,t)) Y_k^+(\epsilon)  ;t \geq 0\right\},
\end{equation}
where $Y_k^+(\epsilon)$ is a Bernoulli random variable with success probability $\kappa_+(\epsilon)$, $\tau_1, \tau_2, \cdots$ are exponential distributed random variables with expectation $1/\beta$ and all defined random variables are independent.

To make step 2 below work, we define another process which dominates the above point process by
\begin{multline}
\label{hatski}
\{\hat{\xi}^{(n),+}_{\epsilon}(t);t \geq 0\}\\
= \left\{\ind(g^*(\alpha^*)< 1) \bar{Y} +\sum_{k=1}^{\tilde{D}^{(n),+}(\epsilon)-1} \ind(\tau_{v,k}<\min(L,t)) Y_k^+(\epsilon)   ;t \geq 0\right\}.
\end{multline}
Here $g^*(\cdot)$ is defined as in \eqref{gstardef} and $\bar{Y}$ is a Bernoulli random variable with success probability $1-g^*(\alpha^*)$, which is independent of everything else in the process. 
So, if $g^*(\alpha^*) \geq 1$, $\{\hat{\xi}^{(n),+}_{\epsilon}(t);t \geq 0\}$ is given by \eqref{pointproc+1}. While if 
$g^*(\alpha^*)< 1$, we obtain $\{\hat{\xi}^{(n),+}_{\epsilon}(t);t \geq 0\}$ by adding with probability $1-g^*(\alpha^*)$ a point at time $0$  to \eqref{pointproc+1}. 
Note that if we create a branching process by adding this possible extra particle to the branching process with mean offspring measure $\mu^*(\cdot)$, for which $g^*(\alpha^*)<1$ then for the new branching process the mean offspring measure, $\bar{\mu}^*(\cdot)$ say, satisfies
$$
\int_0^{\infty} e^{-\alpha^* t} \bar{\mu}^*(dt) = 
\int_0^{\infty} e^{- \alpha^*t} \mu^*(dt) + 
1-g^*(\alpha^*) = g^*(\alpha^*) + 1- g^*(\alpha^*) = 1.$$

The mean offspring measure of the branching process with reproduction process  $\{\hat{\xi}^{(n),+}_{\epsilon}(t);t \geq 0\}$ is then defined through
$$
\mu^{(n),+}_{\epsilon}(t)=\\ \mathbb{E}[\tilde{D}^{(n),+}(\epsilon)-1]\kappa_+(\epsilon) \mathbb{P}(\tau_{v,k}<\min(L,t)) + \max(0,1-g^*(\alpha^*)),
$$
where $\kappa_+(\epsilon)$ is defined in \eqref{kappa+} and
\begin{equation*}\label{expupbp1}
\mathbb{E}[\tilde{D}^{(n),+}(\epsilon)] \leq 
\sum_{k=1}^{\infty} 
 \frac{\sum_{j=k}^{\infty} n j \hat{d}^{(n)}_j (Q + \epsilon)^j}{\ell(n) \mathbb{E}[(Q-\epsilon)^{\tilde{D}}]}
=   \frac{\sum_{j=1}^{\infty}  j^2 \hat{d}^{(n)}_j (Q + \epsilon)^j}{n^{-1}\ell(n) \mathbb{E}[(Q-\epsilon)^{\tilde{D}}]}.
\end{equation*}
Since $1 \geq \hat{d}_j^{(n)} \to d_j$ for all $j$ and is decreasing in $n$ and because the sum $\sum_{j=1}^{\infty} j^2 (Q + \epsilon)^j<\infty$, the numerator decreases to $$\mathbb{E}[D^2(Q+\epsilon)^D] = \mathbb{E}[D]\mathbb{E}[\tilde{D}(Q+\epsilon)^{\tilde{D}}],$$ while the denominator converges to $\mathbb{E}[D]\mathbb{E}[(Q-\epsilon)^{\tilde{D}}]$.
So, $$n_1(\epsilon) = \min\left\{n \in \mathbb{N};  \mathbb{E}[\tilde{D}^{(n),+}(\epsilon)] < \frac{\mathbb{E}[\tilde{D}(Q+2 \epsilon)^{\tilde{D}}]}{\mathbb{E}[(Q-\epsilon)^{\tilde{D}}]}\right\}$$
is well defined and finite. Denote the branching process with reproduction process $\{\hat{\xi}^{(n_1(\epsilon)),+}_{\epsilon}(t);t \geq 0\}$ by $\{\hat{\xi}^{+}_{\epsilon}(t);t \geq 0\}$ and the corresponding offspring mean with $\mu^{+}_{\epsilon}(t)$.
Finally, 
\begin{equation}
\label{XLT1}
\mathbb{P}(\tau_{v,k}<\min(L,t)) = \int_0^t (1-e^{-\beta s})L(ds)= \int_0^t \beta e^{-\beta s} \mathbb{P}(L>s)ds.
\end{equation}
Combining the above terms we obtain that $$\mu_{\epsilon}^{+}(0) = \ind(g^*(\alpha^*)<1) (1-g^*(\alpha^*))$$
and for $t>0$ and $n>n_1$ that
\begin{multline}
\label{mueps+1}
\mu^{+}_{\epsilon}(dt)
\leq   \frac{\mathbb{E}[(\tilde{D}-1)(Q+\epsilon)^{\tilde{D}}]}{\mathbb{E}[(Q-\epsilon)^{\tilde{D}}]} \frac{\mathbb{E}[(Q+2 \epsilon)^{\tilde{D}}]}{(Q- \epsilon)^2} \beta e^{-\beta t} \mathbb{P}(L>t)dt\\
= K_+(\epsilon) \mu^*(dt),\end{multline}
where  $\mu^*(dt)$ is defined in \eqref{endmalt2} and
\begin{equation}
\label{Kplus1}
K_+(\epsilon) = 
\frac{\mathbb{E}[(\tilde{D}-1)(Q+\epsilon)^{\tilde{D}}]}{\mathbb{E}[(\tilde{D}-1)Q^{\tilde{D}}]} \frac{\mathbb{E}[(Q+2\epsilon)^{\tilde{D}}]}{\mathbb{E}[(Q-\epsilon)^{\tilde{D}}]}\frac{Q^2}{(Q-\epsilon)^2} >1.
\end{equation}

So, for $t \in (0,\infty)$, we have
$\mu^*(dt) \leq \mu^{+}_{\epsilon}(dt) \leq K_+(\epsilon) \mu^*(dt),$
where the first inequality is strict for $t$ such that $\mathbb{P}(L>t)>0$.
Therefore, for every $x>\alpha^*$ we have,
\begin{multline}
\label{measbound}
g^*(x) + \max(0,1-g^*(\alpha^*)) = \int_0^{\infty} e^{-x t}\mu^*(dt) +  \max(0,1-g^*(\alpha^*))\\ < 
  \int_0^{\infty} e^{-x t} \mu^{+}_{\epsilon}(dt)
\leq K_+(\epsilon) \int_0^{\infty} e^{-x t}\mu^*(dt) +   \max(0,1-g^*(\alpha^*))\\
= K_+(\epsilon) g^*(x) + \max(0,1-g^*(\alpha^*))
.
\end{multline}
Note that  that $g^*(\alpha^*) + \max(0,1-g^*(\alpha^*))= \max(1,g^*(\alpha^*)=1$.
Then note that all expectations in $K_+(\epsilon)$ as defined in \eqref{Kplus1} are finite by Claim \ref{dalem} and continuous in $\epsilon$. Furthermore, $K_+(\epsilon)$ is clearly increasing in $\epsilon$. By the finiteness of all expectations in   \eqref{Kplus1}, we also obtain that $\lim_{\epsilon \searrow 0} K_+(\epsilon)=1$. 

Together this gives that 
for every $\delta' >0$, we can choose $\epsilon>0$ such that, 
\begin{equation}
\label{appreps1}
\begin{split}
\ & \int_0^{\infty} e^{-(\alpha^* + \delta') t} \mu^{+}_{\epsilon}(dt)\\  \leq & \ 
K_+(\epsilon) \int_0^{\infty} e^{-(\alpha^* + \delta') t}\mu_*(dt) + \max(0,1-g^*(\alpha^*)) <1
\end{split}
\end{equation}
and 
\begin{equation*}
\lim_{x \searrow \alpha^*} \int_0^{\infty} e^{-x t} \mu^{+}_{\epsilon}(dt) >1.
\end{equation*}
It is easy to check that $\int_0^{\infty} e^{-x t} \mu^{+}_{\epsilon}(dt)$ is continuous and monotonous in $x$ on $(\alpha^{\dagger}, \infty)$. Therefore, for $\epsilon$ such that \eqref{appreps1} is satisfied, there exists $\alpha^+_{\epsilon} \in (\alpha^*,\alpha^* + \delta')$ such that 
$\int_0^{\infty} e^{-\alpha^+_{\epsilon} t} \mu^+_{\epsilon}(dt) =1$.\\

\noindent\textbf{Step 2:}

\noindent In this step, let $\mu^+_{\epsilon}(dt)$ be as in Step 1. We wish to show that for every $\delta \in (0,|\alpha^*|)$ there exists $\epsilon>0$  and $\alpha^+_{\epsilon} \in (\alpha^*,\alpha^*+\delta)$ such that
\begin{enumerate}[label=(\roman*)]
\item  $1 = \int_0^{\infty} e^{-\alpha^+_{\epsilon}t}  \mu^+_{\epsilon}(dt)$, 
\item $\int_0^{\infty} t e^{|\alpha^+_{\epsilon}| t} L(dt)<\infty$, 
\item $\int_0^{\infty} t e^{|\alpha^+_{\epsilon}| t} \mu^+_{\epsilon}(dt)<\infty$ and 
\item $\mathbb{E}\left[\int_0^{\infty}e^{|\alpha^+_{\epsilon}| t}\hat{\xi}_{\epsilon}^+(dt) \log^+(\hat{\xi}_{\epsilon}^+(\infty))\right]< \infty$.
\end{enumerate}

From the last paragraph of Step 1, (i) follows immediately. 
Furthermore, by assumption we have
$\int_0^{\infty} e^{-x t} L(dt)<\infty$ for all $x > \alpha^*$. This implies that 
$\int_0^{\infty} t e^{-x t} L(dt)<\infty$ for all $x > \alpha^*$. (ii) follows now from $\alpha^+_{\epsilon}>\alpha^*$.

Because $\alpha^+_{\epsilon} > \alpha^* \geq \alpha^{\dagger}$ and the second inequality in \eqref{measbound} we obtain  by the definition of $\alpha^{\dagger}$ that $\int_0^{\infty} e^{-\alpha^+_{\epsilon} t}\mu_*(dt) 
 < \infty$ and thus that
$$\int_0^{\infty} t e^{- \alpha^+_{\epsilon} t} \mu^+_{\epsilon}(dt) \leq  K_+(\epsilon) \int_0^{\infty} e^{-\alpha^+_{\epsilon} t}\mu_*(dt) 
+ \ind(g^*(\alpha^*)<1)(1-g^*(\alpha^*))<\infty.$$

Finally, 
\begin{equation*}
\begin{split}
\ & \ \mathbb{E}\left[\int_0^{\infty}e^{|\alpha^+_{\epsilon}| t}\hat{\xi}_{\epsilon}^+(dt) \log^+(\hat{\xi}_{\epsilon}^+(\infty))\right]\\
\leq & \ \mathbb{E}\left[\int_0^{\infty}e^{|\alpha^+_{\epsilon}| t}\hat{\xi}_{\epsilon}^+(dt) \log^+(\tilde{D}^+(\epsilon)-1))\right]\\
\leq & \ \mathbb{E}[(\tilde{D}^{+}(\epsilon)-1) \log^+(\tilde{D}^+(\epsilon)-1)] \kappa_+(\epsilon)
\int_0^{\infty}e^{|\alpha^+_{\epsilon}| t} \beta e^{-\beta t} P(L>t) dt\\
\leq & \ \mathbb{E}[(\tilde{D}^{+}(\epsilon)-1)^2] \kappa_+(\epsilon)
\int_0^{\infty}e^{|\alpha^-_{\epsilon}| t} \beta e^{-\beta t} P(L>t) dt\\
= & \ \frac{\mathbb{E}[(\tilde{D}^{+}(\epsilon)-1)^2]}{\mathbb{E}[\tilde{D}^{+}(\epsilon)-1]} \kappa_+(\epsilon)
\int_0^{\infty}e^{|\alpha^-_{\epsilon}| t} \mu^-_{\epsilon}(dt).
\end{split}
\end{equation*}
It follows from Claim \ref{dalem} that the quotient of the expectations is  finite, while the integral is finite by step (iii). So condition (iv) is met.\\

\noindent\textbf{Step 3:}

\noindent
For $v \in S^{(n)}(0)\setminus  S^{(n)}(\infty)$ 
recall from \eqref{Vstardef} that $V_*^{(n)}(t)$ is the set of vertices infected after time $t$, by an infector infected before or at time $t$.

For $\delta \in (0,|\alpha^*|)$, we are interested in 
$$\mathbb{P}\left(I^{(n)}\left(T'_{\gamma}(n) + \frac{\log n}{|\alpha^*|-\delta}\right) \neq \emptyset \mid \mathcal{M}^{(n)}\right).$$
Observe
\begin{equation*}
\begin{split}
\ &\left\{I^{(n)} \left(T'_{\gamma}(n) + \frac{\log n}{|\alpha^*|-\delta}
\right) \neq \emptyset\right\} \\ 
= &\left\{ I^{(n)} \left(T'_{\gamma}(n) + \frac{\log n}{|\alpha^*|-\delta} \right) \cap
I^{(n)} \left(T'_{\gamma}(n) \right) \neq \emptyset \right\}\\
\ &\bigcup \left(\bigcup_{v \in V_{*}^{(n)}(T'_{\gamma}(n))} \left\{ J_v^{(n)} \left( \frac{\log n}{|\alpha^*|-\delta}-(\sigma(v)-T'_{\gamma}(n))\right) \neq \emptyset \right\}\right). 
\end{split}
\end{equation*}
So,
\begin{equation*}
\begin{split}
\ & \mathbb{P}\left(I^{(n)}\left(T'_{\gamma}(n) + \frac{\log n}{|\alpha^*|-\delta}\right) \neq \emptyset \mid \mathcal{M}^{(n)}\right)\\
\leq &
\mathbb{P}\left(I^{(n)} \left(T'_{\gamma}(n) + \frac{\log n}{|\alpha^*|-\delta} \right) \cap 
I^{(n)} \left(T'_{\gamma}(n) \right) \neq \emptyset \mid \mathcal{M}^{(n)} \right)\\
\ & + \sum_{v \in V_{*}^{(n)}(T'_{\gamma}(n))}
\mathbb{P}\left( J_v^{(n)} \left( \frac{\log n}{|\alpha^*|-\delta}-(\sigma(v)-T'_{\gamma}(n))\right) \neq \emptyset  \mid \mathcal{M}^{(n)}\right). 
\end{split}
\end{equation*}
We treat the two terms on the right hand side separately. Observe that
\begin{multline*}
\mathbb{P}\left(I^{(n)} \left(T'_{\gamma}(n) + \frac{\log n}{|\alpha^*|-\delta} \right) \cap 
I^{(n)} \left(T'_{\gamma}(n) \right) \neq \emptyset \mid \mathcal{M}^{(n)} \right)\\
 \leq 
\mathbb{P}\left( \bigcap_{v \in V^{(n)}} \left\{L_v > \frac{\log n}{|\alpha^*|-\delta}\right\} \right)
\leq
\sum_{v \in V^{(n)}}  \mathbb{P}\left(L_v > \frac{\log n}{|\alpha^*|-\delta} \right)\\= n  \mathbb{P}\left(L > \frac{\log n}{|\alpha^*|-\delta} \right). 
\end{multline*}
Assume that $\liminf_{n \to \infty} n  \mathbb{P}\left(L > \frac{\log n}{|\alpha^*|-\delta} \right) >0$. Then there exists an sequence of integers $1= n_0, n_1, n_2, \cdots$ and a constant $c_2>0$ such that
$n_i  \mathbb{P}\left(L > \frac{\log n_i}{|\alpha^*|-\delta} \right) >c_2$ and $n_{i+1}/n_i>e^{|\alpha^*|-\delta}$ for all $i \in \mathbb{N}_0$.
This implies that 
\begin{equation*}
\begin{split}
\ & \ \int_0^{\infty} e^{(|\alpha^*|-\eta)t} \mathbb{P}(L>t) dt\\
= & \ \sum_{i=1}^{\infty} \left(\int_{\frac{\log n_i}{\alpha^*-\eta}-1}^{\frac{\log n_i}{\alpha^*-\eta}} e^{(|\alpha^*|-\eta)t} \mathbb{P}(L>t) dt 
+  \int_{\frac{\log n_i}{\alpha^*-\eta}}^{\frac{\log n_{i+1}}{\alpha^*-\eta}-1} e^{(|\alpha^*|-\eta)t} \mathbb{P}(L>t) dt\right)\\
\geq & \ \sum_{i=1}^{\infty} \int_{\frac{\log n_i}{\alpha^*-\eta}-1}^{\frac{\log n_i}{\alpha^*-\eta}} e^{(|\alpha^*|-\eta)t} \mathbb{P}(L>t) dt \\
\geq & \  \sum_{i=1}^{\infty}  e^{-(|\alpha^*|-\eta)}n_i \mathbb{P}\left(L>\frac{\log n_i}{|\alpha^*|-\eta}\right) dt\\
\geq & \ \sum_{i=1}^{\infty}  e^{-(|\alpha^*|-\eta)} c_2, 
\end{split}
\end{equation*}
which is infinite. 
But by Assumption \eqref{THMass} and by  \eqref{equivalence} we have   
$$\int_0^{\infty} e^{(|\alpha^*|-\eta) t} \mathbb{P}(L>t) dt<\infty \qquad \mbox{for $\eta \in (0,|\alpha^*|)$.}$$ So, we arrive at a contradiction and
therefore  $n  \mathbb{P}\left(L > \frac{\log n}{|\alpha^*|-\delta} \right) \to 0$ and thus 
$$\mathbb{P}\left(I^{(n)} \left(T'_{\gamma}(n) + \frac{\log n}{|\alpha^*|-\delta} \right) \cap
I^{(n)} \left(T'_{\gamma}(n) \right) \neq \emptyset \mid \mathcal{M}^{(n)} \right) \to 0.$$

Now consider, 
$$ \sum_{v \in V_{*}^{(n)}(T'_{\gamma}(n))}
\mathbb{P}\left( J_v^{(n)} \left( \frac{\log n}{|\alpha^*|-\delta}-(\sigma(v)-T'_{\gamma}(n))\right) \neq \emptyset  \mid \mathcal{M}^{(n)}\right).$$
Recall that $\zeta(v)$ is the vertex through which $v$ is infected.
By definition $\sigma(v)-\sigma(\zeta(v)) \geq \sigma(v) - T'_{\gamma}(n)>0$ for $v \in V_{*}^{(n)}(T'_{\gamma}(n))$ and $|J_v^{(n)}(t)|=0 \Rightarrow |J_v^{(n)}(t')|=0$  for all $t'>t$. So,
\begin{equation*}
 \begin{split}
\ & \mathbb{P}\left(
\sum_{v \in V_{*}^{(n)}(T'_{\gamma}(n))} \left|J_v^{(n)} \left( \frac{\log n}{|\alpha^*|-\delta}-(\sigma(v)-T'_{\gamma}(n))\right)\right|=0
\mid \mathcal{M}^{(n)}\right)\\
\geq &
\mathbb{P}\left(
\sum_{v \in V_{*}^{(n)}(T'_{\gamma}(n))} \left|J_v^{(n)} \left( \frac{\log n}{|\alpha^*|-\delta}-(\sigma(v)-\sigma(\zeta(v))\right)\right|=0
\mid \mathcal{M}^{(n)}\right)
\end{split}
\end{equation*}
By steps 1 and 2 we know that we can choose $\epsilon>0$ such that 
$$\left\{\sum_{v \in V_{*}^{(n)}(T'_{\gamma}(n))} \left|J_v^{(n)} (t-(\sigma(v)-\sigma(\zeta(v)))\right|; t \geq T'_{\gamma}(n) \right\}$$ is dominated by 
$$\left\{\sum_{v \in V_{*}^{(n)}(T'_{\gamma}(n))} \left|Z_v (t-(\sigma(v)-\sigma(\zeta(v)))\right|; t \geq T'_{\gamma}(n) \right\},$$ 
where for $v\in V^{(n)}$, we let $\{Z_v(t);t\geq 0\}$ be independent branching processes all with Malthusian parameter $\alpha^+_{\epsilon}$, independent of $\{\sigma(v);v\in V^{(n)}\}$ and $\{\sigma(\zeta(v));v\in V^{(n)}\}$ and satisfying the conditions of Lemma \ref{sublem}.
Since $V_{*}^{(n)}(T'_{\gamma}(n)) \subset V^{(n)}$ we obtain that 
\begin{equation*}
\begin{split}
\ & \mathbb{P}\left(
\sum_{v \in V_{*}^{(n)}(T'_{\gamma}(n))} \left|J_v^{(n)} \left( \frac{\log n}{|\alpha^*|-\delta}-(\sigma(v)-\sigma(\zeta(v))\right)\right|=0
\mid \mathcal{M}^{(n)}\right)\\
\geq \ & 
\mathbb{P}\left(
\sum_{v \in V^{(n)}\setminus v_0} \left|Z_v\left( \frac{\log n}{|\alpha^*|-\delta}-(\sigma(v)-\sigma(\zeta(v))\right)\right|=0\right)\\
= & \prod_{v \in V^{(n)}\setminus v_0}\mathbb{P}\left(
 \left|Z_v\left( \frac{\log n}{|\alpha^*|-\delta}-(\sigma(v)-\sigma(\zeta(v))\right)\right|=0\right)\\
\geq & \  1 -\sum_{v \in V^{(n)}\setminus v_0}\mathbb{P}\left(
 \left|Z_v\left( \frac{\log n}{|\alpha^*|-\delta}-(\sigma(v)-\sigma(\zeta(v))\right)\right|>0\right).
\end{split}
\end{equation*}
Our next observation is that $$
\{\sigma(v)-\sigma(\zeta(v));v \in  V^{(n)}\setminus v_0\}
\subset
\{\tau_{v,v'} ;(v,v') \in \mathbf{s}^{(n)}, \tau_{v,v'} \leq L_v\}, 
$$ 
where $\mathbf{s}^{(n)}$ and $\{L_v; v\in V^{(n)}\}$ are defined in Section \ref{sec:construct}.
Therefore with $Z_{v,v'}$ i.i.d.\ copies of $Z_v$,
\begin{equation*}
 \begin{split}
\ & \sum_{v \in V^{(n)}\setminus v_0}\mathbb{P}\left(
 \left|Z_v\left( \frac{\log n}{|\alpha^*|-\delta}-(\sigma(v)-\sigma(\zeta(v))\right)\right|>0\right)\\
\leq & \sum_{v \in V^{(n)}} \sum_{v' =1}^{d_v}  \mathbb{P}\left( \ind(\tau_{v,v'} \leq L_v)
 \left|Z_{v,v'}\left( \frac{\log n}{|\alpha^*|-\delta}-\tau_{v,v'}\right)\right|>0\right)\\
= &  \sum_{v \in V^{(n)}} \int_0^{\frac{\log n}{|\alpha^*|-\delta}} \sum_{v' =1}^{d_v} \mathbb{P}\left( \ind(\tau_{v,v'} \leq t) \left|Z_{v,v'}\left( \frac{\log n}{|\alpha^*|-\delta}-\tau_{v,v'}\right)\right|>0\right) L_v(dt)\\
= & \sum_{v \in V^{(n)}}  d_v \int_0^{\frac{\log n}{|\alpha^*|-\delta}}\int_0^t \beta e^{-\beta s} \mathbb{P}\left(  \left|Z_{v,v'}\left( \frac{\log n}{|\alpha^*|-\delta}-s\right)\right|>0\right) ds L_v(dt)
\end{split}
\end{equation*}
By Lemma \ref{sublem} we know that there exists a constant $c_3>0$ such that $\mathbb{P}(Z_{v,v'}(t) >0) \leq c_3 e^{-|\alpha_{\epsilon}^+|t}$ for all $t >0$.
Therefore, 
\begin{equation*}
 \begin{split}
\ & \sum_{v \in V^{(n)}\setminus v_0}\mathbb{P}\left(
 \left|Z_v\left( \frac{\log n}{|\alpha^*|-\delta}-(\sigma(v)-\sigma(\zeta(v))\right)\right|>0\right)\\
\leq & \  c_3\sum_{v \in V^{(n)}}  d_v \int_0^{\frac{\log n}{|\alpha^*|-\delta}}\int_0^t \beta e^{-\beta s} 
 e^{-|\alpha_{\epsilon}^+| ( \frac{\log n}{|\alpha^*|-\delta}-s)}
  ds L_v(dt)\\
=  & \ c_3\sum_{v \in V^{(n)}}  d_v \int_0^{\frac{\log n}{|\alpha^*|-\delta}}\int_s^{\frac{\log n}{|\alpha^*|-\delta}}L_v(dt) \beta e^{-\beta s} 
 e^{-|\alpha_{\epsilon}^+| ( \frac{\log n}{|\alpha^*|-\delta}-s)} ds \\
\leq & \  c_3\sum_{v \in V^{(n)}}  d_v \int_0^{\frac{\log n}{|\alpha^*|-\delta}}\mathbb{P}(L>s) \beta e^{-\beta s} 
 e^{-|\alpha_{\epsilon}^+| ( \frac{\log n}{|\alpha^*|-\delta}-s)} ds\\
= & \  c_3\sum_{v \in V^{(n)}}  d_v \int_0^{\frac{\log n}{|\alpha^*|-\delta}}\mathbb{P}(L>s) \beta e^{-\beta s} e^{|\alpha_{\epsilon}^+| s}
 n^{-  \frac{|\alpha_{\epsilon}^+|}{|\alpha^*|-\delta}} ds
\end{split}
\end{equation*}
We can choose $\epsilon>0$, such that $\alpha^+_{\epsilon}$ exists and $|\alpha^+_{\epsilon}| \in (|\alpha^*|-\delta,|\alpha^*|)$.
So,  using \eqref{endmalt1}
the above expression is bounded from above by
\begin{multline}
\label{intermed}
c_3  n^{-  \frac{|\alpha_{\epsilon}^+|}{|\alpha^*|-\delta}} \sum_{v \in V^{(n)}}  d_v \int_0^{\infty}\mathbb{P}(L>s) \beta e^{-\beta s} e^{|\alpha^*| s}
 ds\\
\leq c_3  n^{1-  \frac{|\alpha_{\epsilon}^+|}{|\alpha^*|-\delta} }\frac{\ell(n)}{n} \frac{1}{\mathbb{E}\left[(\tilde{D}-1)Q^{\tilde{D}-2}\right]}
 \end{multline}
 And by $|\alpha^+_{\epsilon}|>|\alpha^*|-\delta$, we know that $n^{1-  \frac{|\alpha_{\epsilon}^+|}{|\alpha^*|-\delta} } \to 0$. Furthermore, by assumption A1, $\frac{\ell(n)}{n}$ converges to a finite limit and $\mathbb{E}\left[(\tilde{D}-1)Q^{\tilde{D}-2}\right]$ is necessarily positive. 
So, \eqref{intermed} converges to 0 as $n \to \infty$ and therefore 
$$ \sum_{v \in V_{*}^{(n)}(T'_{\gamma}(n))}
\mathbb{P}\left( J_v^{(n)} \left( \frac{\log n}{|\alpha^*|-\delta}-(\sigma(v)-T'_{\gamma}(n))\right) \neq \emptyset  \mid \mathcal{M}^{(n)}\right) \to 0.$$
\end{proof}

\begin{proof}[Proof of Lemma \ref{endlemdown}]
Throughout the proof we restrict ourselves to the event $\mathcal{M}^{(n)} \cap \mathcal{A}^{(n)}(\epsilon)$ (defined as in Lemma \ref{eeslemma2}) for some $\epsilon>0$ conveniently chosen. Because there exists $c_1>0$ such that $|S^{(n)}(t_1^{(n)}(\epsilon))|-|S^{(n)}(\infty)| > c_1 n$ w.h.p.,  we  immediately obtain that there exists $\gamma \in (0,1-q^*)$ such that $T'_{\gamma}(n) \in (t_1^{(n)}(\epsilon),T^*(n))$ w.h.p.
The proof consists of the following steps:

\begin{enumerate}
 \item Show that there exists with high probability a constant $\gamma>0$ such that for  $v \in V_*^{(n)}(T'_{\gamma}(n))$ and for $\delta \in (0,|\alpha^*|)$ small enough, we can construct a  branching process which is dominated by $\{J_v^{(n)}(t);t \geq 0\}$ and has Malthusian parameter larger than $-(|\alpha^*|+ \delta)$ (i.e.\ the absolute value of the Malthusian parameter is less than $|\alpha^*|+ \delta)$.
 \item Show that there exists $\gamma>0$ and $\delta>0$ such that the dominated branching process satisfy the conditions of Lemma \ref{sublem}.
 \item Show that there exist $c_1>0$ such that $$\mathbb{P}\left(n^{-1}|V_*^{(n)}(T'_{\gamma}(n))|>c_1\mid \mathcal{M}^{(n)}\right) \to 1.$$ 
 \item Show that for every $\delta \in (0,1)$, there exist $\gamma>0$, such that 
 $$\mathbb{P}\left(T^{*}(n)-T'_{\gamma}(n) > \frac{\log n}{|\alpha^*|+\delta}\mid\mathcal{M}^{(n)}\right) \to 1.$$ 
\end{enumerate}

\noindent{\textbf{Step 1:}}

\noindent Let $\epsilon >0$ be small and chosen appropriately later. Recall the definitions (\ref{kappat}) and (\ref{piknt}).
For $t_1^{(n)}(\epsilon)$ as in Lemma \ref{eeslemma2}, and $t>t_1^{(n)}(\epsilon)$ and on $\mathcal{A}^{(n)}(\epsilon)$, we have
\begin{equation}
\label{gammalowbound}
\kappa^{(n)}(t) \geq \frac{|\mathcal{E}^{(n)}_S(\infty)|}{\ell(n)-|\mathcal{E}^{(n)}_P(t_1^{(n)}(\epsilon))|-1}
\geq \frac{\mathbb{E}[(Q-\epsilon)^{\tilde{D}}] \ell(n)}{(Q+ \epsilon)^2 \ell(n)}
= \kappa_-(\epsilon),
\end{equation}
where the last equality is the definition of $\kappa_-(\epsilon)$.
Similarly, for $t>t_1^{(n)}(\epsilon)$, on $\mathcal{A}^{(n)}(\epsilon)$ and for $k \in \mathbb{N}_{\leq \lfloor 1/\epsilon \rfloor}$, 
\begin{multline}
\label{lowboundD}
\pi_{\geq k}^{(n)}(t) = \frac{\sum_{v\in S^{(n)}(t-)} d_v \ind(d_v \geq k)}{|\mathcal{E}^{(n)}_S(t-)|} \geq \frac{\sum_{v\in S^{(n)}(\infty)} d_v \ind(d_v \geq k)}{|\mathcal{E}^{(n)}_S(t_1^{(n)}(\epsilon))|}\\
\geq \frac{\mathbb{E}[(\ind(\tilde{D}\geq k)(Q - \epsilon)^{\tilde{D}}]\ell(n)}{\mathbb{E}[(Q+\epsilon)^{\tilde{D}}] \ell(n)} 
=
\mathbb{P}(\tilde{D}^{-}(\epsilon) \geq k),
\end{multline}
where again the last equality serves as a definition.
To be complete, let $\mathbb{P}(\tilde{D}^{-}(\epsilon) >  \lfloor 1/\epsilon \rfloor) =0$.
So, $\tilde{D}^{-}(\epsilon) \ind(\mathcal{A}^{(n)}(\epsilon))$ is stochastically dominated by the random variable defined through $\pi_{\geq k}^{(n)}(t)$ for $t>t^{(n)}_1(\epsilon)$.  

As in the proof of Lemma \ref{lemint2}, if
$v$ is a vertex infected at time $t$. Then $v$ has degree distribution defined through $\pi_{\geq k}^{(n)}(t)$. 
One of the $d_v$ half-edges attached to $v$ is paired at time $t$, while the other $d_v-1$ are still unpaired at time $t$. 
Again as in the proof of Lemma \ref{lemint2},
let $L_v$ be the infectious period of $v$ and let $\tau_{v,1}, \tau_{v,2}, \cdots \tau_{v,d_v-1}$ be independent exponentially distributed random variables with expectation $1/\beta$ assigned to the different unpaired half-edges of $v$. If $\tau_{v,i} \leq L_v$ then a contact made by $v$ at time $t+\tau_{v,i}$ is with a susceptible with probability $\kappa^{(n)}(t)$. 

Let $L'_{\epsilon} = \min(L,1/\epsilon)$, be a random variable representing a life length, which is distributed as $L$ cut off at length $1/\epsilon$. By \eqref{gammalowbound} and \eqref{lowboundD} 
we  obtain that for all $v \in V_*^{(n)}(t_1)$, the process
$\{|J_v^{(n)}(t)|;t \geq 0\}$ stochastically dominates a branching process in which particles give birth at ages given by the point process 
$$\{\hat{\xi}^-_{\epsilon}(t);t \geq 0\} 
= \left\{\sum_{k=1}^{\tilde{D}^{-}(\epsilon)-1} \ind(\tau_{k}<\min(L'_{\epsilon},t)) Y_k^-(\epsilon)  ;t \geq 0\right\},$$ 
where $Y_k^-(\epsilon)$ is a Bernoulli random variable with success probability $\kappa_-(\epsilon)$, $\tau_1,\tau_2, \cdots$ are exponentially distributed random variables with expectation $1/\beta$ and all defined random variables are independent.
 
The mean offspring measure of this branching process is  given by
$$\{\mu^-_{\epsilon}(t);t \geq 0\}= \mathbb{E}[\tilde{D}^{-}(\epsilon)-1]\kappa_-(\epsilon) \mathbb{P}(\tau_{k}<\min(L'_{\epsilon},t)),$$
where 
\begin{equation*}\label{gammalowex} 
\kappa_-(\epsilon) = \frac{\mathbb{E}[(Q-\epsilon)^{\tilde{D}}]}{(Q+ \epsilon)^2}
\end{equation*}
by \eqref{gammalowbound}. Further,
\begin{equation}\label{explowbp}
\mathbb{E}[\tilde{D}^{-}(\epsilon)] = \sum_{k=1}^{\lfloor 1/\epsilon \rfloor} \frac{\mathbb{E}[(\ind(\tilde{D}\geq k)(Q - \epsilon)^{\tilde{D}}]}{\mathbb{E}[(Q+\epsilon)^{\tilde{D}}]} 
= \frac{\mathbb{E}[\tilde{D}(Q - \epsilon)^{\tilde{D}}\ind(\tilde{D} \leq \lfloor 1/\epsilon \rfloor)]}{\mathbb{E}[(Q+\epsilon)^{\tilde{D}}]}.
\end{equation}
Finally, 
$$
\mathbb{P}(\tau_{k}<\min(L'_{\epsilon},t)) =
\int_0^{t}\beta e^{-\beta s} \mathbb{P}(L'>s)ds
=
\int_0^{\min(t,1/\epsilon)}\beta e^{-\beta s} \mathbb{P}(L>s)ds.
$$
So,
\begin{multline*}
\mu^-_{\epsilon}(dt)\\
= \left(\frac{\mathbb{E}[\tilde{D}(Q - \epsilon)^{\tilde{D}}\ind(\tilde{D} \leq \lfloor 1/\epsilon \rfloor)]}{\mathbb{E}[(Q+\epsilon)^{\tilde{D}}]}-1\right) \frac{\mathbb{E}[(Q-\epsilon)^{\tilde{D}}]}{(Q+ \epsilon)^2} \beta e^{-\beta t} \mathbb{P}(L>t)\ind(t<1/\epsilon)dt\\
= K_-(\epsilon) \ind(t<1/\epsilon) \mu^*(dt),
\end{multline*}
where 
\begin{equation}
\label{Kmin}
K_-(\epsilon) = 
\left(\frac{\mathbb{E}[\tilde{D}(Q - \epsilon)^{\tilde{D}}\ind(\tilde{D} \leq \lfloor 1/\epsilon \rfloor)]-\mathbb{E}[(Q+\epsilon)^{\tilde{D}}]}{\mathbb{E}[(\tilde{D}-1)Q^{\tilde{D}-2}]}\right) 
\frac{\mathbb{E}[(Q-\epsilon)^{\tilde{D}}]}{\mathbb{E}[(Q+\epsilon)^{\tilde{D}+2}]}
\end{equation}
and  $\mu^*(dt)$ is defined in \eqref{endmalt2}.

Because   there exists with high probability $\gamma \in (0,1-q^*)$ such that $T'_{\gamma}(n) \in (t_1^{(n)}(\epsilon),T^*(n))$, we obtain that with high probability and for all $v \in V_*^{(n)}(T'_{\gamma}(n))$, we have constructed branching process with reproduction process $\{\hat{\xi}^-_{\epsilon}(t);t \geq 0\}$ and mean offspring measure $\{\mu^-_{\epsilon}(t);t \geq 0\}$, which is stochastically dominated by the process $\{|J_v^{(n)}(t)|;t \geq 0\}$.

Note that all expectations in \eqref{Kmin} are finite, that $K_-(\epsilon) <1$ and that $K_-(\epsilon) \nearrow 1$ as $\epsilon \searrow 0$.

By definition of $\alpha^*$ we know that for all $\delta>0$, 
$$\int_0^{\infty} e^{-(\alpha^* - \delta) t}  \mu^*(dt)>1.$$ 
So, for all $\delta>0$, there exist $\epsilon_0= \epsilon_0(\delta)>0$ such that for all $\epsilon \in (0,\epsilon_0)$
$$\int_0^{\infty} e^{-(\alpha^* - \delta) t} \mu^{-}_{\epsilon}(dt) =
\int_0^{1/\epsilon} K_-(\epsilon) e^{-(\alpha^* - \delta) t} \mu^*(dt) \geq 1.
$$
We also know that for all $\epsilon>0$,
$\int_0^{1/\epsilon} e^{-x t}  \mu^*(dt)$ is continuous is $x$ and 
$\int_0^{1/\epsilon} e^{-x t}  \mu^*(dt)<1$ for all $x>\alpha^*$, which implies that for all $\epsilon>0$.
$\int_0^{1/\epsilon} e^{-\alpha^* t}  \mu^*(dt) \leq 1$. 
In particular,
$$ \int_0^{\infty} e^{-\alpha^*  t} \mu^{-}_{\epsilon}(dt)
= \int_0^{1/\epsilon} K_-(\epsilon) e^{-\alpha^*  t} \mu^*(dt)  
<\int_0^{1/\epsilon} e^{- \alpha^*  t} \mu^*(dt) \leq 1.$$
So, for all $\delta >0$ there exists $\epsilon_0= \epsilon_0(\delta)>0$ such that for all $\epsilon \in (0,\epsilon_0)$,  there exists $\alpha_{\epsilon}^-\in (\alpha^*,\alpha^*+\delta)$ such that 
$$ \int_0^{\infty} e^{-\alpha^-_{\epsilon}  t} \mu^{-}_{\epsilon}(dt) =1.$$
This concludes Step 1 of the proof.
\\ 

\noindent\textbf{Step 2:}

\noindent In this step we wish to show that there exists $\epsilon>0$ such that
\begin{enumerate}[label=(\roman*)]
\item there exists $\alpha^-_{\epsilon}<0$ such that $1 = \int_0^{\infty} e^{-\alpha^-_{\epsilon}t}  \mu^-_{\epsilon}(dt)$,
\item $\int_0^{\infty} t e^{|\alpha^-_{\epsilon}| t} L'_{\epsilon}(dt)<\infty$,
\item $\int_0^{\infty} t e^{|\alpha^-_{\epsilon}| t} \mu^-_{\epsilon}(dt)<\infty$ and
\item $\mathbb{E}\left[\int_0^{\infty}e^{|\alpha^-_{\epsilon}| t}\hat{\xi}_{\epsilon}^-(dt) \log^+(\xi(\infty))\right]< \infty$.
\end{enumerate}

We note first that since $\mu^-_{\epsilon}(\cdot)$ has mass on a bounded interval, (i) and (iii) are trivially satisfied.
Similarly, because $L'_{\epsilon}$ has bounded support (ii) is also satisfied. 

Finally, 
\begin{equation*}
\begin{split}
\ & \  \mathbb{E}\left[\int_0^{\infty}e^{|\alpha^-_{\epsilon}| t}\hat{\xi}_{\epsilon}^-(dt) \log^+(\hat{\xi}_{\epsilon}^-(\infty))\right]\\
\leq & \  \mathbb{E}\left[\int_0^{\infty}e^{|\alpha^-_{\epsilon}| t}\hat{\xi}_{\epsilon}^-(dt) \log^+(\tilde{D}^-(\epsilon)-1))\right]\\
= & \ \mathbb{E}[(\tilde{D}^{-}(\epsilon)-1) \log^+(\tilde{D}^-(\epsilon)-1)] \kappa_-(\epsilon)
\int_0^{\infty}e^{|\alpha^-_{\epsilon}| t} \beta e^{-\beta t} P(L'>t) dt\\
\leq & \ \mathbb{E}[(\tilde{D}^{-}(\epsilon)-1)^2] \kappa_-(\epsilon)
\int_0^{1/\epsilon}e^{|\alpha^-_{\epsilon}| t} \beta e^{-\beta t} P(L>t) dt\\
= & \ \frac{\mathbb{E}[(\tilde{D}^{-}(\epsilon)-1)^2]}{\mathbb{E}[\tilde{D}^{-}(\epsilon)-1]} 
\int_0^{1/\epsilon}e^{|\alpha^-_{\epsilon}| t} \mu^-_{\epsilon}(dt).
\end{split}
\end{equation*}
It follows from Claim \ref{dalem} that the quotient of expectations is  finite, while the integral is trivially finite. So assumption (iv) is met.\\

\noindent\textbf{Step 3: }

\noindent Let $\gamma \in (0,1-q^*)$ and $\gamma' \in (\gamma, 1-q^*)$. 
By the definition $$T'_{\gamma}(n) = \inf \{t>0; n^{-1}|S^{(n)}(t)|<1-\gamma\},$$ we obtain that if $T'_{\gamma'}(n)<\infty$, then $n^{-1}|S^{(n)}(t)|\geq 1-\gamma'$ for $t<T'_{\gamma'}(n)$ and in particular,
$n^{-1}|S^{(n)}(T'_{\gamma}(n))|\geq 1-\gamma'$.
 Combined with Lemma \ref{lemint1} this gives that $|S^{(n)}(T'_{\gamma}(n))|-|S^{(n)}(\infty)| = \theta(n)$ w.h.p.

For $t>0$, let $V_*^{(n)}(t)$ be as before the set of vertices infected after time $t$, which are infected by vertices infected before time $t$.  
Assume that  $|V_{*}^{(n)}(t)|= o(n)$.
From Step 1 we know that $|S^{(n)}(T'_{\gamma}(n))|-|S^{(n)}(\infty)|$ is stochastically smaller than the total progeny of $|V_{*}^{(n)}(t)|$ sub-critical branching processes with mean offspring measure $\mu^+_{\epsilon}$ and thus expected total number of children per particle $\mu^+_{\epsilon}(\infty)$. However the total size of such a branching process has expected size $(1-\mu^+_{\epsilon}(\infty))^{-1} = \theta(1)$. This implies that if $|V_{*}^{(n)}(T'_{\gamma}(n))|= o(n)$, then $\mathbb{E}[|S^{(n)}(T'_{\gamma}(n))|-|S^{(n)}(\infty)|] = o(n)$, which implies that $ |S^{(n)}(T'_{\gamma}(n)|-|S^{(n)}(\infty)|  = o(n)$ w.h.p.,
which is a contradiction. This finishes step 3.\\

\noindent \textbf{Step 4:}

\noindent Let $\delta \in (0,|\alpha^*|)$.  
We can and do choose $\epsilon>0$, such that $\alpha^-_{\epsilon}$ exists and $|\alpha^-_{\epsilon}| \in (|\alpha^*|,|\alpha^*| + \delta/2)$. By Lemma \ref{eeslemma2} we know that $\mathbb{P}(t_1^{(n)}(\epsilon)<T^*(n)|\mathcal{M}^{(n)}) \to 1$ and we choose $\gamma$ such that  $T'_{\gamma}(n) \in (t_1^{(n)}(\epsilon),T^*(n))$. 

Observe  that for  $t>T'_{\gamma}(n)$,  
\begin{multline}
|I^{(n)}(t)|= \sum_{v \in  V_{*}^{(n)}(T'_{\gamma}(n))} |J_v(t-\sigma(v))| + |I^{(n)}(t) \cap I^{(n)}(T'_{\gamma}(n))|\\
\geq \sum_{v \in  V_{*}^{(n)}(T'_{\gamma}(n))} |J_v(t-\sigma(v))|,
\end{multline}
where $J_v(s)=0$ for $s<0$.

Recall that  $\{J_v(t); t \geq 0\}$ dominates a branching process with mean offspring measure $\{\mu_{\epsilon}^-(t); t \geq 0\}$. Consider a sequence of i.i.d.\ copies of this branching process indexed by $v \in |V_{*}^{(n)}(T'_{\gamma}(n))|$ and let $Z^-_{\epsilon,v}(t)$ be the number of alive particles in the copy indexed by $v$ at time $t$.
So, $|I^{(n)}(t)|$ is stochastically larger than 
$\sum_{v \in V_{*}^{(n)}(T'_{\gamma}(n))} Z^-_{\epsilon,v}(t-\sigma(v))$.
By the independence of the branching processes we then  obtain that
\begin{equation}
\label{sumprob3}
\begin{split}
\ & \mathbb{P}(|I^{(n)}(T'_{\gamma}(n)+t )|=0\mid V_{*}^{(n)}(T'_{\gamma}(n)))\\
\leq & \displaystyle\prod_{v \in V_{*}^{(n)}(T'_{\gamma}(n))} \mathbb{P}(Z^-_{\epsilon,v}(T'_{\gamma}(n)+t-\sigma(v))=0\mid V_{*}^{(n)}(T'_{\gamma}(n)))\\
\leq & \displaystyle\prod_{v \in V_{*}^{(n)}(T'_{\gamma}(n))} \mathbb{P}(Z^-_{\epsilon,v}(t)=0)\\
= & \left( \mathbb{P}(Z^-_{\epsilon,1}(t)=0\right)^{|V_{*}^{(n)}(T'_{\gamma}(n))|}\\
= &  \mathbb{P}\left(Z^-_{\epsilon,1}(t)=0\mid Z^-_{\epsilon,v}(0)= |V_{*}^{(n)}(T'_{\gamma}(n))|\right).
\end{split}
\end{equation}
For the second inequality we used that $\{Z^-_{\epsilon,k}(t)=0\}$ is increasing in $t$.
Using the above gives us for all $c_1 \in (0,1)$
\begin{equation}\label{sumprob4}
\begin{split}
\ & \ \mathbb{P}(|I^{(n)}(T'_{\gamma}(n)+t )|=0\mid \mathcal{M}^{(n)})\\
= & \
\mathbb{P}(|I^{(n)}(T'_{\gamma}(n)+t )|=0 \cap |V_{*}^{(n)}(T'_{\gamma}(n))| > c_1 n \mid \mathcal{M}^{(n)})\\
\ & \ +\mathbb{P}(|I^{(n)}(T'_{\gamma}(n)+t )|=0 \cap |V_{*}^{(n)}(T'_{\gamma}(n))| \leq c_1 n \mid \mathcal{M}^{(n)})\\
\leq & \ \mathbb{P}\left(|I^{(n)}(T'_{\gamma}(n)+t )|=0\mid |V_{*}^{(n)}(T'_{\gamma}(n))| > c_1 n\right) \mathbb{P}(|V_{*}^{(n)}(T'_{\gamma}(n))| > c_1 n\mid \mathcal{M}^{(n)})\\ 
\ & \ + \mathbb{P}(|V_{*}^{(n)}(T'_{\gamma}(n))| \leq c_1 n\mid \mathcal{M}^{(n)})\\
\leq & \ \mathbb{P}\left(|I^{(n)}(T'_{\gamma}(n)+t )|=0\mid |V_{*}^{(n)}(T'_{\gamma}(n))| > c_1 n\right)\\
\ & \ + \mathbb{P}(|V_{*}^{(n)}(T'_{\gamma}(n))| \leq c_1 n\mid \mathcal{M}^{(n)})\\
\leq & \ \mathbb{P}\left(Z^-_{\epsilon,1}(t)=0\mid Z^-_{\epsilon,v}(0)= c_1 n\right) 
+ \mathbb{P}(|V_{*}^{(n)}(T'_{\gamma}(n))| \leq c_1 n\mid \mathcal{M}^{(n)}).
\end{split}
\end{equation}

Now we can apply Corollary \ref{bpsubcor}, which gives that for all $\delta>0$
$$
\mathbb{P}\left (Z^-_{\epsilon,1}\left(\frac{\log c_1 n}{|\alpha_{\epsilon}^-|+\delta/3}\right)=0\mid |Z^-_{\epsilon,1}(0)|= c_1 n\right) \to 0.
$$
For sufficiently large $n$ we have $\frac{\log c_1 n}{|\alpha_{\epsilon}^-|+\delta/3} \geq 
\frac{\log n}{|\alpha_{\epsilon}^-|+\delta/2}$ and thus we obtain
$$
\mathbb{P}\left (Z^-_{\epsilon,1}\left(\frac{\log n}{|\alpha_{\epsilon}^-|+\delta/2}\right)=0\mid |Z^-_{\epsilon,1}(0)|= c_1 n\right) \to 0.
$$
By step 3 we also know that there exists $c_1 \in (0,1)$ such that $$\mathbb{P}(|V_{*}^{(n)}(T'_{\gamma}(n))| \leq c_1 n\mid \mathcal{M}^{(n)}) \to 0.$$
So, from \eqref{sumprob4} we obtain
$$\mathbb{P}\left(|I^{(n)}\left(T'_{\gamma}(n)+  \frac{\log n}{|\alpha_{\epsilon}^-|+\delta/2}\right)|=0\mid \mathcal{M}^{(n)}\right) \to 0.$$
By $|\alpha^-_{\epsilon}| <|\alpha^*|+ \delta/2$ we then obtain 
$$
\mathbb{P}\left(|I^{(n)}\left(T'_{\gamma}(n)+\frac{\log n}{|\alpha^*|+ \delta} \right)|=0\mid \mathcal{M}^{(n)}\right) \to 0.
$$
This in turn leads to
\begin{multline*}
\mathbb{P}\left(T^*(n)-T'_{\gamma}(n)< \frac{\log n}{|\alpha^*|+ \delta}\mid \mathcal{M}^{(n)}\right) \\
=
\mathbb{P}\left(|I^{(n)}\left(T'_{\gamma}(n))+\frac{\log n}{|\alpha^*+\delta|}\right)|=0\mid \mathcal{M}^{(n)}\right) \to 0.
\end{multline*}
and the proof is complete.
\end{proof}

\subsection{Proof of Theorem \ref{mainthmsec}}\label{Secproofmainthmsec}

In this section we show how Theorem \ref{mainthmsec} relatively straightforward follows from the proof of Theorem \ref{mainthm}.

The way we prove it is to show that for every $\eta>0$ with high probability no vertex in the population is infectious while having at least one susceptible neighbor for a period at least $(|\alpha^*|-\eta)^{-1}\log n$. Furthermore, we show that for all vertices infected after time $T'_{\gamma}(n)$ (as defined in \eqref{Tprime}) condition \eqref{equivalence} is satisfied if we replace $L$ by $L'$ as defined in \eqref{Lprime}. Then we can use the proof of Lemma \ref{endlemup} with the replacement for $L$, while Lemma \ref{endlemdown} holds irrespective of the distribution of $L$.

\begin{proof}[Proof of Theorem \ref{mainthmsec}]
Observe that 
\begin{equation}
\begin{split}
\label{cutoffproof1}
\ & \ \mathbb{P}\left(\displaystyle\max_{v\in V^{(n)}} L'_v>t\right)\\
= & \ 1-\displaystyle\prod_{v\in V^{(n)}} \left[1-\mathbb{P}(L_v>t)\left(1-(1-e^{-\beta t})^{d_v}\right) \right]\\
\leq & \  \displaystyle\sum_{v\in V^{(n)}} \mathbb{P}(L>t)\left(1-(1-e^{-\beta t})^{d_v}\right)\\
\leq & \ \displaystyle\sum_{v\in V^{(n)}} \mathbb{P}(L>t) d_v e^{-\beta t}\\
= & \ \mathbb{P}(L>t) \ell(n) e^{-\beta t}. 
\end{split}
\end{equation}
We also know from the definition of $\alpha^*$ (see \eqref{endmalt1} and \eqref{endmalt2}) that
\begin{equation}\label{cutfinite}
\int_0^{\infty} e^{(|\alpha^*|-\eta)t} e^{-\beta t} \mathbb{P}(L>t) dt < \infty \qquad \mbox{for every $\eta\in (0,\alpha^*)$}
\end{equation}
and thus that 
\begin{equation*}
e^{(|\alpha^*|-\eta)t} e^{-\beta t} \mathbb{P}(L>t) \to 0 \qquad \mbox{as $t \to \infty$.}
\end{equation*}
which, after filling in $t= (|\alpha^*|-\eta)^{-1} \log n$ implies
\begin{equation}\label{cutofflim}
n^{1-\frac{\beta}{|\alpha^*|-\eta}} \mathbb{P}\left(L>\frac{\log n}{|\alpha^*|-\eta}\right) \to 0 \qquad \mbox{as $n \to \infty$.}
\end{equation}
Filling in $t= (|\alpha^*|-\eta)^{-1} \log n$ in \eqref{cutoffproof1} we obtain that
$$\mathbb{P}\left(\displaystyle\max_{v\in V^{(n)}} L'_v>\frac{\log n}{|\alpha^*|-\eta}\right) \leq 
\mathbb{P}\left(L>\frac{\log n}{|\alpha^*|-\eta}\right) \ell(n) n^{-\frac{\beta}{|\alpha^*|-\eta}},
$$
which by \eqref{cutofflim} and $\ell(n) = O(n)$ implies that for all $\eta\in (0,\alpha^*)$, 
$$\mathbb{P}\left(\displaystyle\max_{v\in V^{(n)}} L'_v>\frac{\log n}{|\alpha^*|-\eta}\right) \to 0.$$
So with high probability no vertex infected before time $T'_{\gamma}(n)$ is both still infectious and has susceptible neighbors at time $T'_{\gamma}(n)+ \frac{\log n}{|\alpha^*|-\eta}$.

Our next step is to observe that the epidemic spread does not change if for all $v \in V^{(n)}$ we say that $v$ recovers $L'_v$ instead of $L_v$ time units after the infection time $\sigma(v)$.

in the first step of the proof of Lemma \ref{endlemup} we can then replace $L$ by a random variable $L'$, with a distribution defined through
\begin{multline*}
\mathbb{P}(L' >t)  =  \mathbb{P}(L>t) \mathbb{E}\left[1-(1-e^{-\beta t})^{\tilde{D}^{(n),+}(\epsilon) }\right]\\
 \leq \mathbb{P}(L>t) \mathbb{E}\left[\tilde{D}^{(n),+}(\epsilon)\right]e^{-\beta t}.
\end{multline*}
We further use that by Claim \ref{dalem}  $\mathbb{E}\left[\tilde{D}^{(n),+}(\epsilon)\right] < \infty$.

We may apply Lemma \ref{endlemup} with $L$ replaced by $L'$ and check whether condition \eqref{equivalence} holds:
\begin{equation}
\int_0^{\infty} e^{(|\alpha^*|-\eta)t} \mathbb{P}(L'>t) dt \leq
\int_0^{\infty} e^{(|\alpha^*|-\eta)t} \mathbb{P}(L>t)  \mathbb{E}\left[\tilde{D}^{(n),+}(\epsilon)\right]e^{-\beta t} dt,
\end{equation}
which is indeed finite by \eqref{cutfinite}.

\end{proof}

\appendix
\section{Proof of Lemma \ref{eeslemma2}}\label{Appendixsec}

In this appendix we prove Lemma \ref{eeslemma2}. We repeat some definitions and the statement of the Lemma.

Recall from \eqref{Qdef} that $Q= 1-\psi+\psi \tilde{q}^*$ and let
$\{\mathcal{X}^{(n)}(t);t \geq 0\}$
be defined as in Section \ref{sec:construct}.
Furthermore, for all $n \in \mathbb{N}$ and all $k \in \mathbb{N}$  we define 
$$\hat{d}^{(n)}_k = \sup_{n' \geq n}\mathbb{P}(D^{(n')} = k).$$
For all $\epsilon \in (0,\psi (1-\tilde{q}^*))= (0,1-Q)$ we define $\mathcal{A}^{(n)}_1(\epsilon) = \{ t_1^{(n)}(\epsilon)< \infty\}$, where
 $$t_1^{(n)}(\epsilon) = \max(t_a^{(n)}(\epsilon) ,t_b^{(n)}(\epsilon) ,t_c^{(n)}(\epsilon)),$$
and
\begin{equation*}
\begin{split}
t_a^{(n)}(\epsilon) & = \inf\{t>0;|\mathcal{E}_S^{(n)}(t)| \leq \mathbb{E}[(Q + \epsilon)^{\tilde{D}}]\ell(n)\}, \\
t_b^{(n)}(\epsilon)& = \inf\{t>0;|\mathcal{E}_P^{(n)}(t)| \geq  \ell(n) -1 - (Q + \epsilon)^2 \ell(n)\},\\
t_c^{(n)}(\epsilon)& = \inf\{t>0; \sum_{v\in S^{(n)}(t)} d_v \ind(d_v \geq k)  \leq
\sum_{j=k}^{\infty} n j \hat{d}^{(n)}_j (Q + \epsilon)^j
\mbox{ for all $k \in \mathbb{N}$}\}.
\end{split}
\end{equation*}
We also 
defined the event $\mathcal{A}^{(n)}_2(\epsilon)$, which is the event that the following holds.
\begin{eqnarray*}
|\mathcal{E}_S^{(n)}(\infty)| & > & \mathbb{E}[(Q - \epsilon)^{\tilde{D}}] \ell(n), \\
|\mathcal{E}_P^{(n)}(\infty)| & < &  \ell(n) -1 - (Q - \epsilon)^2 \ell(n), \\
\sum_{v\in S^{(n)}(\infty) } d_v \ind(d_v \geq k) & \geq & \sum_{j=k}^{\infty} n j \mathbb{P}(D= j) 
(Q - \epsilon)^{j} \mbox{ for all $k \in \mathbb{N}_{\leq \lfloor 1/\epsilon \rfloor}$}.
\end{eqnarray*}
Finally, $\mathcal{A}^{(n)}(\epsilon)=\mathcal{A}^{(n)}_1(\epsilon)\cap \mathcal{A}^{(n)}_2(\epsilon)$.

\begin{lemma}
For all $\epsilon \in (0,\psi (1-\tilde{q}^*))$, it holds that $\mathbb{P}(\mathcal{A}^{(n)}(\epsilon)|\mathcal{M}^{(n)}) \to 1$ and there exists $c_1>0$, such that 
$$\mathbb{P}(|S^{(n)}(t_1^{(n)}(\epsilon))|-|S^{(n)}(\infty)| > c_1 n|\mathcal{M}^{(n)}) \to 1.$$ 
\end{lemma}

\begin{proof}
We start with some definitions.
Let $K_1 = K_1^{(n)}(\epsilon)$ be a Poisson distributed random variable with expectation 
$\ell(n) |\log(Q + \epsilon/2)|$
and let $K_2 = K_2^{(n)}(\epsilon)$ be a Poisson distributed random variable with expectation 
$\ell(n) |\log(Q - \epsilon/2)|$. Both $K_1$ and $K_2$ are independent of the epidemic process.
Let $\mathbf{x}^{(n)}$ and $\mathbf{x}^{(n)}(t)$ be  as in Section \ref{sec:construct} and define for $i \in \{1,2\}$ 
$$t'(K_i)=\inf\{t>0;|\mathbf{x}^{(n)}(t)|\geq K_i\}.$$ 
For notational convenience define $\bar{S}^{(n)}_i$ for $i \in \{1,2\}$  as the set of vertices which have no half-edge among the first $K_i$ elements of $\mathbf{x}^{(n)}$. So  $\bar{S}^{(n)}_1$ is equal to $S^{(n)}(t'(K_1))$ on $\mathcal{M}^{(n)}$.  

Our strategy is now as follows.
First we show that
\begin{equation}
\label{K1finiteA}
\mathbb{P}(K_1<|\mathbf{x}^{(n)}(\infty)|<K_2|\mathcal{M}^{(n)})  \to 1
\end{equation}
and that there exists $c_1>0$, such that 
\begin{equation}
\label{eeslemma1}
\mathbb{P}\left(|S^{(n)}(t'(K_1))|-|S^{(n)}(\infty)|>c_1 n|\mathcal{M}^{(n)}\right) \to 1.
\end{equation}
After that we show that for $i \in \{a,b,c\}$ 
\begin{equation}
\label{t1bounds}
\mathbb{P}\left(\left|\mathbf{x}^{(n)}\left(t_i^{(n)}(\epsilon)\right)\right|<K_1 \mid \mathcal{M}^{(n)}\right)  \to 1
\end{equation}
and that\\ (i) the number of half edges that belong to vertices in $V^{(n)}$ that have none of their half-edges among the first $K_2$ elements of $\mathbf{x}^{(n)}$ exceeds with high probability
$\mathbb{E}[(Q - \epsilon)^{\tilde{D}}] \ell(n)$,\\
(ii) the number of half-edges that are themselves or are paired with half edges among the first 
$K_2$ elements of $\mathbf{x}^{(n)}$ is with high probability less than $ \ell(n) -1 - (Q - \epsilon)^2 \ell(n)$,\\
(iii) For every $k \in \mathbb{N}$, the number of half edges that belong to vertices of degree at least $k$ in $V^{(n)}$ that have none of their half-edges among the first $K_2$ elements of $\mathbf{x}^{(n)}$ is with high probability at least
$\sum_{j=k}^{\infty} n j \mathbb{P}(D^{(n)} = j) (Q - \epsilon)^{j}$.
Together this proves the Lemma.

Because the elements of $\mathbf{x}^{(n)}$ are i.i.d.\ and uniform among all $\ell(n)$ half-edges, we have by well-known properties of the Poisson distribution (see e.g.\ \cite[p.~317]{Resn13}) that the number of times a given half-edge is among the first $K_1$ (resp.\ $K_2$) elements of $\mathbf{x}^{(n)}$ is Poisson distributed with expectation $|\log(Q + \epsilon/2)|$ (resp.\ Poisson distributed with expectation $|\log(Q - \epsilon/2)|$) and independent for different half-edges. This implies that the events that different half-edges are not  among the first $K_1$  elements of $\mathbf{x}^{(n)}$ are independent and have probability $$e^{-|\log(Q + \epsilon/2)|}= Q + \epsilon/2.$$ Similarly,  
the events that different half-edges are not among the first $K_2$  elements of $\mathbf{x}^{(n)}$ are independent and have probability $Q - \epsilon/2$.
So, the probability that none of the half-edges belonging to a uniformly chosen vertex is part of the first $K_1$ elements of $\mathbf{x}$  is given by
$$
\sum_{k=0}^{\infty} \mathbb{P}(D^{(n)}=k)  (Q + \epsilon/2)^k \to \sum_{k=0}^{\infty} \mathbb{P}(D=k) (Q + \epsilon/2)^k,
$$ 
and the probability that none of the half-edges belonging to a uniformly chosen vertex is part of the first $K_2$ elements of $\mathbf{x}$  is given by
$$
\sum_{k=0}^{\infty} \mathbb{P}(D^{(n)}=k)  (Q - \epsilon/2)^k \to \sum_{k=0}^{\infty} \mathbb{P}(D=k) (Q - \epsilon/2)^k,
$$ 
where the limits follows from  $D^{(n)} \darrow D$ and bounded convergence.

So, by a variant of the (weak) law of large numbers (e.g.\ \cite[Problem 7.11.20]{Grim01}), 
we obtain that the fraction of the vertices with no half-edges among the first $K_1$ half-edges of  $\mathbf{x}^{(n)}$ (i.e.\ $n^{-1}|\bar{S}^{(n)}_1|$)  converges in probability to 
$\sum_{k=0}^{\infty} \mathbb{P}(D=k) (Q + \epsilon/2)^k$
and 
that the fraction of the vertices with no half-edges among the first $K_2$ half-edges of  $\mathbf{x}^{(n)}$ (i.e.\ $n^{-1}|\bar{S}^{(n)}_2|$)  converges in probability to 
$\sum_{k=0}^{\infty} \mathbb{P}(D=k) (Q - \epsilon/2)^k$.

By Lemma \ref{lemint1}, equation (\ref{qstar}) and by
$q^*= \sum_{k=0}^{\infty} \mathbb{P}(D=k) Q^k $
we have 
\begin{equation}
\label{eqlemint}
\left(\frac{1}{n}|S^{(n)}(\infty)|-\sum_{k=0}^{\infty} \mathbb{P}(D=k) Q^k\right)\ind(\mathcal{M}^{(n)}) \parrow 0 .
\end{equation}
Because $\sum_{k=0}^{\infty} \mathbb{P}(D^{(n)}=k) x^k$ is strictly increasing for $x \in [0,1)$, 
the above implies immediately that \eqref{K1finiteA} and \eqref{eeslemma1} hold.

In a similar fashion we obtain that the probability that a uniformly chosen half-edge belongs to a vertex of which none of the half-edges  is part of the first $K_1$ elements of $\mathbf{x}^{(n)}$  is given by
$$\sum_{k=0}^{\infty} \mathbb{P}(\tilde{D}^{(n)}=k) (Q + \epsilon/2)^k$$
and the probability that a uniformly chosen half-edge belongs to a vertex of which none of the half-edges  is part of the first $K_2$ elements of $\mathbf{x}^{(n)}$  is given by
$$\sum_{k=0}^{\infty} \mathbb{P}(\tilde{D}^{(n)}=k) (Q - \epsilon/2)^k.$$
Using the same law of large numbers argument as above we obtain that the fraction of half-edges belonging to vertices of which none of the half-edges  is part of the first $K_1$ elements of $\mathbf{x}^{(n)}$  converges in probability to $\sum_{k=0}^{\infty} \mathbb{P}(\tilde{D}=k) (Q + \epsilon/2)^k$,
which is strictly less than $\mathbb{E}[(Q + \epsilon)^{\tilde{D}}]$.
Similarly, 
the fraction of half-edges belonging to vertices of which none of the half-edges  is part of the first $K_2$ elements of $\mathbf{x}^{(n)}$  converges in probability to $\sum_{k=0}^{\infty} \mathbb{P}(\tilde{D}=k) (Q - \epsilon/2)^k$,
which is strictly more than $\mathbb{E}[(Q - \epsilon)^{\tilde{D}}]$.
Together with \eqref{K1finiteA} this implies  
$$\mathbb{P}(t_a^{(n)}(\epsilon)<\infty|\mathcal{M}^{(n)}) \to 1$$
and 
$$\mathbb{P}\left(\frac{1}{\ell(n)} |\mathcal{E}^{(n)}_S(\infty)| > \mathbb{E}[(Q - \epsilon)^{\tilde{D}}]\mid \mathcal{M}^{(n)}\right) \to 1. $$

Now we turn our attention to $|\mathcal{E}^{(n)}_P(t)|$.
In $G^{(n)}$, all half-edges are paired uniformly at random. For a half-edge not to be part of $\mathcal{E}^{(n)}_P(t)$, neither the half-edge itself nor its partner should be part of $\mathbf{x}^{(n)}(t)$. 
Again the probability that two given half-edges are not among the first $K_1$ (resp.\ $K_2$) elements of $\mathbf{x}^{(n)}$ is $(Q +\epsilon/2)^2$ (resp.\  $(Q -\epsilon/2)^2$).
So using \eqref{K1finiteA} and the above law of large numbers again we obtain
$$
\left(\frac{\ell(n)-|\mathcal{E}^{(n)}_P(t'(K_1))|}{\ell(n)}- (Q +\epsilon/2)^2\right)\ind(t'(K_1)<\infty) \parrow 0.
$$
By 
$$
\frac{\ell(n)-|\mathcal{E}^{(n)}_P(t'(K_1))|-1}{\ell(n)} < \frac{\ell(n)-|\mathcal{E}^{(n)}_P(t'(K_1))|}{\ell(n)}$$
and \eqref{K1finiteA} we then obtain that $\mathbb{P}(t'(K_1)>t^{(n)}_b(\epsilon)|\mathcal{M}^{(n)}) \to 1$.
Furthermore, again by \eqref{K1finiteA} we also obtain that
$$
\mathbb{P}\left(\frac{\ell(n)-|\mathcal{E}^{(n)}_P(\infty)|-1}{\ell(n)}> (Q -\epsilon)^2\mid\mathcal{M}^{(n)}\right) \parrow 1.
$$

Finally, we consider $\sum_{v\in \bar{S}_1^{(n)}} d_v\ind(d_v \geq k)$.
By the definition of $\bar{S}_1^{(n)}$ and $D^{(n)}$ this sum is equal to $\sum_{j=k}^{\infty} j B^{(n)}_{j}$,  where 
$B^{(n)}_{j}$ ($j \in \mathbb{N}$) are independent binomially distributed random variable with parameters $n \mathbb{P}(D^{(n)}=j)$ and $(Q+\epsilon/2)^{j}$.

We want to show that $\mathbb{P}(t'(K_1)>t_c^{(n)}(\epsilon)|\mathcal{M}^{(n)})\to 1$. That is, we want 
\begin{equation}
\label{Dhatfirst}
\mathbb{P}\left(\sum_{j=k}^{\infty} j B^{(n)}_{j} \leq  n \sum_{j=k}^{\infty} j\hat{d}^{(n)}_j (Q+\epsilon)^j \mbox{ for all $k \in \mathbb{N}$} \right) \to 1.
\end{equation}
Since $\mathbb{P}(D^{(n)}=j)$ is  at most $\hat{d}^{(n)}_j$ for all $j \in \mathbb{N}$, it is enough to prove that 
\begin{equation}\label{Chebyuse}
\mathbb{P}\left(\sum_{j=k}^{\infty} j B^{(n)}_{j} \leq  n \sum_{j=k}^{\infty} j\mathbb{P}(D^{(n)}=j) (Q+\epsilon)^j \mbox{ for all $k \in \mathbb{N}$} \right) \to 1.
\end{equation}
if for given $k$ we have $\mathbb{P}(D\geq k) =0$, then by assumption (A4) of Assumptions \ref{mainassump} for all large enough $n$
$$\sum_{j=k}^{\infty} j B^{(n)}_{j}= n \sum_{j=k}^{\infty} j\mathbb{P}(D^{(n)}=j) (Q+\epsilon)^j=0.$$ So assume  
$\mathbb{P}(D\geq k) > 0$.
\begin{eqnarray*}
\mathbb{E}\left[\sum_{j=k}^{\infty} j B^{(n)}_{j}\right] & = & n \sum_{j=k}^{\infty} \mathbb{P}(D^{(n)}=j) j (Q+\epsilon/2)^j\\
Var\left[\sum_{j=k}^{\infty} j B^{(n)}_{j}\right]  & = & n \sum_{j=k}^{\infty} \mathbb{P}(D^{(n)}=j)j^2 (Q+\epsilon/2)^j  (1-(Q+\epsilon/2)^j)\\
\ & \leq &  n \sum_{j=k}^{\infty} \mathbb{P}(D^{(n)}=j)j^2 (Q+\epsilon/2)^j.
\end{eqnarray*}
So applying Chebyshev's inequality,
we obtain that 
\begin{equation}\label{Chebmult}
\begin{split}
\ &  n\mathbb{P}\left(\sum_{j=k}^{\infty} j B^{(n)}_{j} >  n \sum_{j=k}^{\infty} j\mathbb{P}(D^{(n)}=j) (Q+\epsilon)^j\right)\\
\leq & \frac{ \sum_{j=k}^{\infty} \mathbb{P}(D^{(n)}=j)j^2 (Q+\epsilon/2)^j}{ \left(\sum_{j=k}^{\infty} j\mathbb{P}(D^{(n)}=j) [(Q+\epsilon)^j-(Q+\epsilon/2)^j]\right)^2}\\
\to & \frac{ \mathbb{E}[\ind(D\geq k)D^2 (Q+\epsilon/2)^{D}]}{\left(\mathbb{E}\left[ \ind(D\geq k) D [(Q+\epsilon)^{D}-(Q+\epsilon/2)^{D}]\right]\right)^2},
\end{split}
\end{equation}
where we have used $D^{(n)} \to D$ and that $\sum_{j=k}^{\infty} j^2 (Q+\epsilon)^j \to 0$ as $k \to \infty$.
The quotient in \eqref{Chebmult} is independent of $n$ and by the assumption  
$\mathbb{P}(D\geq k) > 0$ it  is also finite. So,
$$\mathbb{P}\left(\sum_{j=k}^{\infty} j B^{(n)}_{j} >  n \sum_{j=k}^{\infty} j\mathbb{P}(D^{(n)}=j) (Q+\epsilon)^j\right) \to 0 \qquad \mbox{for all $k \in \mathbb{N}$.} $$
In other words, for all $k_0 \in \mathbb{N}$ and all $\epsilon_1>0$ there exists $n_0 \in \mathbb{N}$ such that for all $n >n_0$
$$\sum_{k=1}^{k_0} \mathbb{P}\left(\sum_{j=k}^{\infty} j B^{(n)}_{j} >  n \sum_{j=k}^{\infty} j\mathbb{P}(D^{(n)}=j) (Q+\epsilon)^j\right) < \epsilon_1/2.$$

So, in order to prove that 
\begin{equation*}
\mathbb{P}\left(\sum_{j=k}^{\infty} j B^{(n)}_{j} \leq  n \sum_{j=k}^{\infty} j\mathbb{P}(D^{(n)}=j) (Q+\epsilon)^j \mbox{ for all $k \in \mathbb{N}$} \right) \to 1,
\end{equation*}
it is enough to show that for every $\epsilon_1>0$, there exists  $k_0 \in \mathbb{N}$ such that
\begin{equation*}
\mathbb{P}\left(\sum_{j=k}^{\infty} j B^{(n)}_{j} \leq  n \sum_{j=k}^{\infty} j\mathbb{P}(D^{(n)}=j) (Q+\epsilon)^j \mbox{ for all $k \in \mathbb{N}_{> k_0}$} \right) > 1-\epsilon_1/2.
\end{equation*}
which holds if for every $\epsilon_1>0$, there exists  $k_0 \in \mathbb{N}$ such that
\begin{equation*}
\mathbb{P}\left( B^{(n)}_{k} \leq  n \mathbb{P}(D^{(n)}=k) (Q+\epsilon)^k \mbox{ for all $k \in \mathbb{N}_{> k_0}$} \right) > 1-\epsilon_1/2.
\end{equation*}
Observe that 
\begin{multline*}
\mathbb{P}\left( B^{(n)}_{k} \leq  n \mathbb{P}(D^{(n)}=k) (Q+\epsilon)^k \mbox{ for all $k \in \mathbb{N}_{\geq k_0}$} \right)\\
\geq   1- \sum_{k=k_0}^{\infty} \mathbb{P}\left(B^{(n)}_{k} >  n \mathbb{P}(D^{(n)}=k) (Q+\epsilon)^k\right). 
\end{multline*}

Define 
$$\mathcal{K}^{(n)}_0(\epsilon)= \left\{k\in \mathbb{N}_{>k_0}; n \mathbb{P}(D^{(n)}=k) (Q+3\epsilon/4)^k \leq 1\right\}.$$
We should prove that for every $\epsilon_1>0$, there exists  $k_0 \in \mathbb{N}$ such that
\begin{equation*}
\begin{split}
\ & \sum_{k=k_0+1}^{\infty} \mathbb{P}\left(B^{(n)}_{k} >  n \mathbb{P}(D^{(n)}=k) (Q+\epsilon)^k\right)\\
= & \sum_{k\in \mathcal{K}^{(n)}_0(\epsilon)} \mathbb{P}\left(B^{(n)}_{k} >  n \mathbb{P}(D^{(n)}=k) (Q+\epsilon)^k\right)\\
\ & + \sum_{k\in \mathbb{N}_{>k_0} \setminus \mathcal{K}^{(n)}_0(\epsilon)} \mathbb{P}\left(B^{(n)}_{k} >  n \mathbb{P}(D^{(n)}=k) (Q+\epsilon)^k\right)  <\epsilon_1/2.
\end{split}
\end{equation*}
To do this, we consider the two sums in the middle term separately.

We have 
\begin{equation}
\begin{array}{ll}
\ & \displaystyle\sum_{k\in \mathcal{K}^{(n)}_0(\epsilon)} \mathbb{P}\left( B^{(n)}_{k} >  n \mathbb{P}(D^{(n)}=k) (Q+\epsilon)^k\right)  \\
\leq & \displaystyle\sum_{k\in \mathcal{K}^{(n)}_0(\epsilon)} \mathbb{P}\left( B^{(n)}_{k} >0\right) \\
= & \displaystyle\sum_{k\in \mathcal{K}^{(n)}_0(\epsilon)} 
1-\left(1- (Q+\epsilon/2)^k\right)^{n \mathbb{P}(D^{(n)}=k)}\\
 \leq & \displaystyle\sum_{k\in \mathcal{K}^{(n)}_0(\epsilon)} n \mathbb{P}(D^{(n)}=k) (Q+3\epsilon/4)^k \left(\frac{Q+\epsilon/2}{Q+3\epsilon/4}\right)^k\\
 \leq & \displaystyle\sum_{k\in \mathcal{K}^{(n)}_0(\epsilon)}  \left(\frac{Q+\epsilon/2}{Q+3\epsilon/4}\right)^k\\
 \leq & \displaystyle\sum_{k = k_0+1}^{\infty}  \left(\frac{Q+\epsilon/2}{Q+3\epsilon/4}\right)^k\\
 = & \left(\frac{Q+\epsilon/2}{Q+3\epsilon/4}\right)^{k_0+1}\frac{4(Q+3\epsilon/4)}{\epsilon}, 
\end{array}
\end{equation}
which is less than $\epsilon_1/4$ for all large enough $k_0$.

Now consider 
$$\sum_{k\in \mathbb{N}_{>k_0} \setminus \mathcal{K}^{(n)}_0(\epsilon)}\mathbb{P}\left( B^{(n)}_{k} >  n \mathbb{P}(D^{(n)}=k) (Q+\epsilon)^k\right).$$
Let $k_0'$ be such that $\frac{(Q+\epsilon)^{k_0'}}{(Q+\epsilon/2)^{k_0'}}>7$. Further assume that $k_0$ was chosen such that $k_0>k_0'$.
Then by  \cite[Cor.\ 2.4]{Jans11} we obtain that for $k\in \mathbb{N}_{>k_0} \setminus \mathcal{K}^{(n)}_0(\epsilon)$
\begin{eqnarray*}
\ & \ &  \mathbb{P}\left( B^{(n)}_{k} >  n \mathbb{P}(D^{(n)}=k) (Q+\epsilon)^k\right)\\
\ & = &  \mathbb{P}\left( B^{(n)}_{k} >  n \mathbb{P}(D^{(n)}=k) (Q+3\epsilon/4)^k \left(\frac{Q+\epsilon}{Q+3\epsilon/4}\right)^k\right)\\
\ & \leq & \mathbb{P}\left( B^{(n)}_{k} >  \left(\frac{Q+\epsilon}{Q+3\epsilon/4}\right)^k\right)\\
\ & \leq & \exp\left[-\left(1+\frac{\epsilon}{4Q+3\epsilon}\right)^k\right]\\
\ & \leq & \exp\left[-\left(1+k\left(\frac{\epsilon}{4Q+3\epsilon}\right)\right)\right].
\end{eqnarray*}
So,
\begin{multline*}
\sum_{k\in \mathbb{N}_{>k_0} \setminus \mathcal{K}^{(n)}_0(\epsilon)} \mathbb{P}\left(B^{(n)}_{k} >  n \mathbb{P}(D^{(n)}=k) (Q+\epsilon)^k\right)\\
\leq \sum_{k\in \mathbb{N}_{>k_0} \setminus \mathcal{K}^{(n)}_0(\epsilon)} e^{-1} \exp\left[-k\left(\frac{\epsilon}{4Q+3\epsilon}\right)\right]
\leq e^{-1}\sum_{k=k_0+1}^{\infty} \left(e^{-\frac{\epsilon}{4Q+3\epsilon}}\right)^k\\
= e^{-1} \left(e^{-\frac{\epsilon}{4Q+3\epsilon}}\right)^{k_0+1}\frac{1}{1-e^{-\frac{\epsilon}{4Q+3\epsilon}}},
\end{multline*}
which is less than $\epsilon_1/4$ for all large enough $k_0$.

So we have proved that for every $\epsilon_1>0$, there exists  $k_0 \in \mathbb{N}$ such that
$$
\sum_{k=k_0}^{\infty} \mathbb{P}\left(B^{(n)}_{k} >  n \mathbb{P}(D^{(n)}=k) (Q+\epsilon)^k\right) <\epsilon_1/2
$$
as desired. Together with \eqref{K1finiteA} this implies that  $\mathbb{P}(t^{(n)}_c(\epsilon)<t'(K_1)|\mathcal{M}^{(n)})$ converges to 1 and thus that $\mathbb{P}(t^{(n)}_c(\epsilon)<\infty|\mathcal{M}^{(n)})$ converges to 1.

To prove that  
$$\mathbb{P}\left(\sum_{v\in S^{(n)}(\infty) }  d_v\ind(d_v \geq k) \geq   
\sum_{j=k}^{\infty} n j \mathbb{P}(D=j) 
(Q - \epsilon)^{j} \mbox{ for all $k \in \mathbb{N}_{\leq \lfloor1/\epsilon\rfloor}$}\right)$$
converges to 1, it is by \eqref{K1finiteA} enough to prove that
\begin{equation}
\label{lasteq}
\mathbb{P}\left(\sum_{v\in S_2^{(n)}}  d_v\ind(d_v \geq k) \geq   
\sum_{j=k}^{\infty} n j \mathbb{P}(D=j) 
(Q - \epsilon)^{j} \mbox{ for all $k \in \mathbb{N}_{\leq \lfloor1/\epsilon\rfloor}$}\right).
\end{equation}
converges to 1.
Using the same law of large numbers argument used throughout this appendix we obtain that 
$n^{-1}\sum_{v\in S_2^{(n)}}  d_v\ind(d_v \geq k)$ converges in probability to 
$\sum_{j=k}^{\infty} j \mathbb{P}(D=j) 
(Q - \epsilon/2)^{j}$, which immediately implies \eqref{lasteq}.
\end{proof}

\providecommand{\noopsort}[1]{}
\providecommand{\bysame}{\leavevmode\hbox to3em{\hrulefill}\thinspace}
\providecommand{\MR}{\relax\ifhmode\unskip\space\fi MR }
\providecommand{\MRhref}[2]{%
  \href{http://www.ams.org/mathscinet-getitem?mr=#1}{#2}
}
\providecommand{\href}[2]{#2}

\end{document}